\documentclass[a4paper,10pt]{article}
\usepackage{epsfig,amssymb,amsmath}
\usepackage{amsmath}
\usepackage{amssymb}
\usepackage{amsthm}
\usepackage{indentfirst}
\usepackage{xspace}
\usepackage{multirow}
\usepackage{verbatim}
\usepackage{authblk}
\usepackage{tikz-cd}
\usepackage{capt-of}
\usepackage{silence}
\usepackage{hyperref}
\usepackage{MnSymbol}

\usepackage{graphicx} 
\usepackage{epstopdf}
\usepackage{enumerate}
\usepackage{url}
\usepackage{bm}
\usepackage{epsfig}
\usepackage{color}
\usepackage[labelfont={rm}]{subcaption}

\usepackage{enumerate}
\usepackage[missing=]{gitinfo2}
\usepackage[mathscr]{euscript}

\numberwithin{equation}{section}

\usepackage[margin=1.2in]{geometry}

\newtheorem{theorem}{Theorem}[section]
\newtheorem{lemma}[theorem]{Lemma}
\newtheorem{proposition}[theorem]{Proposition}
\newtheorem{corollary}[theorem]{Corollary}

\newtheorem{definition}[theorem]{Definition}

\newtheorem{theorem*}{Theorem}
\newtheorem{lemma*}{Lemma}
\newtheorem{proposition*}
{Proposition}
\newtheorem{corollary*}{Corollary}
\newtheorem{conjecture*}{Conjecture}

\newenvironment{remark}[1][Remark.]{\begin{trivlist}
\item[\hskip \labelsep {\bfseries #1}]}{\end{trivlist}}

\title{\textbf{Spectral networks for polynomial cubic differentials}}
\author{Omar Kidwai\thanks{kidwai@math.cuhk.edu.hk}}
\affil{Department of Mathematics, The Chinese University of Hong Kong \newline Shatin, N.T., Hong Kong S.A.R., China}
\author{Guillaume Tahar\thanks{guillaume.tahar@bimsa.cn}}
\affil{Beijing Institute of Mathematical Sciences and Applications\newline
No. 544, Hefangkou Village,
Huaibei Town, Huairou District, Beijing, China}
\date{}

\begin{document}

\maketitle

\vspace{-8mm}

\begin{abstract}
We study cubic differentials and their spectral networks on Riemann surfaces, focusing on the polynomial case on the Riemann sphere. We introduce the notion of {spectral core} as the primary tool for our study, refining the classical notion of core in the theory of flat surfaces, and show that it controls the birthing process of spectral network trajectories. As an application, we completely characterize the polynomial cubic differentials having saddle connections or critical tripods when the degree $d$ is at most $3$; in particular, we obtain the relevant degenerations as the phase is varied and determine explicitly the wall-and-chamber structure. In this case, we obtain the BPS structure according to Gaiotto-Moore-Neitzke's algorithm, and verify that it satisfies the Kontsevich-Soibelman wall-crossing formula. In physics language, this corresponds to computing the BPS spectrum of a certain four-dimensional $\mathcal{N}=2$ quantum field theory, known as the $(A_{2},A_{d-1})$ generalized Argyres-Douglas theory.
\end{abstract}

\tableofcontents

\section{Introduction}

This paper studies the behaviour of trajectories of cubic differentials $\varphi$ on a Riemann surface $X$ in the case where $X=\mathbf{P}^1$ and $\varphi$ has polynomial representation in some coordinate $x$, i.e. 
 \begin{equation*}
 \varphi =    P(x) dx^{\otimes 3}.
 \end{equation*}
with $P$ a degree $d$ polynomial. Such a problem is of interest for many reasons, as trajectories of cubic differentials appear in many places in mathematics, often (but not only) in the guise of a \emph{spectral network}. A spectral network is a certain graph-like object drawn on the surface, introduced by physicists Gaiotto-Moore-Neitzke in the study of four-dimensional $\mathcal{N}=2$ supersymmetric quantum field theory \cite{Gaiotto:2012rg}, with similar objects having appeared earlier under the name of \emph{Stokes graphs}, \emph{new Stokes lines} and \emph{virtual turning points} \cite{hondabook}. These appear in the theory WKB analysis \cite{hondabook,Berk1982NewSL,GAIOTTO2013239}, harmonic maps \cite{Neitzke17,dumasneitzke}, flat connections \cite{Hollands:2013qza,nikolaevabelianization,hollandskidwai} and ought to be relevant to the theory of Bridgeland stability conditions \cite{bridgeland2015quadratic, wu,kontsevich2008stability}. In physics, the setting of polynomial cubic differentials corresponds to the four-dimensional supersymmetric quantum field theory ``of generalized Argyres Douglas type $(A_2,A_{d-1})$'' \cite{Gaiotto:2012rg,Neitzke17}. Several conjectures in these subjects have recently been formulated precisely for such a setting in the above works, but thus far a rigorous description of the corresponding spectral networks has not been carried out, forming a roadblock of sorts. The motivation for this paper has been to rectify this situation.

The above subjects have been studied intensively in the case of quadratic differentials within the past decade. In that case, the behaviour of the spectral network simplifies dramatically (trajectories never intersect), and rigorous treatments are often possible. Recent results include the Borel summability of WKB solutions to second-order Schr\"odinger-like operators away from the Stokes graph \cite{koikeschafke,nikolaev2023exact,NEMES2021105562}, the correspondence between quadratic differentials and stability conditions \cite{bridgeland2015quadratic,chq,haiden}, the relation to topological recursion \cite{IWAKI2022108191,iwaki2023topological}, cluster algebras \cite{iwaki2014exact}, orthogonal polynomials \cite{ammt}, and hyperk\"ahler geometry \cite{GAIOTTO2013239, bridgeland_dt_geometry,bs_hk}. In each case, one hopes to extend these ideas to cubic and higher differentials.

When passing to cubic differentials, many new difficulties appear. Most significantly, trajectories can intersect, which does not occur in the quadratic differential case, and can in fact ``give birth'' to further trajectories, which may then give birth to additional trajectories, possibly never terminating and ending up dense in the surface. Thus the description of spectral networks becomes much more complex in passing from quadratic to cubic differentials. Recently some work has been done establishing existence of reasonable-behaving spectral networks \cite{kuwagaki2024generic}, but explicit and rigorous descriptions of higher rank spectral networks associated to any particular cubic differential remain absent from the literature.

In this paper we provide some initial steps for the rigorous determination of the structure of spectral networks coming from cubic differentials; more precisely, we will establish some tools and then apply them to determine the so-called ``BPS structure'' obtained by counting special degenerate trajectories in the network. Our approach is to take the perspective of flat geometry, where many existing tools can be used and generalized. A meromorphic cubic differential on a Riemann surface $X$ determines a flat metric on the complement of its singular points, i.e. it is locally made of pieces of the Euclidean plane; trajectories are then geodesics. One of our main tools in the paper will be the notion of \emph{spectral core}, which generalizes the classical notion of core of a flat surface to the setting of spectral networks. The spectral core controls all of the possible birthed trajectories, giving us constraints on the possible complications that arise.

Our results are as follows. We first formulate the definition of a spectral network for polynomial cubic differentials, and prove basic results on the behaviour of their trajectories. We define the notion of spectral core, and prove that it is always the union of polygons whose boundary sides alternate between zeros and regular points. Then we turn to the specific examples, treating the case $d=2,3$ explicitly. We show that polynomial cubic differentials always have at most finitely many saddle and tripod connections, which correspond to BPS states in the physics terminology. For low degree, we determine precisely which cubic differentials have saddles and tripods of any given phase of their periods. We also determine the locus where the spectrum of saddles and tripods (after rotating the differential) changes. This latter data can be packaged into the data of a \emph{BPS structure} associated to the cubic differential, so that what we have determined is the so called \emph{wall of marginal stability} in the terminology of \cite{Gaiotto:2012rg}. Finally, we show that the resulting family of BPS structures is in fact a \emph{variation of BPS structure} -- in particular that the count of saddles and tripods satisfies the Kontsevich-Soibelman wall-crossing formula \cite{kontsevich2008stability}.



\subsection{Motivation/Applications}
Although the main examples and questions addressed in this paper may seem exotic to experts on flat surfaces, they are in fact closely related to a range of interesting geometric applications, which justifies their careful study. We elaborate here on some of these.

\paragraph{WKB analysis.}
One of the definitive roles of spectral networks in mathematics is their importance in the study of formal solutions to Schr\"odinger-like equations, known as \emph{exact WKB analysis} \cite{hondabook,iwaki2014exact,aoki1991new,Aoki_2005,takei,shapiro}. These were introduced in the case of second order ODEs (corresponding to quadratic differentials) as \emph{Stokes graphs}, and the fundamental result establishes that the complement of the Stokes graph is precisely the locus in the $\hbar$-plane where the formal divergent ``WKB solutions'' can be resummed appropriately \cite{koikeschafke,nikolaev2023exact,NEMES2021105562}. For higher-order equations, the corresponding spectral networks \cite{Gaiotto:2012rg,hondabook,Berk1982NewSL} are much more complicated and arise from cubic or higher differentials (or tuples thereof). Our result thus establishes spectral networks relevant for a particular class of third-order equations of the form
\begin{equation*}
\left(    \hbar^3 \frac{d}{dx^3}+R(x,\hbar)\frac{d}{dx}+Q(x,\hbar))\right)\psi=0
\end{equation*}
with $Q=Q_0+\hbar Q_1+\ldots$ and $R$ certain polynomials in $\hbar$ such that $Q_0(x)dx^3$ is a polynomial cubic differential. From here, one can proceed to study the summability of the formal series and establish rigorously their properties.

\paragraph{BPS structures and Donaldson-Thomas theory.}
The notion of BPS structure was introduced by Bridgeland \cite{bridgeland2019riemann} to describe the behaviour of (generalized) Donaldson-Thomas invariants of Joyce-Song/Kontsevich-Soibelman in algebraic geometry, or equivalently, the BPS indices of Gaiotto-Moore-Neitzke in supersymmetric QFT. Roughly speaking, an (integral) BPS structure $(\Gamma,Z,\Omega)$ consists of
\begin{itemize}
\item 
a lattice $\Gamma$ equipped with an anti-symmetric pairing $\langle \cdot , \cdot \rangle$, 
\item 
a group homomorphism $Z : \Gamma \to {\mathbb C}$ called the \emph{central charge}, and
\item 
a map of sets $\Omega : \Gamma \to {\mathbb Z}$ called the \emph{BPS invariants}, 
\end{itemize}
satisfying certain conditions, crucially that $\Omega(\gamma)=\Omega(-\gamma)$ always. 

Mathematically, the prime known examples arise from the space of stability conditions ${\rm Stab}(\mathcal{D})$ for certain triangulated CY3 categories $\mathcal{D}$, in which stability conditions are identified with certain quadratic differentials \cite{bridgeland2015quadratic,chq}. From a stability condition, we then obtain DT invariants, forming a BPS structure. Indeed, a similar story is believed to exist for higher differentials (or even tuples of differentials), but very little in this direction has been established \cite{wu}. Nonetheless, especially through the work of Gaiotto-Moore-Neitzke, it is reasonable to expect the theory to extend to higher differentials.

In this paper we construct a BPS structure following the approach of Gaiotto-Moore-Neitzke from a specific class of cubic differentials $\varphi$. From this perspective, what we are doing is establishing rigorously the expected properties of the BPS spectrum of the quantum field theory associated to $\varphi$. As far as we aware, this is the first such proof for theories associated to cubic (as opposed to quadratic) differentials. Part of our goal in doing so is to provide tools that may be of wider use in analyzing such problems.

\paragraph{Harmonic maps.}
The spectral networks appearing in this paper appear in the work of \cite{Neitzke17,dumasneitzke} on harmonic maps from $\mathbb{C}$ into the Riemannian symmetric space $\mathrm{SO}(3)\setminus\mathrm{SL}(3,\mathbb{R})$. In this case, the harmonic maps are minimal with polynomial growth. The key idea is that the asymptotic behaviour at infinity of the harmonic map determines a point in a certain moduli space, and the coordinates of that point can be computed by Gaiotto-Moore-Neitzke's TBA integral equations \cite{GAIOTTO2013239,gaiotto2010four}. These integral equations fundamentally depend on knowing the behaviour of spectral networks -- in particular their degenerations -- arising from the cubic differential. The same integral equations appear in the remarkable conjectural construction of Hitchin's hyperk\"ahler metric in the famous work \cite{GAIOTTO2013239}. Our results form a first step of a rigorous proof of their conjectures in this class of examples.

\paragraph{Cluster algebras and cluster coordinates.}
 In particular, by considering cubic differentials of degree $d>1$, one expects to obtain spectral networks closely related to the cluster algebra structure on the Grassmannian ${\rm Gr}(3,d+3)$ \cite{Neitzke17,scott_2006}. For $2\leq d\leq 5$ these are known to be cluster algebras of ``finite type'' and thus more tractable than the general case. 
 More precisely, a procedure known as abelianization \cite{GAIOTTO2013239,Neitzke17,Hollands:2013qza,nikolaevabelianization,hollandskidwai,wu,morrissey} is expected to construct a cluster atlas on certain character varieties and their generalizations from the data of a spectral network of a cubic differential $\varphi$. We expect that the techniques of this paper can be used to determine exactly the spectral networks appearing in these examples, from which it would then be possible to compute the abelianization map and study a number of further conjectures.

\subsection{Main results}

The paper consists roughly of two parts. In the first, we establish some general results on spectral networks coming from cubic differentials, and in the second part give a complete description of the BPS structure in the example of polynomial cubic differentials of degree at most $3$.

We first introduce the notion of \emph{spectral core} $\mathrm{SCore}(X,\varphi)$ (see Definition \ref{defn:score}) and establish its basic properties and useful lemmas in some generality (not only in the polynomial case). Roughly speaking, the spectral core consists of polygons formed by the complement of real half-planes embedded around poles of order $3$ or higher. For example, we will see for certain differentials with three zeroes on $\mathbf{P}^1$, the spectral core is given by Figure~\ref{fig:introscore}.

\begin{figure}[h]
    \centering
    \vspace{0.2cm}
    \includegraphics[height=0.15\linewidth,width=0.55\linewidth, trim=3cm 9.5cm 3cm 8.5cm, clip]{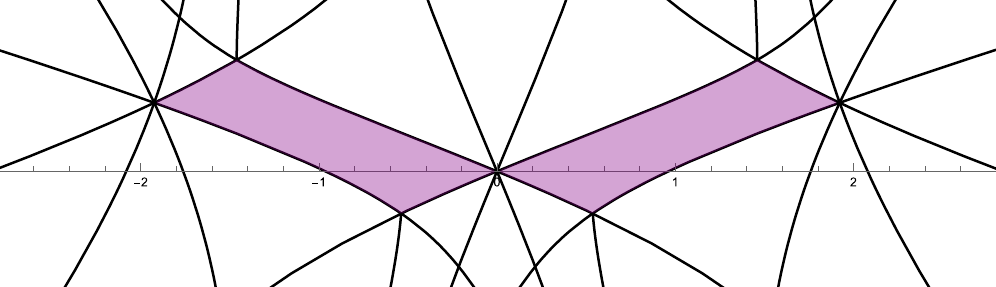}  
    \vspace{0.0cm}
    \caption{One possible spectral core (shaded) of a cubic differential with three simple zeroes.}
    \label{fig:introscore}
    
\end{figure}

The fundamental point is that $\mathrm{SCore}(X,\varphi)$ controls the birthing process of the spectral network. Its structure is furthermore well-constrained, giving us tools to analyze the behaviour of the network.

\begin{theorem}\label{thm:main1intro}For any cubic differential $(X,\varphi)$ such that $X$ is of genus $g$ and $\varphi$ has $n$ zeroes and $p$ poles of order $3$ or higher,
\begin{enumerate}[a)]
    \item  The spectral core $\mathrm{SCore}(X,\varphi)$ is the union of \begin{equation*}4g-4+2n+2p-\delta\end{equation*} Euclidean triangles with disjoint interiors, where $\delta$ is the total number of saddles on its boundary. Furthermore, for any consecutive corners of $\mathrm{SCore}(X,\varphi)$, cannot both be regular (neither a pole nor a zero).
    
    \item The starting point of any trajectory in the spectral network $\mathcal{W}(\varphi)$ is contained in the spectral core.
    \end{enumerate}
\end{theorem}
One corollary of this theorem is that the spectral core, and thus the possible locations of new trajectories, is determined \emph{only knowing $\mathcal{W}^{(1)}$}, the initial trajectories in the network emanating directly from the zeroes of the differential. 
    

We then specialize to polynomial cubic differentials $\varphi_d$. We give criteria and bounds for counting saddles and tripods for arbitrary $d$ (see Corollaries \ref{cor:UpperBound} and~\ref{cor:UpperBoundTripod}). Our second main result treats the examples of $\varphi_d$ for $d=0,1,2,3$ explicitly:

\begin{align*}
&\varphi_0=dx^3 \qquad\qquad \varphi_2= \alpha (x^2-1)dx^3 \\ 
&\varphi_1 = xdx^3 \qquad\qquad \varphi_3=\frac{\alpha x (x-1)}{(x-t)^9}dx^3
\end{align*}
where $\alpha\in \mathbb{C}^\times$, $t\in\mathbb{C^\times}$
In particular, we determine the BPS structure according to the Gaiotto-Moore-Neitzke algorithm by determining the degenerations of spectral networks $\mathcal{W}_\vartheta(\varphi):=\mathcal{W}(e^{-3i\vartheta} \varphi)$ as $\vartheta\in[0,\frac{\pi}{3})$ is varied.
\begin{theorem} \label{thm:main2intro}
Let $\mathcal{W}_\vartheta(\varphi_d)$ denote the spectral network of $\varphi_d$ for $\vartheta\in[0,\frac{\pi}{3})$. \begin{itemize} \item For $d=0,1$, there are no degenerations in the spectral network. 
\item When $d=2$, for all $\alpha\in\mathbb{C}^*$ a saddle connection appears in the spectral network $\mathcal{W}_\vartheta(\varphi_d)$ at exactly one phase. 
Tripods do not appear.

\item When $d=3$, degenerations of $\mathcal{W}_\vartheta(\varphi_d)$ are controlled by a wall-and-chamber structure, determined by the loci $\Delta_k|_{k=1,2,3,4}$ at which two saddle connections make an angle of $k\frac{\pi}{3}$ with each other. We have:

\begin{itemize}\item In all chambers, the saddle connections homotopic to $[-\infty,0]$ and $[1,\infty]$ appear.
\item A saddle connection between $0$ and $1$
appears if and only if $t$ is in the chamber $\mathcal{C}_A$, $\mathcal{C}_B$ or $\mathcal{C}_C$ (or in a wall between them).
\item A tripod between all three zeroes appears if and only if $t$ is in the chamber $\mathcal{C}_A$ or $\mathcal{C}_B$ (or on the wall $\Delta_{1}^{-}$ or $\Delta_{1}^{+}$).
\end{itemize}
\end{itemize}
\end{theorem}
We depict the nontrivial chamber structure for $d=3$ in Figure \ref{fig:introdiagram}, taking $t$ inside a fundamental domain $T$ containing all equivalence classes of $\varphi_d$. The chambers $\mathcal{C}_{A}$, $\mathcal{C}_{B}$, $\mathcal{C}_{C}$ and $\mathcal{C}_{D}$ are separated by walls $\Delta_{1}^{\pm}$, $\Delta_{2}$, $\Delta_{3}$ and $\Delta_{4}$ drawn between singular points $0,1,e^{\frac{i\pi}{3}}$.
\begin{figure}[h]
    \centering
    \includegraphics[width=0.45\linewidth]{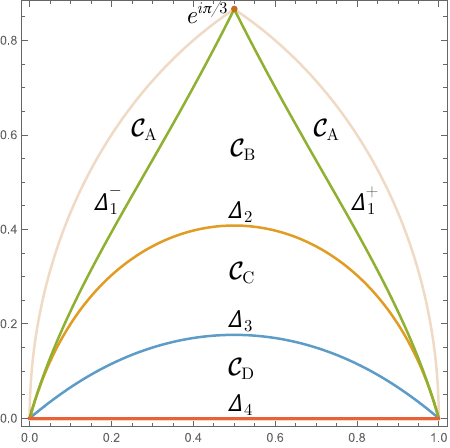}
    \caption{Fundamental domain $T$ in the $t$-plane with walls and chambers listed.}
    \label{fig:introdiagram}
\end{figure}
Representatives of the degenerations and the networks in between them for the chambers are illustrated in Figures \ref{fig:chamberD}, \ref{fig:chamberC}, \ref{fig:chamberB} in Section \ref{sub:d3WallsSecond}.

Once the degenerations are understood, they can be packaged into the corresponding BPS structure \cite{Gaiotto:2012rg,GAIOTTO2013239,Neitzke17}. The lattice is taken to be $\Gamma=H_1(\Sigma^{\times},\mathbb{Z})$, and central charge $Z(\gamma)=\int_\gamma \sqrt[3]{\varphi}$, where $\Sigma^\times$ is a canonical branched triple covering of $X$ determined by $\varphi$. According to Gaiotto-Moore-Neitzke, we can define $\Omega_t$ associated to the cubic differential $\varphi$ as nonzero for homology classes obtained by lifting these degenerations in a canonical way, and $0$ on all other classes (the precise value of $\Omega$ on active classes is subtle to define, but for our purposes is well-known to be equal to $1$). Determining the BPS structure in this way as $t$ varies we obtain a family of BPS structures $(\Gamma_t,Z_t,\Omega_t)_{t\in\mathrm{int}{T}}$.  As a result our main theorem gives the following corollary, conjectured for a particular value of $t$ in \cite{Neitzke17} which motivated our whole investigation:

\begin{corollary}
The Gaiotto-Moore-Neitzke algorithm applied to $\varphi_d$ produces an integral BPS structure $(\Gamma,Z,\Omega)$ with finitely many active classes. Taking as a basis $\gamma_1,\gamma_2$ saddles along $[1,\infty]$ and $[0,1]$ respectively, together with lifts $\gamma_3, \gamma_4$ of a small loop around $\infty$, we have (c.f. \cite{Neitzke17}): 
\begin{itemize}\item In all chambers, the classes \begin{align*}
&\pm(1,0,0,0), \quad \pm(-1,1,1,1), \quad \pm(0,-1,-1,-1) \\ &\pm(0,-1,-1,0), \quad \pm(1,0-1,-1), \quad \pm(-1,1,2,1)\end{align*}  are active and appear as saddle connections.
\item The classes 
\begin{align*}&\pm(0,1,0,0), \quad \pm(-1,0,1,0), \quad \pm(1,-1,-1,0) \end{align*} 
are active and appear as saddle connections if and only if $t$ is in the chamber $\mathcal{C}_A$, $\mathcal{C}_B$ or $\mathcal{C}_C$.
\item The classes 
\begin{align*}\pm(1,1,0,0), \quad \pm(-2,1,2,1),\quad \pm(1,-2,-2,-1) 
\end{align*}
are active and appear as critical tripods if and only if $t$ is in the chamber $\mathcal{C}_A$ or $\mathcal{C}_B$.
\end{itemize}
\end{corollary}

Finally, we verify that the result for different chambers satisfies the expected wall-crossing formula:

\begin{theorem} For $d=3$, the BPS structures $(Z_t,\Gamma_t,\Omega_t)$ form a variation of BPS structures over the 1-dimensional complex manifold $M=\mathrm{int}\,T$.
\end{theorem}
The analogous statement is true for $d<3$ as well, but is essentially trivial there due to the absence of wall-crossing. 

This final theorem first of all confirms the expectations of physicists about the behaviour of the BPS spectrum in the case of the relevant theory \cite{Gaiotto:2012rg,gaiotto2012n}. Mathematically, it also provides some circumstantial evidence for the presence of a rich story involving cubic differentials and stability conditions, hinting towards a Bridgeland-Smith type correspondence for higher differentials \cite{chq,bridgeland_dt_geometry}. Some work towards this idea was done in \cite{wu}, though the analysis in that work was for the fixed chamber $\mathcal{C}_D$. Even in the $d=3$ case, it would be extremely desirable to determine an appropriate 3-Calabi-Yau category whose space of stability conditions could be identified (globally) with the cubic differentials we have considered, to determine their Donaldson-Thomas invariants, and obtain agreement with the BPS invariants obtained in physics (analogous to \cite{bridgeland2015quadratic,chq,haiden,kidwai2024donaldson,qiu2024moduli} in the quadratic differential case).

In addition, there remain numerous interesting questions to address with the approach we have taken. Most immediately, we hope to apply our machinery to determine the BPS structure for higher degree polynomial cubic differentials; especially, the $d=4$ and $d=5$ cases, corresponding to the case in which one expects additional finiteness properties. For $d>5$, one would like to furthermore understand the conditions ensuring finiteness of the network. We leave these problems for future work.




\subsection{Acknowledgements}
This project would never have existed without the suggestion and encouragement of Andy Neitzke, who posed several related questions while hosting the first named author a number of years ago. He hopes this step towards the mathematical theory of \cite{GAIOTTO2013239} can serve as a token of gratitude. In addition, we thank Tom Bridgeland, Aaron Fenyes, Nikita Nikolaev, and Shinji Sasaki for helpful discussions. Many figures appearing in this paper were plotted using \cite{neitzkeswn} written by Andy Neitzke.

The work of O.K. was supported by the Leverhulme Trust’s Research Project Grant “Extended Riemann-Hilbert Correspondence, Quantum Curves, and Mirror Symmetry” and the Improvement on Competitiveness in Hiring New Faculties Funding Scheme of The Chinese University of Hong Kong. Part of this research was conducted while visiting the Okinawa Institute of Science and Technology (OIST) through the Theoretical Sciences Visiting Program (TSVP)’s thematic program “Exact Asymptotics: From Fluid Dynamics to Quantum Geometry”.

The work of G.T. is supported by the Beijing Natural Science Foundation IS23005 and the French National Research Agency under the project TIGerS (ANR-24-CE40-3604).

\section{Flat geometry of cubic differentials}

\subsection{Flat structure}

On a compact Riemann surface $X$ with canonical bundle $\omega_X$, a (meromorphic) cubic differential $\varphi$ is a meromorphic section of the line bundle $\omega_{X}^{\otimes 3
}$. Such an object defines a geometric structure (an atlas of distinguished coordinate charts) in the following way. Denoting by $X^{\ast}$ the surface $X$ punctured at the zeros and poles of $\varphi$, any branch of $\int^{z} \varphi^{1/3}$ (with any fixed lower endpoint in $X^{\ast}$) defines a local biholomorphism from open subsets of $X^{\ast}$ to open subsets of $\mathbb{C}$.

The transition maps of this atlas of local coordinates include translations and rotations of order three. More precisely, every transition map is of the form
\begin{equation}
    z \mapsto \omega^r z + c
\end{equation} 
where $\omega$ is a primitive third root of unity, $r=0,1,2$, and $c$ is a constant. An atlas of this form is called a \emph{$\frac{1}{3}$-translation structure} (see \cite{BCGGM3,tahar} for references). To avoid such cumbersome terminology, in the rest of the paper, we will designate such a pair $(X,\varphi)$ as a \textit{flat surface}, with the understanding we really mean this particular class of flat surfaces. For sake of simplicity, we will sometimes use the single notation $\varphi$ to denote both the cubic differential $\varphi$ and the underlying complex structure $X$.
\par
In such a surface, the pullback of the standard metric of the flat plane defines a flat metric $|\varphi^{1/3}|$ that is singular at the zeros, simple poles and double poles of $\varphi$ (poles of order at least three correspond to points at infinity for the flat metric). Furthermore, there is a globally defined notion of direction, or \emph{slope} (up to a rotation of angle $\frac{2\pi}{3}$), given by passing to any chart.
\par
Outside zeros and poles of the differential, the surface is locally isometric to the standard Euclidean plane.

\subsection{Local model of singularities}\label{sub:localmodel}

The local behaviour of $\varphi$ around any given point $p$ in $X$ is well-known (see e.g. \cite{BCGGM3} or \cite{tahar}), which we recall here. We refer to points which are either zeros or poles of $\varphi$ as \emph{singularities}, and all other points as \emph{regular points}. It is easy to see around any regular point that  $\varphi$ is locally biholomorphic to $(\mathbb{C},dz^3)$.

Any singularity of order $k \in \mathbb{N}^{\ast} \cup \lbrace{ -1,-2, \rbrace}$ is called a \textit{conical singularity} of the flat metric. Indeed, some neighborhood of such a point is isometric to a neighborhood of the vertex of an infinite cone of angle $\frac{2\pi}{3}(k+3)$ (see Section 2.3 in \cite{BCGGM3} for details).
Poles of order greater than or equal to $3$ have a different character than conical singularities, and we refer to them as \emph{higher order poles}. The local behaviour for higher order poles is also completely classified (see again \cite{BCGGM3}):
\begin{itemize}
\item At a pole of order $3$, $\varphi$ is locally biholomorphic to $\frac{\alpha^{3}dz^{3}}{z^{3}}$ for some $\alpha \in \mathbb{C}^{\ast}$. We refer to $\alpha^{3}$ as the \textit{cubic residue} at the pole. Some neighborhood of a triple pole is isometric to an end of an infinite cylinder foliated by closed geodesics of length $|\alpha|$. The slope of the (positively oriented) geodesics encompassing the triple pole is given by the argument of $\int \frac{\alpha dz}{z}=2i\alpha\pi$ (and is defined up to a rotation of angle $\frac{2\pi}{3}$);

\item At a pole of order $k \geq 4$ where $k \notin 3\mathbb{N}$, $\varphi$ is locally biholomorphic to $\frac{dz^{3}}{z^{k}}$ and a neighborhood of the pole is isometric to a neighborhood of infinity in an infinite cone of angle $\frac{2\pi}{3}(k-3)$;

\item At a pole of order $k \geq 4$ whose order is divisible by $3$, $\varphi$ is locally biholomorphic to $\left( \frac{1}{z^{k/3}} + \frac{\alpha}{z} \right)^{3}dz^{3}
$ where $\alpha^{3}$ is the \textit{cubic residue} of $\varphi$ at the pole. If $\alpha\neq 0$, a neighborhood of the pole is isometric to a neighborhood of infinity in an infinite cone of angle $\frac{2\pi}{3}(k-3)$ where a semi-infinite rectangle has been removed, or if $\alpha=0$ without removing it.
\end{itemize}

\subsection{Spectral cover}\label{sub:spectralCover}

Any flat surface $(X,\varphi)$ has a cover $\pi$ (unique up to a rotation of order three) where the pullback $\tilde{\varphi}=\pi^*{\varphi}$ of $\varphi$ is the global power of a $1$-form $\lambda$ (equivalently, pullbacks of roots $\varphi^{1/3}$ are single-valued). Such a cover $({\Sigma},\tilde{\varphi})$ has a translation structure given by the 1-form, and we refer to this data $(\Sigma,\pi,\lambda)$ as the \emph{spectral cover}.

For any singularity $z$ of $\varphi$ whose order $k$ belongs to $3\mathbb{Z}$, the preimage of $z$ is formed by three singularities of order $k$ of the cubic differential $\tilde{\varphi}$.

If $k \notin 3\mathbb{Z}$, $z$ is a ramification point of $\pi$ and its preimage is a singularity of order $3k+6$.

\subsection{Topological index and Gauss-Bonnet formula}\label{sub:windingnumber}

\begin{definition}
In a flat surface $(X,\varphi)$, we consider a smooth simple immersed loop $\gamma$ avoiding the zeros and poles of $\varphi$. The \textit{topological index} $\mathrm{ind}_{\gamma}$ of $\gamma$ is defined as:
$$
\frac{1}{2\pi} \int_{\tilde{\gamma}} \arg \tilde{\gamma}'(t)dt
$$
where $\tilde{\gamma}$ is any lift of $\gamma$ to the spectral cover.
\end{definition}

This number belongs to $\frac{1}{3}\mathbb{Z}$ because the holonomy of the flat metric may be nontrivial. We check immediately that the topological index of counterclockwise simple loop around a singularity of order $m$ is $\frac{m+3}{3}$. The same holds for a loop encompassing several singularities whose total order is $m$.

As polygonal loops (formed by finitely many straight segments) can be uniformly approached by smooth immersed loops, the topological index can also be interpreted in terms of angle defects. We obtain the following formula.

\begin{lemma}[Gauss-Bonnet]\label{lem:GaussBonnet}
In a flat surface $(X,\varphi)$, we consider a topological disk such that:
\begin{itemize}
    \item The boundary is formed by $k$ geodesic segments and $k$ corners of interior angles $\theta_{1},\dots,\theta_{k}$\footnote{These corners may be conical singularities and the angles are not required to be smaller than $2\pi$.};
    \item The boundary does not contain any higher order poles of $\varphi$;
    \item The interior of the disk contains singularities (zeros or poles) of $\varphi$ of orders $a_{1},\dots,a_{n} \in \mathbb{Z}$.
\end{itemize}
Then, we have:
$$
\sum\limits_{j=1}^{k} \theta_{j}
=
(k-2)\pi-\frac{2\pi}{3}\sum\limits_{i=1}^{n} a_{i}
.
$$
\end{lemma}

\subsection{Trajectories}\label{sub:Trajectories}

In the flat surface $(X,\varphi)$, a \textit{trajectory} is an oriented arc parametrized by a real interval that does not contain any singularity (except possibly at the endpoints) and that is locally given by straight segments in the distinguished charts. We use the following terminology for various kinds of trajectories:
\begin{itemize}
\item A trajectory is \textit{maximal} if its endpoint (in the future) is a singularity of $\varphi$ (zeros or poles);
\item A trajectory is \textit{simple} if it is the isometric embedding (without self-intersection) of an interval equipped with the standard metric.
\item A \textit{critical trajectory} is a trajectory (of any slope) starting or ending at one of the conical singularities of $(X,\varphi)$.
\item A \textit{saddle connection} is a trajectory with both ends conical singularities (the same trajectory with an opposite slope is also a saddle connection).
\item A \textit{positive trajectory} (resp. \emph{negative trajectory}) is a trajectory whose oriented slope belongs to $\lbrace{ 0, \frac{2\pi}{3}, \frac{4\pi}{3} \rbrace}$ (resp. $\lbrace{ \frac{\pi}{3}, \pi, \frac{5\pi}{3} \rbrace}$) in any chart of the flat structure.
\item A \textit{real trajectory} is a trajectory that is either positive or negative.
\end{itemize}

We will also use the term \emph{tripod} to refer to a collection of three (not necessarily real) trajectories of same slope modulo $2\pi/3$ intersecting at a point. When this tripod is formed by critical trajectories (each of them emanate from a conical singularity of the flat metric), we will refer to it as a \textit{critical tripod}.

In contrast with the mostly studied case of holomorphic $1$-forms, in flat surfaces of infinite area, nontrivial dynamical behaviour of trajectories implies the existence of saddle connection in the same direction.
    \begin{figure}[h]
        \centering
        \begin{subfigure}{0.32\textwidth}
                \includegraphics[width=\linewidth,trim=1mm 1mm 1mm 1mm, clip]{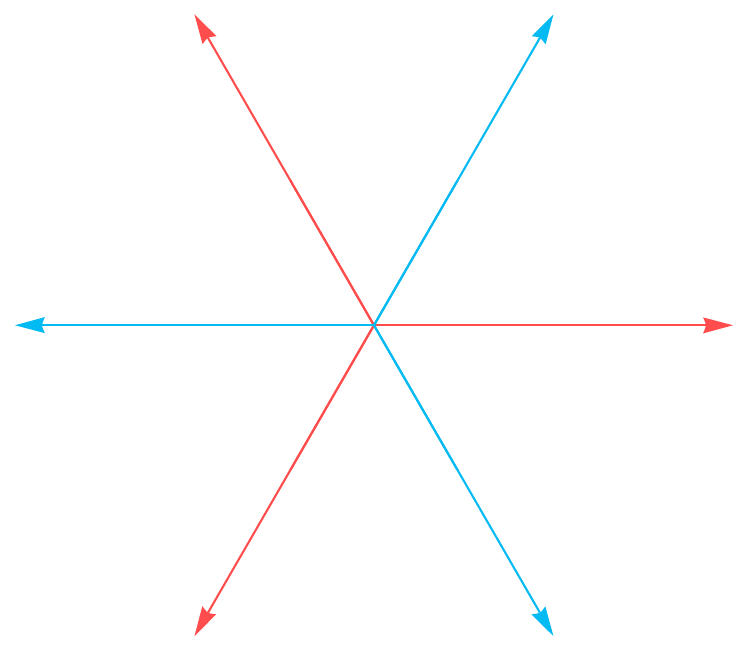}
                    \label{fig:posneg}
        \end{subfigure}

        \caption{Positive (red) and negative (blue) trajectories around a regular point.}
    \end{figure}


\begin{proposition}\label{prop:classTraj}
Suppose that a cubic differential $\varphi$ has at least one higher order pole and at least one conical singularity. Then any maximal real trajectory $\gamma$ belongs to one of the following types:
\begin{enumerate}[i)]
    \item $\gamma$ hits a conical singularity in finite time;
    \item $\gamma$ goes to a higher order pole in infinite time;
    \item $\gamma$ is a (possibly self-intersecting) periodic trajectory;
    \item $\gamma$ is dense in some domain bounded by saddle connections.
\end{enumerate}
In particular, if $\varphi$ has no real saddle connection, any maximal real trajectory either hits a conical singularity in finite time or goes to a higher order pole in infinite time.
\end{proposition}

\begin{proof}
Since $(X,\varphi)$ has at least one pole of order at least three, its lift $({\Sigma},\tilde{\varphi})$ under the canonical triple cover also has such a pole. In the cover $({\Sigma},\tilde{\varphi})$, $\tilde{\varphi}$ is the global power of a meromorphic $1$-form $\lambda$. Therefore ${\Sigma}$ has a structure of translation surface $(\Sigma,\lambda)$ (canonical up to a global rotation of angle $\frac{2\pi}{3}$) and there is a  well-defined directional foliation for each slope. Proposition 5.5 in \cite{tahar}) proves that each (oriented) leaf (on $\Sigma$) belongs to one of the four following types:
\begin{itemize}
    \item periodic trajectories;
    \item trajectories dense in some domain bounded by parallel saddle connections;
    \item trajectories hitting a conical singularity in finite time;
    \item trajectories converging to a pole in infinite time.
\end{itemize}
As periodic trajectories form cylinders such that at least one their ends is bounded by saddle connections, if $({\Sigma},\lambda)$ has no saddle connections in a direction, then no trajectory can be periodic. Similarly, minimal components are bounded by saddle connections with the same slope. It follows that there cannot be dense trajectories in directions where there is no saddle connection. Consequently, if there is no saddle connection in some direction, every trajectory in that direction either hits a conical singularity or goes to a pole.
Finally, every (possibly self-intersecting) trajectory of $(X,\varphi)$ lifts to a simple trajectory of $({\Sigma},\lambda)$ so the classification result applies.
\end{proof}

Outside a specific situation related to the local geometry of triple poles, we prove under the same genericity hypothesis that real trajectories intersect at most finitely many times.

\begin{proposition}\label{prop:intersect}
Suppose $\varphi$ has at least one higher order pole, and that $(X,\varphi)$ has no real saddle connection. For any pair of real trajectories $\gamma,\gamma'$ that do not coincide on any interval, either
\begin{itemize}
\item $\gamma$ and $\gamma'$ intersect finitely many times, or
\item Both $\gamma$ and $\gamma'$ accumulate to some (common) triple pole.
\end{itemize}
Furthermore, for any self-intersecting trajectory, the set of self-intersection points is finite.
\end{proposition}

\begin{proof}
We first consider the case of a self-intersecting trajectory $\gamma$. If the set of self-intersection points is infinite, then this set has an accumulation point $z$. Any small enough neighborhood of $z$ contains infinitely many branches\footnote{A branch is a connected component of the intersection of the trajectory and the neighborhood.} of $\gamma$ accumulating on $z$. It follows that $\gamma$ is infinite and cannot hit a conical singularity in finite time. Trajectory $\gamma$ is therefore a maximal trajectory that converges to a pole of order at least three (see Proposition~\ref{prop:classTraj}) which coincides with $z$. However, the local geometry of these poles is that of the infinite end of a cone (or a cylinder). It follows that for a trajectory converging to the pole, there is a neighborhood of the pole where the trajectory no longer intersects itself.

In the case of a pair of trajectories $\gamma$ and $\gamma'$ that intersect infinitely many times. We denote by $z$ an accumulation point of $\gamma \cap \gamma'$. Since a neighborhood of $z$ has to contain infinitely many branches of $\gamma$ and $\gamma'$, $z$ has to be (up to changing the orientation of $\gamma$ or $\gamma'$) in the limit set of both $\gamma$ and $\gamma'$. In other words, $\gamma$ and $\gamma'$ have to converge to the same pole and $z$ has to coincide with this pole. Looking at the geometry of infinite cones, two such trajectories cannot intersect infinitely many times if they go to infinity in such a cone. On the other hand, if $z$ is a triple pole, a neighborhood of $z$ is isometric to the end of an infinite cylinder. Provided $\gamma$ and $\gamma'$ do not have the same slope in this cylinder, they intersect infinitely many times and the intersection points are arbitrarily far in the cylinder.
\end{proof}

\subsection{Core}\label{sub:core}

In a flat structure induced by a cubic differential, every neighborhood of a pole of order at least three has infinite area. However, the essential information about the metric is contained within a domain of finite area called the \textit{core}. This notion has been introduced in \cite{HKK} and further developed in \cite{tahar}.

Recall that a subset $E$ of a flat surface $(X,\varphi)$ is \emph{convex} if and only if every geodesic segment joining two points of $E$ belongs to~$E$. The \emph{convex hull} of a subset $F$ of a flat surface $(X,\omega)$ is the smallest closed convex subset of $X$ containing~$F$.

\begin{definition}\label{defn:core}
Given a flat surface $(X,\varphi)$, its \emph{core} $\mathrm{Core}(X,\varphi)$ is the convex hull of the conical singularities of $(X,\varphi)$. In particular, the core contains every saddle connection of the surface.
\end{definition}

The core separates the poles from each other. The following result 
shows that the complement of the core has as many connected components as the number of poles (see Proposition~4.4 and Lemma~4.5 of \cite{tahar}).  

\begin{proposition}\label{prop:decomppoles}
For a flat surface $(X,\varphi)$, the boundary of the core is a finite union of saddle connections. Moreover, each connected component of $X \setminus \mathrm{Core}(X,\varphi)$ is a topological disk that contains a unique higher order pole. We refer to these connected components as \textit{polar domains}.
\end{proposition}

The number of boundary saddle connections of polar domains can be related to the topological complexity of the core measured as the number of triangles in any topological triangulation of the core. We recall that a topological triangulation is a system of topological arcs joining the conical singularities of the core that decompose it into topological triangles. The following result follows from Lemmas~4.10 and 7.4 in \cite{tahar}.

\begin{lemma}\label{lem:coreTriangle}
Consider a flat surface $(X,\varphi)$ of genus $g$ with $n$ conical singularities and $p$ poles of order at least three. Let $\beta_{1},\dots,\beta_{p}$ be the number of boundary saddle connections of each polar domain, and write $\beta=\sum\limits_{i=1}^{p} \beta_{i}$.
Then the number of triangles in any topological triangulation of $\mathrm{Core}(X,\varphi)$ is
$$
4g-4+2n+2p- \beta
.
$$
\end{lemma}

In fact, each connected component of the interior of the core can be given a geodesic triangulation made of saddle connection (see Lemma 2.2 in \cite{tahar1}).

\begin{remark}
The decomposition of a flat surface $(X,\varphi)$ into the core and polar domains remains the same when $\varphi$ is replaced by some scaling $\alpha\varphi$ where $\alpha\in \mathbb{C}^{\ast}$. This will \emph{not} be true for the spectral core and spectral polar domains introduced in the sequel.
\end{remark}

\section{Spectral networks and spectral core of cubic differentials}\label{sec:spectral}

In this section, we introduce the construction algorithm for spectral networks in a simplified (but equivalent) form, which is made possible by restricting to the case of cubic differentials. We also discuss the conditions under which the spectral network consists of a finite or infinite number of trajectories.

Next, we introduce the concepts of the spectral core and spectral polar domain, which are refinements of the notions of core and polar domains introduced in Section~\ref{sub:core}. This decomposition of the flat surface is particularly well-suited for describing spectral networks.

\subsection{Spectral networks}

\subsubsection{Construction of spectral networks for cubic differentials}\label{subsub:Algorithm}

For the purposes of this paper, it is actually more convenient to work with a slightly simplified variant of a spectral network in the usual sense. We use the following

\begin{definition}
For a meromorphic cubic differential $\varphi$, we define the \emph{WKB spectral network} $\mathcal{W}(\varphi)$ of $(X,\varphi)$ by an iterative procedure:
\begin{itemize}
    \item The \textit{initial network} $\mathcal{W}^{(0)}=\mathcal{W}^{(0)}(\varphi)$ as the set of maximal trajectories formed by the positive and negative trajectories starting from conical singularities of $(X,\varphi)$.
    \item For any step $k$, the arc system $\mathcal{W}^{(k)}(\varphi)$ is formed by positive and negative trajectories (saddle connections account for both). We define $\mathcal{W}^{(k+1)}$ by adding to $\mathcal{W}^{(k)}$ a positive (resp. negative) maximal trajectory starting at each point where two negative (resp. positive) trajectories of $\mathcal{W}^{(k)}$ intersect, see Figure~\ref{fig:joint}.
    \item The spectral network $\mathcal{W}(\varphi)$ of $\varphi$ is the union $\bigcup\limits_{k \geq 0} \mathcal{W}^{(k)}$, with each trajectory understood to be oriented.
\end{itemize}

We sometimes use $\mathcal{W}_{\vartheta}(\varphi):=\mathcal{W}(e^{-3 i \vartheta}\varphi)$ to denote the spectral network obtained by rotating $\varphi$ by a phase, where $\vartheta\in\mathbb{R}$.
\end{definition}

    \begin{figure}[h]
        \centering
        \begin{subfigure}{0.32\textwidth}
                \includegraphics[width=\linewidth]{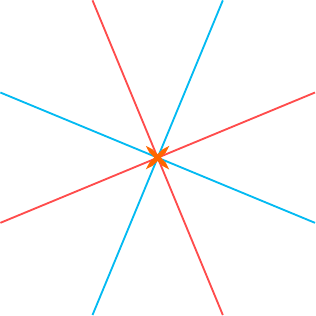}
                \caption{Initial step near a simple zero with all trajectories oriented outward.}
                \label{fig:zero}
        \end{subfigure}
        \hspace{1.5cm}
        \begin{subfigure}{0.3\textwidth}
        \includegraphics[width=\linewidth]{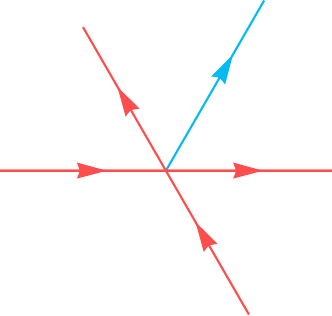}            \caption{Joint with two positive trajectories giving birth to a negative trajectory.}\label{fig:joint}
        \end{subfigure}

        \caption{Construction of spectral networks from cubic differentials.}
        \label{fig:construction}
    \end{figure}
    
\begin{remark}
    The definition we have given is not the usual one. First, we have avoided mention of ``labels'' of the network, in which the slopes are defined in $\mathbb{R}/2\pi\mathbb{Z}$ rather than  modulo $2\pi/3$; these could be restored by considering the network on a cover instead, as in \cite{morrissey,kineider2024spectral}. Furthermore, because we have considered the special case of a $k$-differential ($k=3$), rather than a general tuple of differentials. As a result, the differences of cube roots $\{\alpha_1,\alpha_2,\alpha_3\}$ satisfy 
    $\arg(\alpha_i-\alpha_j)=\arg \alpha_k+\frac{\pi}{2}$ for ${i,j,k}$ a cyclic permutation of $\{1,2,3\}$,
    so that the trajectories are identical to the usual case after rotation by a phase. Furthermore, we ignore complications that arise more generally regarding intersections of trajectories, since we only study the polynomial case as described in \cite{Neitzke17}. See also \cite{kuwagaki2024generic} for recent progress on the technical difficulties in ensuring existence of networks more generally, and \cite{casals2025spectral} for a recent Floer-theoretic perspective.
\end{remark}

One of the basic behaviours of interest in applications is

\begin{definition}
    A (not-necessarily-maximal) trajectory in a spectral network is called a \emph{double trajectory} if it is also a trajectory in the same spectral network when given the opposite orientation (and thus sign).
\end{definition}

\noindent That is, a double trajectory is a pair of overlapping oppositely-oriented trajectories. We sometimes refer to a spectral network as \emph{degenerate} if it contains a double trajectory. Connected components of the union of double trajectories (corresponding to so-called \emph{finite webs} in \cite{Gaiotto:2012db}) appearing in a given spectral network then correspond to BPS states in the physical picture. Saddle trajectories are examples of double trajectories.

    \begin{figure}[h]
        \centering
        \begin{subfigure}{0.28\textwidth}
                \includegraphics[width=\linewidth]{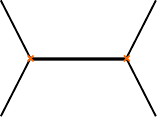}
                \caption{Saddle trajectory between zeros (orange).}\label{fig:saddle}
        \end{subfigure}
        \hspace{1.5cm}
        \begin{subfigure}{0.25\textwidth}
        \includegraphics[width=\linewidth]{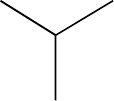}            \caption{Tripod; endpoints may be zeros or regular.}\label{fig:tripod}
        \end{subfigure}

        \caption{Examples of double trajectories}
        \label{fig:degenerations}
    \end{figure}
These two configurations of double trajectories are the only ones which appear in the examples studied in this paper; this is a feature peculiar to our setting with $d\leq 3$ and is by no means the case in general, even for polynomial cubic differentials.

\subsubsection{Examples of spectral networks formed by infinitely many trajectories}
If some trajectory is minimal (somewhere-dense), then it is easy to find infinitely many intersections between critical trajectories of same sign.

\begin{proposition}\label{prop:minimal}
Consider a flat surface $(X,\varphi)$. If there is a real critical trajectory that is dense in some subsurface of $X$, then $\mathcal{W}^{(1)}(\varphi)$ is formed by infinitely many trajectories (and so is the spectral network).
\end{proposition}

\begin{proof}
If there is a critical trajectory $\gamma$ that is dense in some subsurface of $X$, we consider its lift $\tilde{\gamma}$ in the canonical triple cover $({\Sigma},\tilde{\varphi})$. The flat surface $({\Sigma},\tilde{\varphi})$ has a translation structure and $\tilde{\gamma}$ is a horizontal trajectory of this structure. Following standards results about translation structures (see Proposition~5.5 of \cite{tahar}), trajectory $\tilde{\gamma}$ is dense in some surface $Y$ that is either ${\Sigma}$ or a subsurface bounded by some saddle connections having the same direction as $\tilde{\gamma}$. $Y$ is an invariant surface for the horizontal directional flow.
\par
In any case, $Y$ projects on $X$ to a subsurface $\pi(Y)$ whose (possibly empty) boundary is made of real saddle connections. Cases where $\varphi$ does not contain any conical singularity are very few: $(X,\varphi)$ would be either a infinite cylinder, a flat plane or a flat torus. Dense trajectories only happen in the latter case. Ruling out this case, subsurface $\pi(Y)$ contains at least one conical singularity $z$.
Since $\gamma$ accumulates in particular on $z$, we can find a real critical trajectory with the same sign $\gamma$ that intersects infinitely many parallel branches of $\gamma$ in some neighborhood of $z$.
\end{proof}

The specific geometry of triple poles leads to spectral networks of infinite complexity under generic residue conditions.

\begin{proposition}\label{prop:triplepole}
We consider a flat surface $(X,\varphi)$. If $\varphi$ contains a triple pole whose cubic residue $\alpha^{3}$ satisfies $\tan(\arg~\alpha) \notin \sqrt{3}\mathbb{Q} \cup \lbrace{ \infty \rbrace}$, then exactly one of the following statements holds:
\begin{itemize}
    \item $X=\mathbf{P}^1$ and $\varphi$ has exactly two triple poles and no other singularities (in this case, $\mathcal{W}$ is empty);
    \item $\mathcal{W}^{(1)}$ is formed by infinitely many trajectories (and so is the spectral network).
\end{itemize}
\end{proposition}

\begin{proof}
We assume that $(X,\varphi)$ contains a triple pole and at least one conical singularity. neighborhoods of triple poles are foliated by closed geodesics forming a cylinder bounded by at least one conical singularity. We will focus on this cylinder $\mathcal{C}$ which is a translation surface with boundary (and a conical singularity $z_{0}$ on the boundary). In particular, we fix a root $\varphi^{1/3}$ of $\varphi$. Accordingly, we choose a third root $\alpha$ of the cubic residue in such a way that the period of $\varphi^{1/3}$ on positively oriented loops around the pole is $2i\alpha\pi$.
\par
By hypothesis on the argument of $\alpha$, the boundary saddle connections of $\mathcal{C}$ cannot be real (otherwise $\tan(\arg~\alpha)$ would be $0$ or $\pm \sqrt{3}$). It follows that we have two real trajectories (either positive or negative) $\gamma$ and $\gamma'$ going from $z_{0}$ to the pole and forming an angle of $\frac{2\pi}{3}$. Indeed, in contrast with the whole surface $X$ where slopes are defined up to $\frac{2\pi}{3}$, the cylinder has trivial linear holonomy so slopes are completely defined.
\par
Trajectories $\gamma$ and $\gamma'$ intersect infinitely many times and cut out a neighborhood of the triple pole into disjoint isometric parallelograms, see Figure 17 of \cite{Gaiotto:2012rg} for an illustration. The angles at the corners of the parallelogram are $\frac{\pi}{3}$ and $\frac{2\pi}{3}$ while one diagonal (being a closed geodesic of the cylinder) has a slope determined by $\arg~\alpha$.


    
\par
Since the parallelograms have the same periods forming a lattice $\Lambda \subset \mathbb{C}$, any trajectory going to the triple pole projects to a trajectory in a flat torus $\mathbb{C}/\Lambda$. 
\par
Following the algorithm of spectral networks, any intersection between $\gamma$ and $\gamma'$ gives birth to a trajectory which is the bisector of one of the corners of angle $\frac{2\pi}{3}$ of an incident parallelogram. Since there are infinitely many such parallelograms, the only case where $\mathcal{W}^{(1)}$ is formed by infinitely many trajectories implies some of these new trajectories are not disjoint. This would imply that these trajectories project to periodic trajectories in torus $\mathbb{C}/\Lambda$. 
\par
This happens if and only if the two length ratio between the two sides of the parallelogram is a rational number. This is equivalent to require that $\mathbb{C}/\Lambda$ is tiled by equilateral triangles. This is also equivalent to require that the slope of the bisector of a corner of angle $\frac{2\pi}{3}$ in the parallelogram is the same as the argument of a point of $\mathbb{Z}+e^{\frac{i\pi}{3}}\mathbb{Z}$. This condition is then equivalent to require that $\tan(\arg~\alpha) \in \mathbb{Q}\sqrt{3} \cup \lbrace{ \infty \rbrace}$. By hypothesis, we are not in this case and there are infinitely many distinct new trajectories in $\mathcal{W}^{(1)}$.
\end{proof}

\subsubsection{Finiteness results}
In general, determining whether $\mathcal{W}$ has infinitely many trajectories or not is a difficult and unsolved problem. Here we show the much weaker fact that usually this can only happen in infinitely many steps.
\begin{lemma}\label{lem:AlgoFinite}
Suppose that $(X,\varphi)$ satisfies the following conditions:
\begin{itemize}
    \item $\varphi$ has at least one higher order pole,
    \item $\varphi$ has no real saddle connection,
    \item any triple pole $p$ of $\varphi$ has a cubic residue $\alpha_{p}^{3}$ satisfying $\arg(\alpha_{p}) \in \sqrt{3}\mathbb{Q} \cup \lbrace{ \infty \rbrace}$.
\end{itemize}

Then, for any step $k\geq0$, $\mathcal{W}^{(k)}$ is formed by finitely many trajectories.
\end{lemma}

\begin{proof}
We proceed by induction. The initial step $\mathcal{W}^{(0)}$ is formed by finitely many critical trajectories. Assuming that the property holds for $\mathcal{W}^{(k)}$, we prove that $\mathcal{W}^{(k+1)}$ is formed by finitely many trajectories.

Assuming by contradiction that $\mathcal{W}^{(k+1)}$ is formed by infinitely many trajectories, it follows that two trajectories $\gamma,\gamma'$ of $\mathcal{W}^{(k)}$ (having the same sign) intersect infinitely many times (so that the algorithm produces infinitely many new trajectories). Proposition~\ref{prop:intersect} then proves that $\gamma$ and $\gamma'$ converge to the same triple pole $p$. Without loss of generality, we assume that $\gamma$ and $\gamma'$ are positive. Besides, we consider a chart of the cylinder containing $p$ where the slope of $\gamma$ is $0$ while the slope of $\gamma'$ is $\frac{2\pi}{3}$. Trajectories $\gamma$ and $\gamma'$ cut out a neighborhood of $p$ into disjoint isometric parallelograms. Like in the proof of Proposition~\ref{prop:triplepole}, new trajectories appearing at the intersection points of $\gamma$ and $\gamma'$ project to trajectories on a flat torus defined by identifying the sides of the parallelograms. By hypothesis on the argument of $\alpha$, this torus is tiled by equilateral triangles so the new trajectories project to periodic trajectories of the torus. It follows that only finitely many new trajectories appear from the intersection points between $\gamma$ and $\gamma'$. This is the final contradiction.
\end{proof}



\subsection{Spectral core}\label{sub:SpectralCore}
In this section we introduce our main tool in this paper, the \emph{spectral core} of a flat surface $(X,\varphi)$. In contrast with the notion of core described in Section~\ref{sub:core}, the notion of spectral core crucially depends on the slope of real trajectories. In other words, for a cubic differential $\varphi$ and some nonzero complex number $\alpha\varphi$, we can obtain very different spectral cores for $\varphi$ and $\alpha\varphi$.

\subsubsection{Half-plane immersions}

\begin{definition}\label{defn:admissible}
An \textit{admissible half-plane immersion} is an isometric immersion of open half-plane $\mathbb{H}$ in $(X,\varphi)$ such that the image of the boundary $\partial \mathbb{H}$ is made of real trajectories.
\end{definition}
For short, sometime we refer to admissible half-plane immersions as \emph{admissible half-planes} and usually conflate them with their images.

\begin{proposition}\label{prop:admissible}
For any flat surface $(X,\varphi)$, we have exactly one of either:
\begin{enumerate}[a)] 
\item $(X,\varphi)$ is one of the two exceptional cases where
\begin{itemize}
    \item $X$ is $\mathbf{P}^1$, $\varphi$ has exactly two triple poles, no other singularity and $(X,\varphi)$ is an infinite cylinder,
    \item $X$ is an elliptic curve, $\varphi$ has no singularity and $(X,\varphi)$ is thus a flat torus.
\end{itemize}
    
    \item Every admissible half-plane is contained in a unique \textit{maximal} immersion $f$ such that the closure of the image of $f$ contains a conical singularity.
\end{enumerate}
\end{proposition}

\begin{proof}
For a given admissible half-plane immersion into $(X,\varphi)$, either it extends to a maximal immersion of a half-plane or it extends to an immersion of the complete flat plane. In this case, $(X,\varphi)$ is a quotient of $\mathbb{C}$. It follows that $(X,\varphi)$ is either a flat torus (quotient by a lattice) or an infinite flat cylinder (quotient by a translation).
\par
For a given maximal immersion $f$, suppose that the closure of the image of $f$ does not contain any conical singularity. Then, $f$ extends continuously to each point $t$ of the boundary line and $f(t)$ is not a singularity of $\varphi$ ($f$ is an isometry and the flat metric induced by $\varphi$ is geodesically complete). The image of this boundary line is a trajectory $\gamma$ of $(X,\varphi)$. 
\par
Since the closure of the image of $f$ does not contain any conical singularity, $\gamma$ cannot hit any conical singularity and therefore cannot be a critical trajectory. Moreover, as minimal components of the directional dynamics in the spectral cover are bounded by saddle connections, $\gamma$ cannot be minimal because there would be a conical singularity in the closure of $\gamma$ and therefore in the closure of the image of $f$. It follows that $\gamma$ is either a periodic geodesic or a trajectory going from some pole of order at least three pole to another pole of order at least three. In both cases, $\gamma$ is a compact arc in $X$ and $f$ can be extended in a neighborhood this arc. This contradicts the maximality hypothesis.
\end{proof}

\subsubsection{Spectral polar domains and spectral core}

For a given flat surface $(X,\varphi)$ having at least one conical singularity, a \textit{spectral polar domain} is a connected component of the union of the images of admissible half-plane immersions. Spectral polar domains have a simple classification:

\begin{lemma}\label{lem:spectralPoles}
Every spectral polar domain is a topological open disk containing a single higher order pole, which in the case of a triple pole must have real cubic residue. Conversely, any such pole of order $k$ belongs to a unique spectral polar domain, which is:
\begin{itemize}
    \item the union of the images of at most $2k-6$ maximal immersions if $k \geq 4$;
    \item the image of a unique maximal immersion if $k=3$.
\end{itemize}
\end{lemma}

\begin{proof}
We first prove that any spectral polar domain $\mathcal{D}$ retracts to a unique higher order pole and is therefore contractible. Given any point $x \in \mathcal{D}$ in the image of a maximal immersion $f$, there is a unique arc $I_{x}$ of length at least $\pi$ in the circle of directions such that trajectories starting from $x$ in a direction of $I_{x}$ remain in the spectral polar domain. A continuity argument shows that all these trajectories converge to the same pole of higher order $p$. For the same reason, for $y$ close enough to $x$, trajectories starting from $x$ in a direction of $I_{y}$ also converge to $p$. Since $\mathcal{D}$ is a connected component of the union of all the (images of) admissible half-planes, the claim holds for every point in $\mathcal{D}$.
\par
The arc $I_{x}$ does not coincide with the full circle because the boundary of the image of $f$ contains a conical singularity (and therefore at least one trajectory starting from $x$ hits this singularity). Then, arc $I_{x}$ has a well-defined midpoint $\theta_{x}$. Pushing any such point $x$ in the direction of $\theta_{x}$ with unit-speed, we define a continuous flow on the spectral polar domain that realizes a retract on $\lbrace{ p \rbrace}$. Every spectral polar domain is thus contractible and contains a unique pole of higher order.
\par
For any pole $p$ of order $k \geq 4$, there is a neighborhood of $p$ covered by (the images of) $2k-6$ admissible half-planes forming a cyclic family $\mathcal{F}$ where two consecutive half-planes overlap on the immersion of a flat cone of angle $\frac{2\pi}{3}$. This immediately follows from the local geometry of the poles (see Section~\ref{sub:localmodel}). This neighborhood is contained in a spectral polar domain $\mathcal{D}$. Any maximal admissible immersion whose image is contained in $\mathcal{D}$ restricts to one immersion of this cyclic family. Therefore, extending each immersion of $\mathcal{F}$ to a maximal immersion produces a family of $2k-6$ maximal immersions that cover $\mathcal{D}$.
\par
Now we consider the case of triple poles. In this section, we consider only flat surfaces containing at least one conical singularity, and the family of closed geodesics forming a neighborhood of the triple pole is bounded by a family of saddle connections. Unless the cubic residue is real, closed geodesic of the cylinders are not real. In this case, for any admissible immersion of a half-plane, some trajectories contained in the image will go to the triple pole while some other will not. It follows that at least one trajectory contained in the admissible immersion hits a conical singularity. This is a contradiction.
\end{proof}

\noindent The essential notion we introduce in the present paper is the following.

\begin{definition}\label{defn:score}
The \emph{spectral core} $\mathrm{SCore}(X,\varphi)$ is the complement of the union of the spectral polar domains.
\end{definition}

The fundamental point is that the spectral core controls any birthed trajectories of the spectral network.

\begin{theorem}\label{thm:spectralcore}
For any flat surface $(X,\varphi)$, the starting point of any trajectory of its spectral network $\mathcal{W}(\varphi)$ is contained in its spectral core $\mathrm{SCore}(X,\varphi)$.
\end{theorem}

\begin{proof}
By hypothesis, conical singularities belong to $\mathrm{SCore}(X,\varphi)$ so the statement holds for trajectories of $\mathcal{W}^{(0)}$. We are going to prove that induction the following equivalent statement: for any $k \geq 0$ and any admissible half-plane $f$, no trajectory of $\mathcal{W}^{(k)}$ starts inside ${\rm im}(f)$.
\par
Assuming that the statement holds for rank $k$, we just have to prove that no pair of trajectories of same sign in $\mathcal{W}^{(k)}$ intersect inside ${\rm im}(f)$. It will imply the statement for rank $k+1$.
\par
Without loss of generality, we consider two positive trajectories of $\mathcal{W}^{(k)}$ that have a nonempty intersection with the image ${\rm im}(f)$ of some admissible half-plane immersion $f$. We consider the (possibly disconnected) pullbacks $f^{\ast}\gamma$ and $f^{\ast} \gamma'$ in the half-plane. In the half-plane, there are three positive slopes $\theta_{1},\theta_{2},\theta_{3}$:
\begin{itemize}
    \item $\theta_{1}$ is parallel to the boundary line;
    \item $\theta_{2}$ is oriented to leave the half-plane by crossing the boundary line;
    \item $\theta_{3}$ is oriented to go infinitely far from the boundary.
\end{itemize}
By induction hypothesis, the starting points of trajectories $\gamma$ and $\gamma'$ do not belong to ${\rm im}(f)$ so the starting points of each branch of their pullbacks $f^{\ast}\gamma$ and $f^{\ast} \gamma'$ are on the boundary line of the half-plane. Moreover, the branches intersect the boundary line transversely (because we consider the open half-plane). Therefore, all the branches of $f^{\ast}\gamma$ and $f^{\ast} \gamma'$ belong to slope $\theta_{3}$. They are parallel and therefore do not intersect.
\end{proof}

\begin{remark}
The spectral network $\mathcal{W}(\varphi)$ of a cubic differential $\varphi$ can be deduced from its restriction to the spectral core $\mathrm{SCore}(X,\varphi)$. Indeed, the starting point of every trajectory of $\mathcal{W}(\varphi)$ belongs to $\mathrm{SCore}(X,\varphi)$. The complement of the spectral core is formed by the images of the immersed Euclidean half-planes and the behaviour of the (pullbacks of) trajectories of $\mathcal{W}(\varphi)$ in these half-planes just depend on their intersection point with the boundary line of the half-plane and their incidence angle. 
\end{remark}

\subsubsection{Structure of the spectral core}
Here we provide concrete geometric constraints on the spectral core, which will be used to analyze the possible spectral cores of a given cubic differential in the sequel. 

The boundary of any spectral polar domain is formed by (the images of) the boundaries of maximal admissible half-planes. It is thus a piecewise polygonal curve formed by finitely many geodesic segments with real slopes between vertices that may be either conical singularities or regular points of the flat metric. By convention, we count any conical singularity of the boundary of a spectral polar domain as a corner even if the interior angle at this point is equal to $\pi$.

\begin{lemma}\label{prop:SPECpolarBoundary}
In a flat surface $(X,\varphi)$, the two ends of a boundary edge of a spectral polar domain $\mathcal{D}$ cannot be both regular points (at least one of them has to be a conical singularity). In particular, if no boundary edge of $\mathcal{D}$ is a saddle connection, then the corners are alternately regular points and conical singularities, and the boundary has an even number of sides.
\end{lemma}

\begin{proof}
Let $\alpha$ denote a boundary edge of a spectral polar domain $\mathcal{D}$ in $(X,\varphi)$ such that the boundary corners $p,q$ on either end of $\alpha$ are both regular points. Then there is a maximal half-plane $f:\mathbb{H} \longrightarrow X$ such that some interval $I=[a,b] \subset\mathbb{R}$ of the boundary line is mapped to the boundary edge $\alpha$, with $f(a)=p$ and $f(b)=q$. Following Proposition~\ref{prop:admissible}, $f(\mathbb{R})$ contains a conical singularity $s$ (in the exceptional cases of Proposition~\ref{prop:admissible}, there is no such spectral polar domain to begin with).

By definition of a boundary edge, $f^{-1}(s)$ is disjoint from $I$. Moreover, as $p$ is a boundary corner of $\mathcal{D}$, for some $\epsilon>0$ there a locally isometric immersion of an infinite open cone of angle $\pi+\epsilon$ into $X$ that maps the vertex of the cone to $p$ and whose image is contained in $\mathcal{D}$. Since $q$ lies on the boundary of $\mathcal{D}$, it is disjoint from the interior of the cone so the half-line $(-\infty,a)$ is mapped to the image of the interior of the cone. It follows that $(-\infty,a)$ is disjoint from $f^{-1}(s)$; for the same reason, $f^{-1}(s)$ is also disjoint from $(b,+\infty)$. Thus the image $f(\mathbb{R})$ of the boundary line does not contain any conical singularity, contradicting the above.
\end{proof}

We can furthermore give constraints on the angles at the corners of a spectral polar domain.

\begin{proposition}\label{prop:spectralDomainAngles}
For a given flat surface $(X,\varphi)$, consider the spectral polar domain $\mathcal{D}$ around a pole of order $b \geq 4$. The boundary of $\mathcal{D}$ satisfies the following properties:
\begin{itemize}
    \item The angles of corners at conical singularities are of the form $\frac{k\pi}{3}$ with $k \geq 3$. 
    \item The angles of corners at regular points are equal to $\frac{4\pi}{3}$ and these points are intersections of two real critical trajectories of the same sign.
    \item Let $\theta_{1},\dots,\theta_{\beta}$ be the angles of the $\beta$ corners of the spectral polar domain. Then we have $$\sum\limits_{j=1}^{\beta} \theta_{j} = (b-3)\frac{2\pi}{3} + \pi \beta.$$
\end{itemize}
\end{proposition}

\begin{proof}
By definition, the angle of each interior corner of a spectral polar domain $\mathcal{D}$ is of the form $\frac{k\pi}{3}$ with $k \geq 3$, which for conical singularities gives the first statement. 

To show the second statement, note that for regular corner points on the boundary, the angle is either $\frac{4\pi}{3}$ or $\frac{5\pi}{3}$ ($\pi$ cannot appear because then the two adjacent boundary edges would form a real saddle connection in the boundary of $\mathcal{D}$). We prove now that this angle cannot be $\frac{5\pi}{3}$.

Indeed, if there is such a regular boundary corner point $p$ with an angle of $\frac{5\pi}{3}$, any small enough disk ${D}$ centred on $p$ is disjoint from any conical singularity and decomposes into:
\begin{itemize}
    \item A sector of angle $\frac{5\pi}{3}$ contained in the images of two locally isometric immersions $f_1,f_2$ of half-planes;
    \item A sector of angle $\frac{\pi}{3}$ contained in the spectral core $\mathrm{SCore}(X,\varphi)$. We denote by $q,r$ the corners of this sector in the boundary $\partial {D}$ of the disk.
\end{itemize}
The chord connecting $q$ and $r$ can be extended as a line $\ell$ contained in the union of ${D}$, $\mathrm{im}(f_1)$ and $\mathrm{im}(f_2)$. Then we can construct an admissible half-plane $f_3$ such that the image of the boundary line coincides with $\ell$. The image of $f_3$ contains a neighborhood of $p$ in ${D}$, which is a contradiction since $p$ is assumed to be a boundary point of the spectral polar domain.   

The third statement follows from the computation of the topological index of a closed simple loop around the pole of order $b_{i}$ (see Section~\ref{sub:windingnumber}).
\end{proof}

\begin{remark}
    In particular, this result implies that the spectral core can be determined by only knowing $\mathcal{W}^{(1)}$, since all boundaries are (portions of) critical trajectories.
\end{remark}
We can give a more precise description of the spectral core in terms of geodesic triangulations.

\begin{proposition}\label{prop:spectralDomain}
Consider a pair $(X,\varphi)$ such that $X$ is of genus $g$ and $\varphi$ is a meromorphic cubic differential with $p$ poles of orders $b_{1},\dots,b_{p} > 3$ and $n$ singularities of orders $a_{1},\dots,a_{n} \in \mathbb{N}^{\ast} \cup \lbrace{ -1,-2 \rbrace}$ (corresponding to conical singularities in the flat metric).
\par
Assuming that the spectral polar domains have a total of $\delta$ boundary edges that are saddle connections\footnote{Here, a saddle allowed to contribute twice if it has both sides belonging to such boundaries}, the spectral core $\mathrm{SCore}(X,\varphi)$ is the union of 
\begin{equation*}4g-4+2n+2p-\delta
\end{equation*}
Euclidean triangles with disjoint interiors.
\end{proposition}

\begin{proof}
We denote by $r$ the number of regular points that are corners of spectral polar domains. Observe that a given regular point can only appear in only one corner of spectral polar domain because each corner has an angle of $\frac{4\pi}{3}$. Thus, these regular points contribute to $2r \pi$ to the total angle of the decomposition of the surface into spectral core and spectral polar domains.
\par
By hypothesis, there is no triple pole so the spectral core is a flat surface of finite area with geodesic boundary and conical singularities (in the interior and on the boundary). Following~Lemma 2.2 of \cite{tahar1}, it can be endowed with a triangulation by Euclidean triangles whose vertices are the conical singularities. Each triangle has a total angle of $\pi$ while the total angle is the total angle of the conical singularities $\sum\limits_{i=1}^{n} (a_{i} +3)\frac{2\pi}{3}$, the total angle $2r \pi$ of the regular points. We also know that the total angle of the corners of the spectral polar domain of the pole of order $b_{j}$ is $(b_{j}-3)\frac{2\pi}{3} + \pi c_{j}$ where $c_{j}$ is the number of corners in the boundary of this spectral polar domain. It follows that the total angle of the spectral core is
$$
\sum\limits_{i=1}^{n} (a_{i}+3)\frac{2\pi}{3} +2r\pi
- \sum\limits_{j=1}^{p} (b_{j}-3)\frac{2\pi}{3} - \pi \sum c_{j}
= (4g-4+2n-2p+2r-\sum c_{j})\pi
.
$$
Since the total number of boundary corners of spectral polar domains is also the number of boundary edges, we have $\sum c_{j}=2r+\delta$ where $\delta$ is the number of boundary edges that are saddle connections (counted twice if it has both sides belonging to such boundaries). Then we deduce that the total number of triangles is $4g-4+2n+2p-\delta$.
\end{proof}

In particular, we observe that under the generic condition that a cubic differential has no real saddle connection (which implies $\delta=0$), the topological complexity of the spectral core depends only on the number of singularities (instead of their orders) and the genus of the underlying surface.

\section{BPS structures from spectral networks}
In this section we briefly recall several notions related to the theory of BPS structures. We refer to the original article \cite{bridgeland2019riemann} for further discussion.

\subsection{BPS structures}
The notion of BPS structure encodes data appearing both as the output of the analysis of BPS states in certain four-dimensional quantum field theories \cite{GAIOTTO2013239,gaiotto2012n}, as well as Donaldson-Thomas theory \cite{bridgeland2019riemann} applied to 3-Calabi-Yau triangulated categories. This was formalized in \cite{bridgeland2019riemann} as follows.

\begin{definition}
 A \emph{BPS structure} is a tuple $(\Gamma,Z,\Omega)$ of the following data:
\begin{itemize}
\item a free abelian group of finite rank $\Gamma$ equipped with an antisymmetric pairing 
\begin{equation}
\langle \cdot , \cdot \rangle : \Gamma \times \Gamma \rightarrow \mathbb{Z},
\end{equation} 
\item a homomorphism of abelian groups $Z \colon \Gamma \rightarrow \mathbb{C}$, and
\item a map of sets $\Omega: \Gamma \to {\mathbb Q}$,
\end{itemize}
satisfying the conditions
\begin{itemize}
\item \emph{Symmetry}: $\Omega(\gamma)=\Omega(-\gamma)$ for all $\gamma \in \Gamma$.
\item \emph{Support property}: for some (equivalently, any) choice of norm $\|\cdot\|$ on 
$\Gamma \otimes \mathbb{R}$, there is some $C>0$ such that
\begin{equation} \label{eq:support-property}
\Omega(\gamma)\neq  0 \implies |Z(\gamma)|>C\cdot \|\gamma\|.
\end{equation}
\end{itemize}
\end{definition}

We call $\Gamma$ the \emph{charge lattice}, and the homomorphism $Z$ is called the \emph{central charge}. 
The rational numbers $\Omega(\gamma)$ are called the \emph{BPS invariants}.

Some additional terminology is useful in discussing BPS structures. An \emph{active class} is an element $\gamma \in \Gamma$ which has 
nonzero BPS invariant, $\Omega(\gamma)\neq 0$, and a \emph{BPS ray} is a ray $Z(\gamma)\cdot\mathbb{R}_{+}\subset \mathbb{C}^*$ where $\gamma$ is active. We refer to the BPS structure as \emph{finite} if there are only finitely many active classes, and \emph{integral} if $\Omega$ is valued in $\mathbb{Z}\subset\mathbb{Q}$. The BPS structures obtained in this paper will turn out to always be finite, integral BPS structures.

To obtain a BPS structure via some natural geometric method (e.g. Donaldson-Thomas theory or Gaiotto-Moore-Neitzke's method) is in generally quite difficult. Nonetheless, many examples are known to exist \cite{bridgeland2015quadratic}, and many more are believed to (e.g. as in \cite{Gaiotto:2012rg,GAIOTTO2013239}). It remains a major task to deal with cases outside the setting of quadratic differentials, and this is precisely the case we begin to address in the rest of this paper.

Regardless of its origin, given a BPS structure, one can also ask further questions which are of interest. In particular, a certain infinite-dimensional Riemann-Hilbert problem has been formulated \cite{bridgeland2019riemann}, which turns out to be related to the recent notion of ``Joyce structures'' in hyperk\"ahler geometry. Such a space encodes the Donaldson-Thomas invariants into a single geometric structure.

\subsection{Wall-crossing and variation of BPS structures} \label{sec:varbps}

It is well known/expected that the famous wall-crossing phenomenon occurs as a stability condition or quadratic/higher differential is varied. The relationship between the BPS structure on either side of certain \emph{walls} takes the form of an identity, the Kontsevich-Soibelman \emph{wall-crossing formula}, in some appropriate algebra. We will consider sufficiently nice families of BPS structures varying exactly the manner described by this formula, and show that the BPS structure found in our main theorem indeed satisfies such properties. We fix the setting and notation here.

For a lattice $\Gamma$ with an antisymmetric pairing $\langle \cdot, \cdot \rangle$, 
we define the corresponding \emph{twisted torus} $\mathbb{T}_-$ as the set of twisted homomorphisms into $\mathbb{C}^\ast$:
\begin{align}
\mathbb{T}_- := \left\{\xi : \Gamma \rightarrow \mathbb{C}^\ast  \; | \; 
\xi{(\gamma_1+\gamma_2)}= (-1)^{\langle\gamma_1,\gamma_2\rangle}\xi(\gamma_1)\xi(\gamma_2)\right\}
\end{align}

We denote by $x_{\gamma} : {\mathbb T}_{-} \to {\mathbb C}^\ast$
the \emph{twisted character} naturally associated to $\gamma \in \Gamma$,
defined by ${x}_\gamma(\xi) = {\xi}(\gamma)$, which generate the coordinate ring of ${\mathbb T}_{-}$. The twisted characters satisfy
\begin{equation}\label{eq:ttorus}
x_\gamma x_{\gamma'} = (-1)^{\langle \gamma,\gamma' \rangle}x_{\gamma+\gamma'}.
\end{equation}

By \cite[Proposition 4.2]{bridgeland2019riemann}, we can define
the \emph{BPS automorphisms} for finite, integral BPS structures by the explicit birational automorphisms
${S}_\ell : {\mathbb T}_- \dashrightarrow {\mathbb T}_-$:
\begin{equation} \label{eq:BPS-auto}
 {S}_\ell^\ast ({x}_\beta) := {x}_\beta \, 
\prod_{\substack{\gamma \in \Gamma \\ Z(\gamma) \in \ell}} 
(1 - {x}_\gamma)^{\Omega(\gamma) \langle \gamma, \beta \rangle} 
\quad (\beta \in \Gamma).
\end{equation}
Here ${S}_\ell$ is given via its pullback ${S}_\ell^\ast$ 
acting on the coordinate ring of ${\mathbb T}_{-}$.

Often, one is interested in families of BPS structures, such as by varying the differential from which it is obtained. Such families should vary as nicely as possible, and satisfy the wall-crossing formula wherever any jumps in $\Omega$ might occur.
\begin{definition} \label{defn:varbps}
A \emph{variation of BPS structure} over a complex manifold $M$ 
consists of a family of BPS structures 
$(\Gamma_t, Z_t, \Omega_t)$ attached to each ${t} \in M$ 
such that:

\begin{enumerate}[i)]
\item The family $\Gamma_t$ forms a local system of lattices over $M$, and the pairing 
$\langle \cdot ,\cdot \rangle$ is covariantly constant,
 
\item Given a covariantly constant family of $\gamma \in \Gamma_t$, 
the central charges $Z_t(\gamma)$ are holomorphic functions on $M$,
\item The constant $C$ appearing in the support property \eqref{eq:support-property} 
may be chosen uniformly on compact subsets of $M$,
\item \label{item:wallcrossing} The \emph{wall-crossing formula} holds: for each acute sector 
$\leftslice\; \subset \mathbb{C}^{*}$, the counterclockwise product
\begin{equation}
{S}_\leftslice := \prod_{\ell \subset \leftslice} 
{{S}_\ell} 
\end{equation} 
is covariantly constant over $M$, whenever the boundary rays of $\leftslice$ are not BPS rays throughout.
\end{enumerate}
\end{definition}

For the examples of polynomial cubic differentials treated below, the first three properties are essentially obvious, and it will remain to prove the wall-crossing formula which is nontrivial at first for $d=3$.

Given a variation of BPS structures, there are several types of walls in its parameter space -- in the terminology of \cite{kontsevich2008stability} we will be interested in those of the ``first kind'': 
\begin{definition}Given a family of BPS structures $(\Gamma_t,Z_t,\Omega_t)$ over a complex manifold $M$, the \emph{walls of the first kind} are the real codimension-1 loci
\begin{equation*}
W= \{t\in M \,\, | \, \, \exists \, {\gamma_1,\gamma_2}\in\Gamma_t \textrm{ s.t. }Z_t(\gamma_1)/Z_t(\gamma_2)\in\mathbb{R}_{>0} \textrm{ with } \Omega_t(\gamma_1),\Omega_t(\gamma_2)\neq 0\}
\end{equation*}
That is, the loci at which the central charges $Z(\gamma_1)$, $Z(\gamma_2)$ of two active classes are proportional by a real constant.
\end{definition}

The key point in our setting is that walls of the first kind are where saddle connections and critical tripods may exist on one side but not the other.

\subsection{GMN construction}
Gaiotto-Moore-Neitzke provided a recipe to construct BPS structures from a tuples of $k$-differentials, although technical difficulties appear in general when the behaviour of networks becomes complicated. For the setting at hand, we can give the following simplified construction. We first lift to the cover the double trajectories of the spectral networks $\mathcal{W}_\vartheta (\varphi)$ for all $\vartheta\in[0,\frac{\pi}{3})$ as follows:
\begin{definition}
   Let $\Sigma^{\times}$ denote $\Sigma$ with the points at $\infty$ removed. The \emph{canonical lifts} of a saddle connection (resp. critical tripod) in the spectral network $\mathcal{W}_\vartheta(\varphi)$ are given by the following local lifting rule: we draw a curve around the saddle connection (resp. critical tripod) as indicated in Figure~\ref{fig:canonicallifts}, and take a closed path in $\Sigma^\times$ which projects to it\footnote{Such a path exists by lifting any initial point to an arbitrary preimage and following local lifts of the double trajectory.}. Acting by orientation reversal and the cyclic automorphism of $\Sigma$ permuting sheets, we obtain six homology classes $\pm\gamma^{(i)} \in H_1(\Sigma^\times,\mathbb{Z})$, $i=1,2,3$ referred to as \emph{saddle classes} (resp. \emph{tripod classes}) for the cubic differential $\varphi$.
\end{definition}

    \begin{figure}[h]
        \centering
        \begin{subfigure}{0.3\textwidth}
                \includegraphics[width=\linewidth]{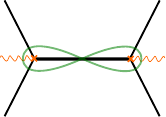}
                \vspace{0.5cm}
                \caption{Lifting a saddle}
                \label{fig:canonicalcyclegr}
        \end{subfigure}
        \hspace{1.5cm}
        \begin{subfigure}{0.37\textwidth}
        \includegraphics[width=\linewidth]{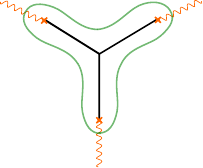}       
        \label{fig:canonicaltripod}
        \caption{Lifting a critical tripod}
        \end{subfigure}

        \caption{Canonical lifting}
        \label{fig:canonicallifts}
    \end{figure}
With the canonical lifts in hand, we can construct a BPS structure. Although Gaiotto-Moore-Neitzke provide a prescription for doing this in great generality assuming certain conditions on the networks involved, the numerical value for most kinds of degenerations has not been computed. Thus we restrict to the case that only saddles and critical tripods appear, which holds for all networks considered in this paper. In this case, the value of the physical ``BPS index'' was computed in \cite{Gaiotto:2012rg}, which we take as our definition of $\Omega$.

\begin{definition}(Restricted GMN construction) Let $\varphi$ be a polynomial cubic differential with spectral cover $(\Sigma,\pi,\lambda)$. The charge lattice is given by $\Gamma=H_1(\Sigma^\times,\mathbb{Z})$ equipped with the intersection pairing. 
    The central charge $Z(\gamma)$ of a homology class $\gamma\in \Gamma$ is given by
    \begin{equation}
        Z(\gamma):=\int_\gamma \lambda
    \end{equation}
Supposing furthermore that all double trajectories appearing in any $\mathcal{W}_\vartheta(\varphi)$ form either saddle connections or critical tripods, the BPS invariant $\Omega(\gamma)$ is given for any $\gamma\in\Gamma$ by
    \begin{equation}
        \Omega(\gamma)=\begin{cases}1, \qquad \gamma\text{ a saddle class or tripod class}, \\ 0, \qquad \text{otherwise.} \end{cases}
    \end{equation}
\end{definition}

By construction, the data $(\Gamma,Z,\Omega)$ then forms a BPS structure as long as it satisfies the support property. Below, we will see that there are finitely many active classes in all our examples, so that the support property is satisfied automatically since map giving the central charge is nondegenerate.

\begin{remark}
    Note that we could have used the so-called \emph{hat-homology} (see \cite{BCGGM3}) as the lattice instead; this would be in accordance with \cite{bridgeland2015quadratic} in the quadratic differential case, and appropriate for certain problems. However, this distinction plays no role here, so we take the full homology lattice for simplicity.
\end{remark}
The corresponding BPS structure for a related class of differentials was discussed in \cite{maruyoshi2013bps}; explicitly, they treat the $d=2$ case only.

\section{Spectral networks from polynomial cubic differentials}

In the rest of this paper we apply the above results to the specific case of cubic differentials of the form $\varphi = P(x)dx^{3}$ where $P(x)$ is a polynomial of degree $d$ defined on the Riemann sphere $\mathbf{P}^1$. We refer to these as \emph{polynomial cubic differentials}. They form a specific class of flat surfaces for which the directional dynamics is easier to describe. Since there in only one complex structure on $\mathbf{P}^1$, we will use the simplified notation $\mathrm{Core}(\varphi)$ and $\mathrm{SCore}(\varphi)$ for the (spectral) core of the flat structure induced by cubic differential $\varphi$.

\subsection{General results on polynomial cubic differentials}

\subsubsection{Geometric constraints of the behaviour of trajectories}\label{subsub:Constraints}

Applying the Gauss-Bonnet formula (see Lemma~\ref{lem:GaussBonnet}), we obtain strong constraints on the behaviour of trajectories for a flat structure induced by a polynomial cubic differential.

\begin{proposition}\label{prop:traj1}
For $\varphi$ a polynomial cubic differential, every trajectory is simple (in particular, the two ends of a saddle connection cannot coincide). Moreover, any maximal trajectory of $\varphi$ either hits a zero in finite time or converges to the pole in infinite time.
\end{proposition}

\begin{proof}
If a trajectory intersects itself, then up to taking a subset of the trajectory, we can assume that it is a topological simple loop formed by a simple geodesic segment.
\par
Such a loop cuts out $\mathbf{P}^1$ into two connected components, one of which does not contain the pole at infinity. This domain has an interior angle $\theta$ at the unique corner on its boundary, and the only interior singularities are some number $n$ zeros of multiplicities $a_{1},\dots,a_{n}$, so by the Gauss-Bonnet Lemma~\ref{lem:GaussBonnet} we have
$$
\theta = -\pi - 2\pi \sum\limits_{i=1}^{n} a_{i}.
$$
which violates the positivity of $\theta$. Thus, there is no self-intersecting trajectory. For the same reason, we can rule out simple periodic geodesics (corresponding to the case where the corner has an angle $\theta$ equal to $\pi$).
\par
Now we consider a maximal trajectory $\gamma$. Suppose that $\gamma$ does not hit a zero in finite time or converges to the pole in infinite time, and consider the spectral cover $({\Sigma},\lambda)$, see Section~\ref{sub:spectralCover}. Since $({\Sigma},\lambda)$ is a translation surface, any lift $\tilde{\gamma}$ of $\gamma$ is a maximal trajectory in $({\Sigma},\lambda)$. It follows from the classification of trajectories in translation surfaces (Proposition 5.5 in \cite{tahar}) that if $\tilde{\gamma}$ does not hit a conical singularity and converges to a pole, $\tilde{\gamma}$ is either periodic or is dense in some domain. In the latter case, the domain contains a simple closed geodesic (see Lemma 5.12 in \cite{tahar}). Consequently, in any case $({\Sigma},\lambda)$ contains a closed geodesic whose projection on $(\mathbf{P}^1,\varphi)$ is either a self-intersecting trajectory or a closed geodesic. In both cases we get a contradiction with what precedes.
\end{proof}

A similar argument shows that no pair of distinct trajectories intersect more than once.

\begin{proposition}\label{prop:traj2}
Let $\varphi$ be a polynomial cubic differential, and consider a pair of trajectories $\gamma,\gamma'$ that do not coincide on any interval. Then $\gamma$ and $\gamma'$ intersect at most once\footnote{We consider a common endpoint of $\gamma,\gamma'$ an intersection point unless it is the pole at infinity.}.
\end{proposition}


\subsubsection{Upper bounds on the number of saddle connections and critical tripods}

We deduce from the results of Section~\ref{subsub:Constraints} an upper bound on the number of saddle connections of a polynomial cubic differential.

\begin{corollary}\label{cor:UpperBound}
If $\varphi$ is a polynomial cubic differential of degree $d$, then $(\mathbf{P}^1,\varphi)$ has no closed saddle connection and at most one saddle connection between each pair of distinct zeros.
\par
In particular, $(\mathbf{P}^1,\varphi)$ has at most $\frac{d(d-1)}{2}$ (unoriented) saddle connections.
\end{corollary}

\begin{proof}
We deduce from Proposition~\ref{prop:traj1} that the two endpoints of a saddle connection cannot coincide. Proposition~\ref{prop:traj2} proves that two distinct saddle connection cannot share the same pair of endpoints. Then the number of saddle connections in $(\mathbf{P}^1,\varphi)$ is bounded by the number of pairs of zeros of $\varphi$ which is $\frac{d(d-1)}{2}$ if the zeros are simple.
\end{proof}

\begin{remark}
For any $d$, $(z^{d}-1)dz^{\otimes 3}$ defines a flat surface with a symmetry of order $d$. The only possible shape of $\mathrm{Core}(\varphi)$ is a regular $d$-gon and $\frac{d(d-1)}{2}$ is the number of sides and diagonals of the $d$-gon. In other words, the upper bound of Corollary~\ref{cor:UpperBound} is optimal.
\end{remark}

A similar bound holds for the counting of tripods. In order to prove it, first prove that two tripods joining the same triple of zeros define the same homology class.

\begin{lemma}\label{lem:tripodHomology}
If $\varphi$ is a polynomial cubic differential of degree $d$, then the starting points of the three trajectories of any tripod of $(\mathbf{P}^1,\varphi)$ are distinct. Furthermore, if two tripods $\downY$ and $\downY'$ connect the same cyclically oriented triple of starting points, then they are homotopic.
\par
In particular, if $\downY$ and $\downY'$ are critical tripods of $\varphi$ with the same cyclically oriented starting points then $\downY$ and $\downY'$ define the same homology classes in $H_1(\Sigma^\times,\mathbb{Z})$.
\end{lemma}

\begin{proof}
We deduce from Proposition~\ref{prop:traj2} that two trajectories of the same tripod cannot start from the same point. Now we assume by contradiction that there exist two tripods that connect the same oriented triple of distinct points $a,b,c$ of orders $k_{a},k_{b},k_{c}$ (if any of these points is not a zero of $\varphi$, then its order is equal to $0$). We refer to the central points of tripods $\gamma$ and $\gamma'$ as $o$ and $o'$. Points $o$ and $o'$ are distinct otherwise because since there cannot be two distinct segments joining the same endpoints, the tripods would then coincide. We consider arcs of $\downY$ and $\downY'$ as follows: $\alpha_1$ connecting $a$ to $o$, $\alpha_2$ connecting $o$ to $b$, and $\alpha=\alpha_1\cup\alpha_2$ (see Figure \ref{fig:tripodproofa}); likewise, $\beta_1$, $\beta_2$, and $\beta$ are defined similarly for $\downY'$. We want to prove that $\alpha$ and $\beta$ are homotopic (i.e. that they cut out a topological disk that does not contain any singularity).

    \begin{figure}[h]
        \centering
                \includegraphics[width=0.30\linewidth]{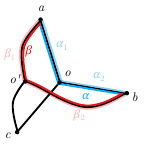}

        \caption{Proof of Lemma \ref{lem:tripodHomology}}
        \label{fig:tripodproofa}
    \end{figure}

We know from Proposition~\ref{prop:traj2} that two segments of the same tripod intersect at most once. This implies that $\alpha$ and $\beta$ intersect at most twice outside $a$ and $b$; in particular, $\alpha_1$ and $\beta_2$ could potentially intersect, and $\alpha_2$ and $\beta_1$ could potentially intersect. We will see that in fact at most one of these holds. Indeed, if both hold, denote by $i$ the intersection of $\alpha_1$ and $\beta_2$, and $j$ the intersection of $\alpha_2$ and $\beta_1$. By construction, there is a triangle $\Delta_{ajo}$ that does not contain $b$ and $o'$. We list the possible intersections between $\beta_2$ and the boundary of triangle $\Delta_{ajo}$ with vertices at $a,j$, and $o$:
\begin{itemize}
    \item $\beta_2$ cannot intersect $\alpha_2$ outside of $b$ because two trajectories intersect at most once so $\beta_2$ cannot intersect $[oj]\subset\alpha_2$;
    \item $\beta_2$ cannot intersect $\beta_1$ outside of $o'$ for the same reason and therefore cannot intersect $[aj]\subset \beta_1$;
    \item $\beta_2$ intersects $\alpha_1$ at $i$ and this is their unique intersection.
\end{itemize}
In other words, the only intersection between $\beta_2$ and the boundary of triangle $\Delta_{ajo}$ is $i$. Since $b$ and $o'$ do not belong to $\Delta_{ajo}$, the number of crossings between $\beta_2$ and the boundary of $\Delta_{ajo}$ should be even. We get a contradiction. It follows that $\alpha$ and $\beta$ intersect at most once outside $a$ and $b$.
\par
Then, there are two cases depending whether $\alpha$ and $\beta$ have any such intersection point or not. We first assume there is; without loss of generality, we assume that this intersection point again called $i$ is the intersection outside $a$ and $b$ of $\alpha_1$ and $\beta_2$. Then $\alpha$ and $\beta$ form a figure-eight curve in $\mathbf{P}^1$. Each loop of the figure-eight is formed by three straight segments: respectively $\beta_1$, $\beta_2^{-}$, $\alpha_1^{+}$ and $\beta_2^{+}$, $\alpha_1^{-}$, $\alpha_2$. Each of these loop defines a triangle with three corners, one of which (at $o$ and $o'$) has an angle of magnitude at least $\frac{2\pi}{3}$. Since these triangles have a total angle that is strictly larger than $\frac{2\pi}{3}$, Lemma~\ref{lem:GaussBonnet} proves then that the interior any of these triangles either contains the pole at infinity or does not contain any zero. If these two triangles do not contain any singularity, then this amounts to say that $\alpha$ and $\beta$ are isotopic to each other. If one of these triangles, say $\Delta_{aio'}$ contains the pole at infinity, then we obtain a contradiction using Lemma~\ref{lem:GaussBonnet}. Indeed, in the worst case, the singularities in the interior of $\Delta_{aio'}$ are the pole and every zero except $a$ and $b$. Then the total order of the singularities in the interior of $aio'$ is at worst $-6-k_{a}-k_{b}$. This implies that the sum of the angles at the three corners is at least $\frac{(3+2(6+k_{a}+k_{b}))\pi}{3}=\frac{(15+2k_{a}+2k_{b})\pi}{3}$. On the other hand, the sum of the angles at the three corners is strictly smaller than $\frac{4\pi}{3}+\pi+\frac{(6+2k_{a})\pi}{3}=\frac{(13+k_{a})\pi}{3}$. This is a contradiction. Therefore, in the cases where $\alpha$ and $\beta$ intersect outside $a$ and $b$, they are homotopic to each other.
\par
If $\alpha$ and $\beta$ have no intersection point outside of $a$ and $b$, then $\alpha \cup \beta$ cuts out two quadrilaterals with corners $a,o,b,o'$, one of which does not contain the pole at infinity in the interior. We denote by $\mathcal{D}$ this quadrilateral. Since the angles at $o$ and $o'$ are either equal to $\frac{2\pi}{3}$ or $\frac{4\pi}{3}$, the total angle at the corners of $\mathcal{D}$ is strictly larger than $\frac{4\pi}{3}$. It follows from Lemma~\ref{lem:GaussBonnet} that the total order of the singularities in the interior of $\mathcal{D}$ cannot be positive (even a simple zero would lead to a total angle smaller or equal to $\frac{4\pi}{3}$). It follows that $\mathcal{D}$ does not contain any singularity of $\varphi$ and defines an homotopy between $\alpha$ and $\beta$.
\par
We have proved that $\alpha_1\cup\alpha_2$ and $\beta_1\cup\beta_2$ are isotopic to each other; the same reasoning holds for the other two pairs of legs of the tripods. Since the cyclic ordering of the legs of respectively $\downY$ and $\downY'$ around respectively $o$ and $o'$ are the same, this provides an homotopy between tripods $\downY$ and $\downY'$ (with fixed starting points of the trajectories).
\par
When $\downY$ and $\downY'$ are critical tripods (formed by critical trajectories), this implies that they define the same homology class in the homology of the spectral cover.
\end{proof}



Then, we prove the bound.

\begin{corollary}\label{cor:UpperBoundTripod}
If $\varphi$ is a polynomial cubic differential of degree $d$, then $\varphi$ has at most $\frac{d(d-1)(d-2)}{6}$ (unoriented) critical tripods.
\end{corollary}

\begin{proof}
We just have to prove that if two tripods $\downY$ and $\downY'$ are drawn between the same oriented triple of distinct zeros, then they are identical. Since there are $\frac{d(d-1)(d-2)}{3}$ such triples, the bound follows.
\par
For such a pair of tripods, Lemma~\ref{lem:tripodHomology} implies $Z([\downY])=Z([\downY'])$. It follows that their trajectories belong to same direction of $\mathbb{R}/\frac{\pi}{3}\mathbb{Z}$.

Suppose that $\downY$ and $\downY'$ intersect outside $a,b,c$. Following the same notation as in the proof of Lemma~\ref{lem:tripodHomology}, we assume without loss of generality that $\alpha_1$ intersects $\beta_2$ in some point $i$. 
Segments $\alpha_1^{+}$, $\beta_2^{-}$, and $\beta_1$ 
form a simple loop that cut out $\mathbf{P}^1$ into two connected components. Denoting by $\mathcal{D}$ the component that does not contain the pole of $\varphi$, we first observe that the total angle at the three corners is at least $\frac{4\pi}{3}$ (angles are integer multiples of $\frac{\pi}{3}$ and the angle at $o'$ is either $\frac{2\pi}{3}$ or $\frac{4\pi}{3}$). Following Lemma~\ref{lem:GaussBonnet}, this implies that the interior of $\mathcal{D}$ contains poles, which is a contradiction, so that $\downY$ and $\downY'$ do not intersect outside $a,b,c$. 
\par
There is only one possible construction for two such tripods where they decompose cut out $\mathbf{P}^1$ into three quadrilaterals and we observe that the orientation of $a,b,c$ for the two tripods are automatically opposite. Therefore, $\downY$ and $\downY'$ coincide.
\end{proof}

\begin{remark}
For any $d$, $(x^{d}-1)dx^{\otimes 3}$ defines a flat surface whose core is a regular $d$-gon. For each triangle drawn between three corners of the $d$-gon whose angles are strictly smaller than $\frac{2\pi}{3}$, we draw a tripod by joining the three corners with the Fermat point of the triangle. The number of such triangles grows as a polynomial of degree three in $d$ so the bound of Corollary~\ref{cor:UpperBoundTripod} is essentially sharp.
\end{remark}

\begin{remark}
Note that Corollary~\ref{cor:UpperBoundTripod} does not give any information about noncritical tripods (those that are formed by trajectories that are not starting at some zero). Such tripods do not exist for $d \leq 3$ but since it is not clear whether for large $d$ the spectral network is finite or not, it is not clear either whether we can bound the number of tripods formed by trajectories starting at arbitrary intersection points.
\end{remark}



\subsection{Example: Spectral networks in the case $d=0,1$}
When $d=0$ there are no zeros and therefore no spectral network.

In the case $d=1$, there is a single zero, from which $8$ critical trajectories emanate which, following Propositions~\ref{prop:traj1} and~\ref{prop:traj2}, do not intersect each other or themselves. It follows that these eight (maximal) trajectories converge to the pole at infinity, cutting $\mathbf{P}^1$ into $8$ infinite cones of angle $\frac{\pi}{3}$.

\begin{figure}[h]
    \centering
    \includegraphics[width=0.4\linewidth]{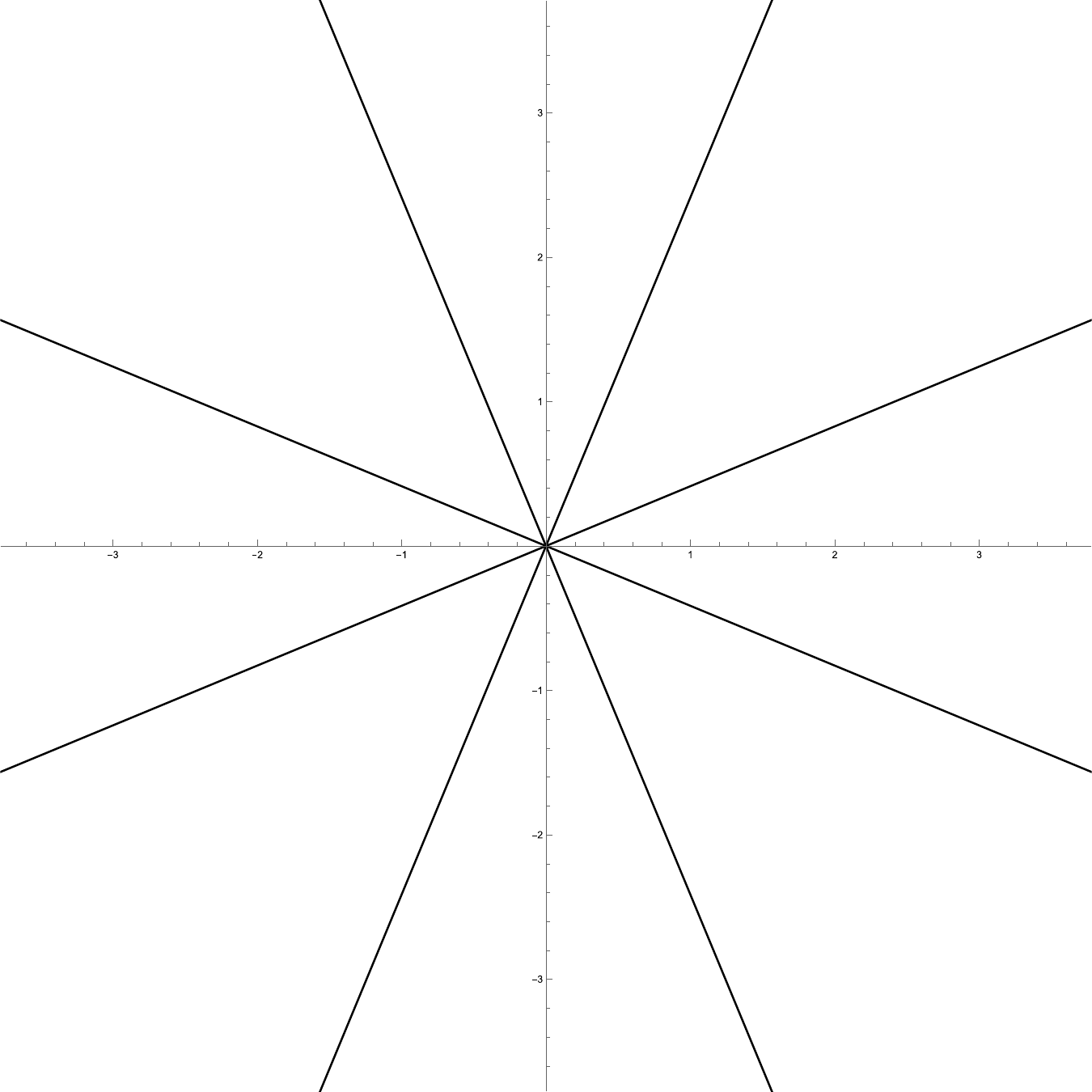}
    \caption{The $d=1$ case: $\varphi=xdx^3$ at $\vartheta\approx0.53$}
    \label{fig:degreeone}
\end{figure}
By the normal form theorem, this network (depicted in Figure \ref{fig:degreeone}) provides a local model for the network near any simple zero of $\varphi$.

\subsection{Example: Spectral networks in the case $d=2$}

Up to a scaling and an affine change of variable, every polynomial cubic differential of degree $d=2$ with simple zeros is of the form $\alpha(x^{2}-1)dx^{\otimes 3}$. The homotopy class of the arcs joining the two zeros has a geodesic representative which is a saddle connection. Lemma~\ref{lem:coreTriangle} implies that for any such differential $\varphi$, $\mathrm{Core}(\varphi)$ coincides with this saddle connection and there is no other saddle connection. We check immediately that this saddle connection is $[-1,1]$.

\subsubsection{Geometry of the spectral core}

Depending on the value of $\arg \alpha$, there are two possible shapes for $\mathrm{SCore}(\varphi)$.

\begin{proposition}\label{prop:d2Spectral}
The spectral core induced by differential $\varphi=\alpha(x^{2}-1)dx^{\otimes 3}$ is:
\begin{itemize}
    \item a parallelogram cut out by positive and negative trajectories where $[-1,1]$ is a diagonal and where the angles at $-1$ and $1$ are equal to $\frac{\pi}{3}$ if $\arg\alpha\notin\frac{\pi}{3}\mathbb{Z}$;
    \item the unique real saddle connection $[-1,1]$ if $\arg\alpha\in \frac{\pi}{3}\mathbb{Z}$.

\end{itemize}
\end{proposition}

\begin{proof}
The saddle connection $[-1,1]$ is a real trajectory if and only if $\arg\alpha\in \frac{\pi}{3}\mathbb{Z}$. Both sides of the saddle connection belong to the image of the boundary of a half-plane under a locally isometric immersion. Since the saddle connection is real, these half-planes are admissible (see Definition~\ref{defn:admissible}) and therefore the spectral polar domain coincides with the usual polar domain (see Section~\ref{sub:core}).
\par
For any other value of $\arg\alpha$, there is no real saddle connection so Proposition~\ref{prop:SPECpolarBoundary} proves that $\mathrm{SCore}(\varphi)$ is formed by two triangles and that these two triangles belong to the same connected component of the interior of $\mathrm{SCore}(\varphi)$. It follows that $\mathrm{SCore}(\varphi)$ is parallelogram. Since $\mathrm{Core}(\varphi) \subset \mathrm{SCore}(\varphi)$, saddle connection $[-1,1]$ is a diagonal of the parallelogram. The claim about the angles also follows from Proposition~\ref{prop:SPECpolarBoundary}.
\end{proof}

\subsubsection{Description of the spectral networks}\label{sub:d2Results}

Following Theorem~\ref{thm:spectralcore}, the spectral network can be deduced from its restriction to the spectral core.

\begin{proposition}\label{prop:d2network}
Consider the polynomial cubic differential $\varphi=\alpha(x^{2}-1)dx^{\otimes 3}$ of degree $d=2$. At the phase $\vartheta$, the spectral network $\mathcal{W}_\vartheta(\varphi)$ is

    \begin{figure}[h]
        \centering

               \begin{subfigure}{0.4\textwidth}   \includegraphics[width=\linewidth,trim = 3cm 3cm 3cm 3cm, clip]{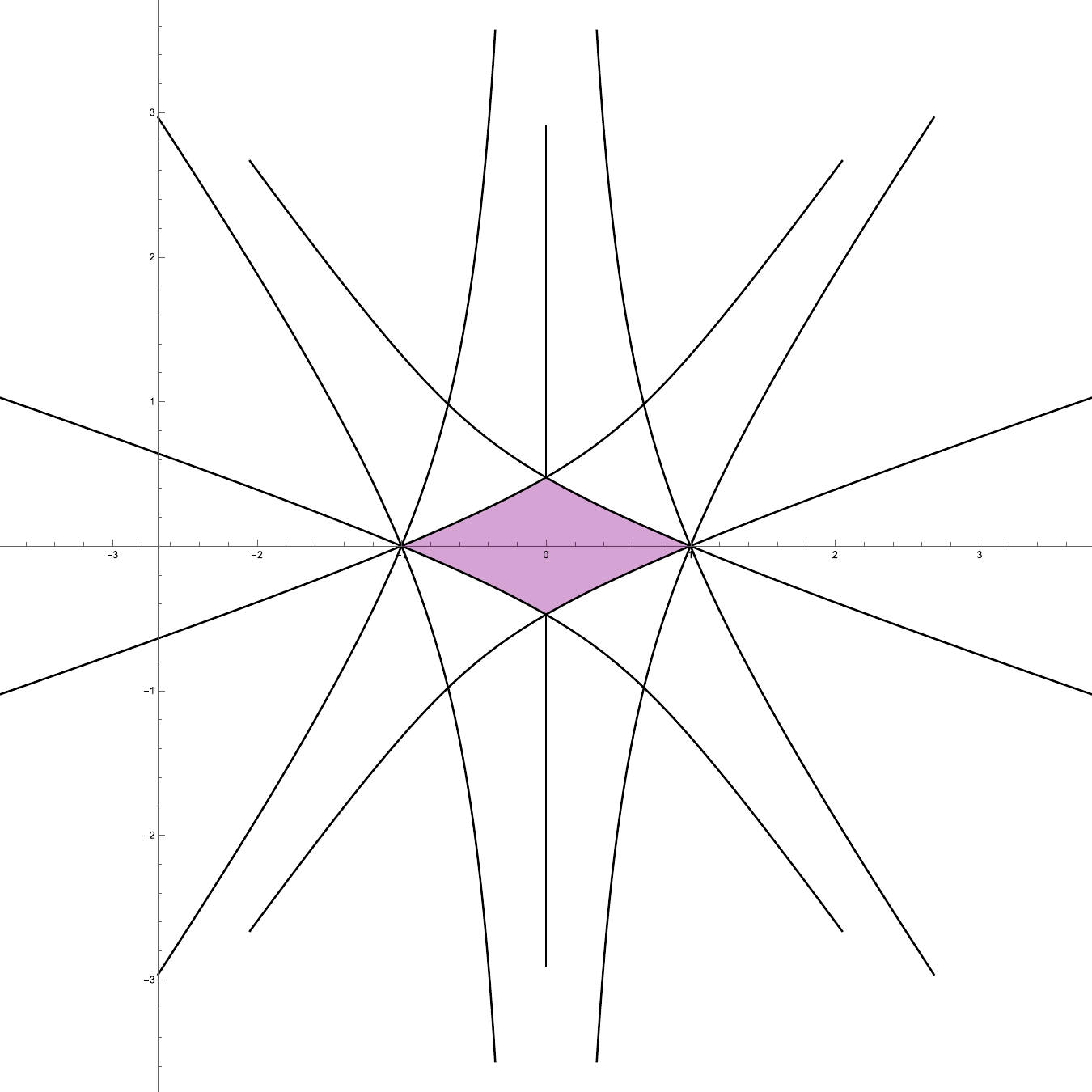}
                \caption{Nondegenerate case with $\mathrm{SCore}(\varphi)$ shaded, $\vartheta\approx\pi/6$}
                \label{fig:d2pica}
        \end{subfigure}
         \hspace{0.5cm}
        \begin{subfigure}{0.4\textwidth}
        \includegraphics[width=\linewidth,trim = 3cm 3cm 3cm 3cm, clip]{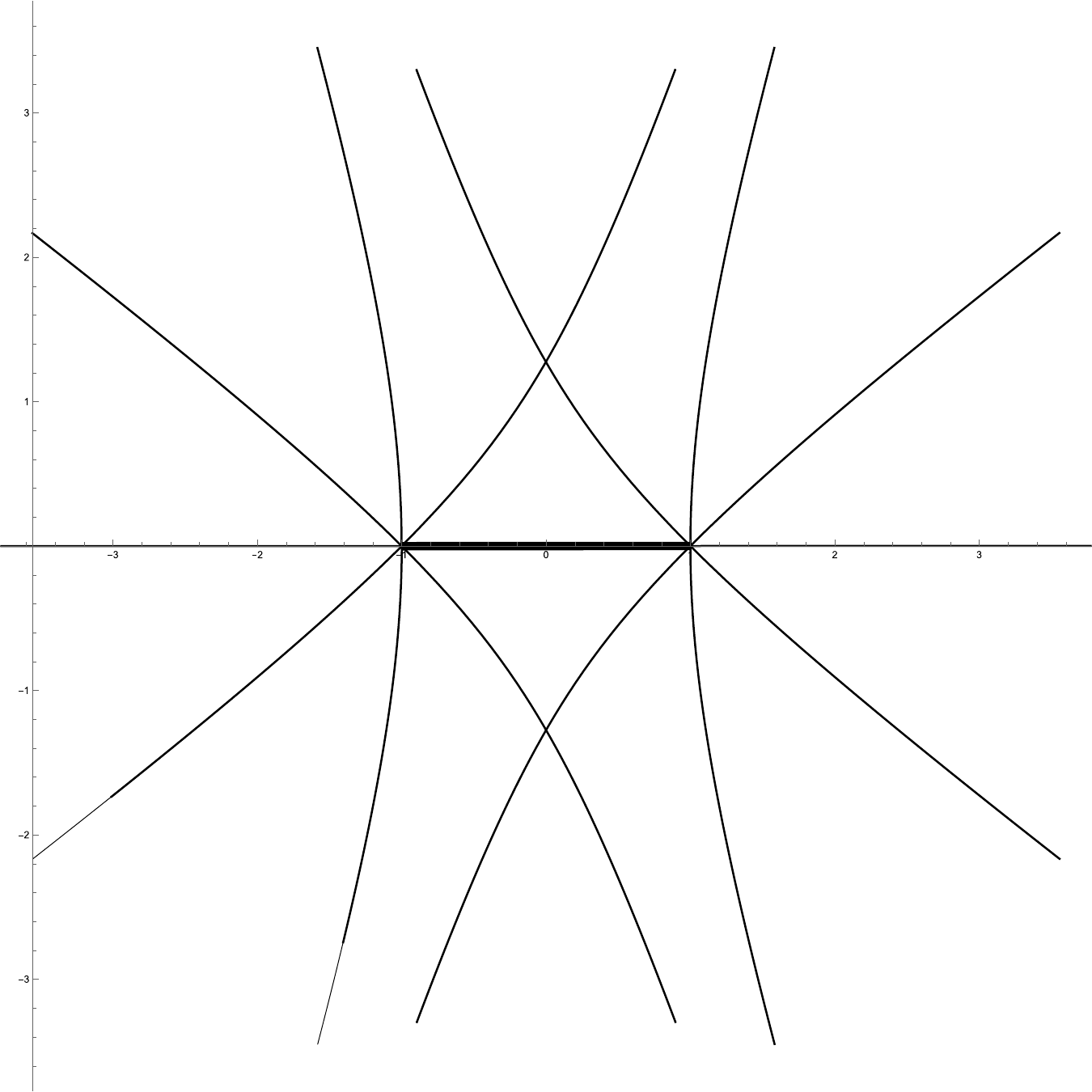}            \caption{Degenerate case with $\mathrm{SCore}(\varphi)=\mathrm{Core}(\varphi)$ a saddle (thickened), $\vartheta\approx0.00$}
        \label{fig:d2picb}
        \end{subfigure}
 
        \caption{Possible networks and spectral cores for $d=2$}
        \label{fig:enter-label}
    \end{figure}
    
    \begin{itemize}

        \item Non-degenerate with 18 trajectories, two of which arise from intersections, when $\vartheta\notin\frac{\pi}{3}\mathbb{Z}$. The combinatorial structure is given by Figure \ref{fig:d2pica}.
        
            \item Degenerate with 16 trajectories, with two forming a saddle connection 
            when $\vartheta=\frac{k \pi}{3}$. The combinatorial structure is given by Figure \ref{fig:d2picb}.

    \end{itemize}
\end{proposition}

\begin{proof}
In both cases, the spectral core is given by Proposition~\ref{prop:d2Spectral}. We first describe the spectral network in the spectral core, then extend the trajectories to the images of the maximal admissible half-planes (see Section~\ref{sub:SpectralCore}). Let us refer to the conical singularities at $-1$ and $1$ as $a$ and $b$, respectively.
\par
In the nondegenerate case, the spectral core is a parallelogram with two opposite angles equal to $\frac{\pi}{3}$ at the two conical singularities and two opposite angles equal to $\frac{2\pi}{3}$ at regular points $i,j$. It follows that $\mathcal{W}(\varphi)$ contains $8$ critical trajectories starting from each of $a$ and $b$ and $2$ trajectories starting at these two corners $i,j$ of the parallelogram because two critical trajectories of same sign intersect there. These two trajectories are necessarily oriented in such a way they do not enter the interior of the parallelogram.
\par
Then, among the $2d+6=10$ maximal admissible half-planes (in this case they are embeddings) that cover the spectral polar domain, we have four whose boundary contains a boundary edge of the spectral core, and six whose boundary contains only a zero.

In the first case, these open half-planes contain four trajectories:
one positive and one negative trajectory starting from a (single) zero, and one positive and one negative trajectory starting from a (single) corner $i$ or $j$. The only possible intersection is between two trajectories of different sign starting from different points. Therefore, no new trajectory appears (as expected from Theorem~\ref{thm:spectralcore}).

\begin{figure}[h]
\centering
\vspace{0.35cm}
\includegraphics[width=0.5\linewidth]{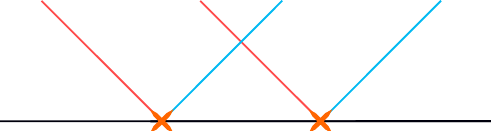}
\vspace{0.2cm}
\caption{Behaviour of critical trajectories on the interior of a half-plane with two conical singularities on the boundary}
\label{fig:halfplanetwsings}
\end{figure}

In the second case, these open half-planes contains two critical trajectories starting from the zero and possibly part of a trajectory starting from the other zero.
\par
Patching these local descriptions, keeping in mind that two consecutive embedded half-planes overlap on an infinite cone of angle $\frac{2\pi}{3}$ cut out by trajectories of the network, we obtain the spectral network $\mathcal{W}_\vartheta(\varphi)$. 
\par
In case b), the spectral core is a saddle connection so among the $2d+6=10$ images of maximal admissible immersions of half-planes (in this case they are embeddings) that cover the spectral polar domain, we have:
\begin{itemize}
    \item two whose boundary contains a boundary edge of the spectral core;
    \item eight whose boundary contains only a zero.
\end{itemize}
The reasoning is the same as in case a) to obtain the spectral network.
\end{proof}

\subsubsection{Generalization to higher order zeros}\label{subsub:HigherDegree}

Actually, the description of the spectral networks for polynomial cubic differentials of degree $d=2$ given above generalizes to any polynomial cubic differential $\varphi=\alpha (x+1)^{a_{1}}(x-1)^{a_{2}}dx^{\otimes 3}$ with two (possibly multiple) zeros. 
\par
Indeed, Proposition~\ref{prop:SPECpolarBoundary} proves that $\mathrm{SCore}(\varphi)$ is either a saddle connection or a parallelogram exactly in the same way as in Proposition~\ref{prop:d2Spectral}. Similarly, $\varphi$ has exactly one saddle connection. Then we obtain the spectral network $\mathcal{W}_\vartheta(\varphi)$ from a spectral network described in Section~\ref{sub:d2Results} as follows. We consider the spectral network of $\alpha (x+1)(x-1)dx^{\otimes 3}$ and consider a pair of real critical trajectories starting from $-1$ and $1$ such that they do not intersect any other trajectory of the network. Then, we graft along these trajectories a flat cone of angle $\frac{(2a_{1}-2)\pi}{3}$ and $\frac{(2a_{1}-2)\pi}{3}$ respectively. On each of these cones, we draw the additional real critical trajectories.

\section{Application: spectral networks for polynomial cubic differentials of degree $d=3$}

The main nontrivial example dealt with in this paper, in which most of the general features can be seen, are polynomial cubic differentials of degree $d=3$ with simple (distinct) zeros.

\subsection{Parametrization}\label{sub:d3Parametrization}

Any polynomial cubic differential of degree $d=3$ with simple zeros can we written (up to biholomorphic change of variable) as:
$$
\frac{ \alpha \, x (x-1)dx^{\otimes 3}}{(x-t)^{9}}
$$
where $\alpha \in \mathbb{C}^{\ast}$ and $t \in \mathbb{C} \setminus \lbrace{ 0,1 \rbrace}$.

We identify the group of biholomorphic automorphisms of $\mathbf{P}^1$ permuting $0$, $1$ and $\infty$ with the symmetric group $S_{3}$. This group is generated by $t \mapsto 1-t$ and $t \mapsto \frac{1}{t}$.

The quotient $\mathcal{T}$ of $\mathbf{P}^1$ by $S_{3}$ has:
\begin{itemize}
    \item an orbifold point of order $2$ corresponding to the orbit $\lbrace{0,1,\infty \rbrace}$;
    \item an orbifold point of order $2$ corresponding to the orbit $\lbrace{-1,\frac{1}{2},2 \rbrace}$;
    \item an orbifold point of order $3$ corresponding to the orbit $\lbrace{ e^{\frac{i\pi}{3}},e^{-\frac{i\pi}{3}} \rbrace}$.
\end{itemize}

A  {fundamental domain} $T$ of the action of $S_{3}$ on the parameter space of $t$ is given by the hyperbolic triangle drawn between $0$, $1$ and $e^{\frac{i\pi}{3}}$ whose boundary arcs are:
\begin{itemize}
    \item the straight segment $[0,\frac{1}{2}]$ (identified with $[\frac{1}{2},1]$);
    \item the circular arc from $1$ to $e^{\frac{i\pi}{3}}$ of center $0$ (identified with the circular arc from $0$ to $e^{\frac{i\pi}{3}}$ of center $1$).
\end{itemize}

\subsection{Geometry of the core}\label{sub:d3Core}

The possible cores of $\varphi$ are fairly simple to classify for polynomial cubic differentials of degree $d=3$. Depending on the shape of the core, $\varphi$ will have either two or three saddle connections.

\begin{lemma}\label{lem:d3dichotomy}
For a given cubic differential $\varphi=\frac{\alpha x (x-1)dx^{\otimes 3}}{(x-t)^{9}}$, exactly one of the following statements holds:
\begin{itemize}
    \item $\mathrm{Core}(\varphi)$ is the union of two of saddle connections, each joining some (distinct) pair of $0$, $1$ and $\infty$ and the angle between them (measured from either side) lies in $[\pi,\frac{5\pi}{3}]$;
    \item $\mathrm{Core}(\varphi)$ is an Euclidean triangle bounded by three saddle connections joining $0$, $1$ and $\infty$.
\end{itemize}
Since every saddle connection of a flat surface belongs to its core, $\varphi$ has either two or three saddle connections.
\end{lemma}

\begin{proof}
The differential $\varphi$ has three zeros and one pole. Denoting by $\beta$ the number of sides of the unique polar domain (see Section~\ref{sub:core}), Lemma~\ref{lem:coreTriangle} shows that the interior of $\mathrm{Core}(\varphi)$ is formed by $4-\beta$ triangles with disjoint interiors.
\par
If $\beta=4$, then $\mathrm{Core}(\varphi)$ has empty interior. It is the union of several saddle connections. Each of their two sides is a boundary side of the polar domain. Thus there are exactly two saddle connections joining the three zeros. At their common zero, the angles are of magnitude at least $\pi$ (for both corners drawn by the two segments at their common end) otherwise $\mathrm{Core}(\varphi)$ which is defined by a convexity property would have nonempty interior. Since each zero defines a conical singularity of angle $\frac{8\pi}{3}$, the claim follows.
\par
If $\beta=3$, then the interior of $\mathrm{Core}(\varphi)$ is a triangle and each of its boundary side is a boundary side of the unique polar domain.
\par
The case $\beta=2$ cannot appear because then the boundary of the polar domain would contain at most two conical singularities. The third conical singularity would be contained in the interior of $\mathrm{Core}(\varphi)$. Since the interior of $\mathrm{Core}(\varphi)$ is formed by two triangles, its total angle is $2\pi$ which is smaller than the conical angle $\frac{8\pi}{3}$ of this interior singularity.
\par
The cases $\beta=1$ and $\beta=0$ cannot appear for the same reasons because the interior of $\mathrm{Core}(\varphi)$ would contain therefore respectively two and three conical singularities of angle $\frac{8\pi}{3}$ each while having a total angle of $3\pi$ and $4\pi$ respectively.
\end{proof}

\subsection{Geometry of the spectral core}\label{sub:d3SpectralCore}
The geometry of the core depends only on the parameter $t$. In contrast, the geometry of the spectral core depends on both $t$ and $\alpha$ (or $\vartheta$). We classify all possible spectral cores in what follows. First we treat the case in which there are no saddles on its boundary:

\begin{lemma}\label{lem:d3spectral}
For a given cubic differential $\varphi=\frac{\alpha x (x-1)dx^{\otimes 3}}{(x-t)^{9}}$, either the boundary of the unique polar domain contains one or more saddle connections, or the spectral core is given by exactly one of the following:
\begin{itemize}
    \item Type $\mathrm{I}$: $\mathrm{SCore}(\varphi)$ is a hexagon with its six interior angles equal to $\frac{2\pi}{3}$;
    \item Type $\mathrm{II}^{-}$: $\mathrm{SCore}(\varphi)$ is a hexagon with interior angles $\pi,\frac{2\pi}{3},\frac{\pi}{3},\frac{2\pi}{3},\frac{2\pi}{3},\frac{2\pi}{3}$;
    \item Type $\mathrm{II}^{+}$: $\mathrm{SCore}(\varphi)$ is a hexagon with interior angles $\pi,\frac{2\pi}{3},\frac{2\pi}{3},\frac{2\pi}{3},\frac{\pi}{3},\frac{2\pi}{3}$;
    \item Type $\mathrm{III}$: $\mathrm{SCore}(\varphi)$ is a hexagon with interior angles $\frac{4\pi}{3},\frac{2\pi}{3},\frac{\pi}{3},\frac{2\pi}{3},\frac{\pi}{3},\frac{2\pi}{3}$;
    \item Type $\mathrm{IV}$: $\mathrm{SCore}(\varphi)$ is a pair of parallelograms with disjoint interiors and a unique common vertex which is a zero of $\varphi$ at which the angles inside the parallelograms are equal to $\frac{\pi}{3}$ while the angles inside the spectral polar domain are equal to $\pi$.
\end{itemize}
where interior angles are listed in counterclockwise order.
\end{lemma}

\begin{proof}
If the boundary of the (unique) spectral domain contains a saddle connection, the first statement holds. We will assume that there is no such real saddle connection. Then Proposition~\ref{prop:spectralDomain} proves that $\mathrm{SCore}(\varphi)$ is formed by $4$ triangles. Since every connected component of the interior of $\mathrm{SCore}(\varphi)$ is bounded by an even number of sides (this also follows from Proposition~\ref{prop:spectralDomain}), either the interior of $\mathrm{SCore}(\varphi)$ is connected or it is formed by two components each being formed by two triangles.
\par
In this second case, each of these components is a quadrilateral with a pair of opposite corners being regular points with an angle of $\frac{2\pi}{3}$. Since angles are integer multiples of $\frac{\pi}{3}$, it follows that the other pair of opposite corners (at the conical singularities of the metric) are actually equal to $\frac{\pi}{3}$. The two parallelograms share a vertex which is a conical singularity. It remains to determine the opposite angles at this vertex inside the spectral polar domain. We deduce from what precedes that they are integer multiples of $\frac{\pi}{3}$ and that they sum to $2\pi$. By definition of a spectral polar domain, these angles cannot be smaller than $\pi$ so they are equal to $\pi$. This corresponds to type $\mathrm{IV}$.
\par
In the remaining cases, $\mathrm{SCore}(\varphi)$ is a hexagon formed by four Euclidean triangles. In its boundary there is alternation of regular and singular points. At the regular points, the angle of the corner is $\frac{2\pi}{3}$. Then, we have to share $2\pi$ between the three remaining corners in such a way that each angle is a multiple of $\frac{\pi}{3}$. We check immediately that there are only four possibilities, corresponding to types $\mathrm{I}$, $\mathrm{II}^{-}$, $\mathrm{II}^{+}$ and $\mathrm{III}$.
\end{proof}

On the other hand, if the boundary of the spectral core $\mathrm{SCore}(\varphi)$ contains some saddle connections, then the interior of $\mathrm{SCore}(\varphi)$ is formed by a smaller number of triangles.

\begin{lemma}\label{lem:d3SingularSpectral}
For a given cubic differential $\varphi=\frac{\alpha x (x-1)dx^{\otimes 3}}{(x-t)^{9}}$ with $s$ real saddle connections, we classify the possible shapes for $\mathrm{SCore}(\varphi)$.
\begin{enumerate}

\item If $s=1$, then there are two cases:
\begin{itemize}
    \item $\mathrm{SCore}(\varphi)$ is the union of a saddle connection and a parallelogram one diagonal of which is another saddle connection;
    \item $\mathrm{SCore}(\varphi)$ is a pentagon with a unique boundary saddle connection.
    \end{itemize}

\item If $s=2$, then there are two cases:
\begin{itemize}
    \item $\mathrm{SCore}(\varphi)$ is a quadrilateral two consecutive boundary sides of which are saddle connections forming an angle of $\frac{\pi}{3}$ or $\frac{2\pi}{3}$;
    \item $\mathrm{SCore}(\varphi)$ is a pair of saddle connections forming an angle equal to $\frac{4\pi}{3}$ or $\pi$.
\end{itemize}
\item If $s=3$, then $\mathrm{SCore}(\varphi)$ is an equilateral triangle bounded by three saddle connections.
\end{enumerate}
\end{lemma}

\begin{proof}
We know from Lemma~\ref{lem:d3dichotomy} that any real saddle connection of $\varphi$ is a boundary saddle connection of $\mathrm{Core}(\varphi)$. Since it is real, it is also a boundary saddle connection of $\mathrm{SCore}(\varphi)$.
\par
Proposition~\ref{prop:SPECpolarBoundary} shows that the interior of $\mathrm{SCore}(\varphi)$ if formed by $4-\beta$ triangles where $\beta$ is the number of saddle connections in the boundary of the spectral polar domain (counted twice if both sides of the saddle connection are boundary sides of the spectral polar domain).
\par
First observe that if a connected component of the interior of $\mathrm{SCore}(\varphi)$ is a triangle, its boundary sides have real slopes and the angles are therefore integer multiples of $\frac{\pi}{3}$. Then, all the angles are equal to $\frac{\pi}{3}$ and the triangle is equilateral. This only happen if the three boundary sides are saddle connections because corners that are regular points have angles of magnitude $\frac{2\pi}{3}$ (see Proposition~\ref{prop:spectralDomainAngles}).
\par
If $\beta=4$, then the spectral polar domain is formed by finitely saddle connections, all of them being real. Since $\mathrm{Core}(\varphi) \subset \mathrm{SCore}(\varphi)$, $\mathrm{Core}(\varphi)$ is a just a pair of saddle connections forming an angle of magnitude in $[\pi,\frac{4\pi}{3}]$. Since we know that they are real, the angle they form is either $\pi$ or $\frac{4\pi}{3}$. This is the second sub-case of the case $s=2$.
\par
If $\beta=3$, then the spectral polar domain is a triangle bounded by three saddle connections. We already described this situation which corresponds to case $s=3$.
\par
If $\beta=2$, then the interior of the spectral polar domain is formed by two triangles. We already ruled out the case where these two triangles are not adjacent. In one case, $\mathrm{SCore}(\varphi)$ is formed by a saddle connection (whose two sides belong to the boundary of the spectral polar domain) and a parallelogram containing another saddle connection (non-real) as a diagonal. This is a sub-case of $s=1$. In a second case, $\mathrm{SCore}(\varphi)$ is a quadrilateral with two boundary saddle connections. Since there is at most one saddle connection for each pair of zeros, these two saddle connections have to be consecutive boundary sides. This is the first sub-case of the case $s=2$.
\par
If $\beta=1$, $\mathrm{SCore}(\varphi)$ is formed by three triangles and we already ruled out the case where one of these triangles is not adjacent to another. The spectral polar domain is thus a pentagon with one boundary saddle connection. This is a sub-case of case $s=1$.
\end{proof}

\subsection{Symmetric loci}\label{sub:d3symmetry}

The flat geometry induced by a cubic differential $\varphi$ is usually hard to deduce from the value of $t$ because the periods of $\varphi^{1/3}$ depend on $t$ in a transcendental way. However, symmetry can help to identify cubic differentials with special properties.

In the $t$-parameter space, $e^{\frac{i\pi}{3}}$ is an orbifold point of order three and therefore corresponds to a flat surface with a symmetry of order three. Following Lemma~\ref{lem:d3dichotomy}, the only possible surface has three saddle connections forming an equilateral triangle (the slope and the length of these saddle connection is then determined by $\alpha$).

The orbifold point of order two $\frac{1}{2}$ should correspond to a flat surface with a symmetry of order two. Using again Lemma~\ref{lem:d3dichotomy}, we deduce that this surface has two saddle connections of same length with equal opposite angles. These angles have to be equal to $\frac{4\pi}{3}$ because the two saddle connections meet in a zero of $\varphi$.

We can also characterize the locus $\mathcal{S} \subset \mathbb{C}$ where the differential admits an anti-holomorphic symmetry.

\begin{lemma}\label{lem:reflection}
For a polynomial cubic differential $\varphi$ of degree three with distinct zeros, the two following statements are equivalent:
\begin{itemize}
    \item $t$ belongs to $\mathcal{S}$ which is the union of the line $\mathrm{Re}(t)=\frac{1}{2}$, the circle of radius $1$ centred on $0$ and the circle of radius $1$ centred on $1$ (see Figure~\ref{fig:sympic});
    \item $\varphi$ admits two saddle connections of equal length.
\end{itemize}
\end{lemma}

\begin{figure}
    \centering
    \includegraphics[width=0.35\linewidth]{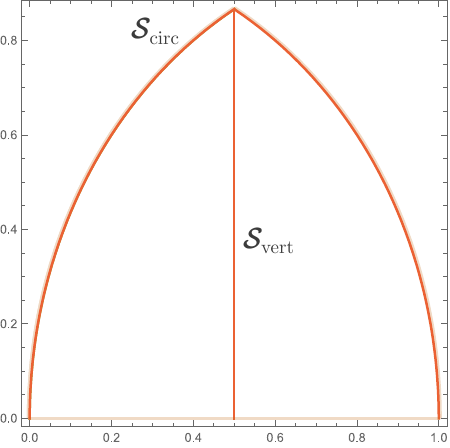}
    \caption{The symmetric locus $\mathcal{S}$}
    \label{fig:sympic}
\end{figure}
\begin{proof}
If $\varphi$ admits two saddle connections of equal length, Lemma~\ref{lem:d3dichotomy} shows that the flat metric induced by $\varphi$ admits an isometry reversing the orientation, fixing one of the three zeros and permuting the two others (a reflection). In the $t$-parameter space, such a map is either $t \mapsto \frac{1}{\bar{t}}$, $t \mapsto 1-\bar{t}$ or $t \mapsto 1-\frac{1}{\bar{t}}$. The invariant locus is therefore the union the line of equation $\mathrm{Re}(t)=\frac{1}{2}$, the circle of radius $1$ centred on $0$ and the circle of radius $1$ centred on $1$. The converse trivially holds.
\end{proof}

\begin{definition}\label{defn:d3Sloci}
The locus $\mathcal{S} \subset \mathbb{C}$ defines two curves in the quotient space $\mathcal{T}$:
\begin{itemize}
    \item the vertical open straight segment $\mathcal{S}_{\mathrm{vert}}$ between $\frac{1}{2}$ and $e^{\frac{i\pi}{3}}$;
    \item the open circular arc $\mathcal{S}_{\mathrm{circ}}$ between $e^{\frac{i\pi}{3}}$ and $0$ (identified with the circular arc between $e^{\frac{i\pi}{3}}$ and $1$).
\end{itemize}    
\end{definition}

The magnitude of one of the two angles drawn between the two equal saddle connections provides a real coordinate on $\mathcal{S}$. Since $t=e^{\frac{i\pi}{3}}$ corresponds to an angle of $\frac{\pi}{3}$ while $t=\frac{1}{2}$ corresponds to an angle of $\frac{4\pi}{3}$, we immediately deduce that such an angle sweeps interval $(\frac{\pi}{3},\frac{4\pi}{3})$ on the vertical straight segment $\mathcal{S}_{\mathrm{vert}}$  and an interval $(0,\frac{\pi}{3})$ when $t$ sweeps the circular arcs of $\mathcal{S}_{\mathrm{circ}}$.

\subsection{$\Delta$-loci in the $t$-parameter space}\label{sub:d3Delta}

Special configurations of saddle connections define curves in the $t$-parameter space. Following Lemma~\ref{lem:d3dichotomy}, we already know that for polynomial cubic differentials of degree $d=3$, two saddle connections always share a unique zero.

\begin{definition}\label{defn:Delta}
We define loci in the $t$-parameter space by
\begin{equation}
    \Delta_k :=\left\{t\in {T} \,\,\Big| \,\, \varphi \text{ has a pair of saddle connections forming an angle of }\frac{k\pi}{3} \right\}, \quad k=1,2,3,4
\end{equation}
In particular $e^{\frac{i\pi}{3}} \in \Delta_{1}$ (see Section~\ref{sub:d3symmetry}). The locus $\Delta_{1}$ decomposes further into $\Delta_{1}^{-}$, $\{e^{\frac{i\pi}{3}}\}$ and $\Delta^{+}_{1}$ depending whether $\mathrm{Re}(t)$ is strictly smaller, equal or strictly larger than $\frac{1}{2}$, respectively.
\end{definition}

\begin{remark}
We check numerically that $\Delta_{1}^{-}$ corresponds to triangles of angles $\theta_1,\theta_2,\theta_3$ (in counterclockwise order) with $\theta_1<\frac{\pi}{3}$, $\theta_2=\frac{\pi}{3}$ and $\theta_3 > \frac{\pi}{3}$. Likewise, $\Delta_{1}^{+}$ corresponds to triangles of angles $\theta_1,\theta_2,\theta_3$ (in counterclockwise order) with $\theta_1>\frac{\pi}{3}$, $\theta_2=\frac{\pi}{3}$ and $\theta_3 < \frac{\pi}{3}$.
\end{remark}

The length ratio between the two saddle connections incident to the angle of fixed magnitude is a convenient real coordinate so loci $\Delta_{1},\Delta_{2},\Delta_{3},\Delta_{4}$ define real curves on the quotient space $\mathcal{T}$. Moreover, these loci are by definition globally invariant by reflection. It follows also from Lemma~\ref{lem:d3dichotomy} that their only common point is $t=0$ (identified with $t=1$). In other words, $\Delta_{1},\Delta_{2},\Delta_{3},\Delta_{4}$ are disjoint outside their endpoints. 

Loci $\Delta_{1}$, $\Delta_{2}$ and $\Delta_{3}$ define three disjoint arcs drawn between $0$ and $1$ which are are invariant under $t \mapsto 1- \bar{t}$. The geometric interpretation of $\Delta_{1}$, $\Delta_{2}$ and $\Delta_{3}$ is straightforward:
\begin{itemize}
\item $\Delta_{1}$ corresponds to flat metrics where the core is a triangle with an angle equal to $\frac{\pi}{3}$, one angle of magnitude larger than $\frac{\pi}{3}$ and one angle of magnitude smaller than $\frac{\pi}{3}$.
\item $\Delta_{2}$ corresponds to flat metrics where the core is a triangle with an angle equal to $\frac{2\pi}{3}$ and two angles of magnitude smaller than $\frac{\pi}{3}$.
\item $\Delta_{3}$ corresponds to flat metrics where the core is a pair of saddle connections forming an angle equal to $\pi$ and an opposite angle equal to $\frac{5\pi}{3}$.
\end{itemize}

The locus $\Delta_{4}$ can be explicitly determined:

\begin{lemma}\label{lem:Delta4}
In the $t$-parameter space, $\Delta_{4}$ coincides with the real axis $\mathbb{R} \setminus \lbrace 0,1,\infty \rbrace $. In the fundamental domain ${T}$, this corresponds to segment $(0,\frac{1}{2}]$ (identified with $[\frac{1}{2},1)$).
\end{lemma}

\begin{proof}
For any $t \in (0,1)$ (which is conjugated to $(-\infty,0)$ and $(1,\infty)$ by the action of $S_{3}$) and $\alpha \in (0,\infty)$, immediate computation of $\varphi^{1/3}$ shows that intervals $(1,\infty)$ and $(-\infty,0)$ are real trajectories and therefore real saddle connections. A symmetry argument shows that the opposite angles at $\infty$ between these two saddle connections should be equal and therefore equal to $\frac{4\pi}{3}$ since the conical angle at $\infty$ is $\frac{8\pi}{3}$. The fact that these two intervals are saddle connections does not depend on $\alpha$. This proves that $\mathbb{R} \setminus \lbrace{ 0,1,\infty \rbrace}$ is contained in $\Delta_{4}$.
\par
Conversely, the flat metric determined by some differential $\varphi$ (up to rotation) in $\Delta_{4}$ is entirely determined by the lengths of the two saddle connections (we are in the first case of Lemma~\ref{lem:d3dichotomy}). Up to scaling, this locus is therefore a connected curve whose endpoints correspond to the collision of two zeros (at parameters $0,1,\infty$ identified by the action of $S_{3}$). This implies that $\Delta_{4}$ actually coincides with $\mathbb{R} \setminus \lbrace{ 0,1,\infty \rbrace}$.
\end{proof}

Although it is generally not possible, the cubic differentials we study are simple enough to allow an explicit formula for the periods of certain saddle connections.

\begin{lemma}\label{lem:SCPeriod}
For any parameter $t \in \mathrm{int}\,T$, the homotopy classes (relative to endpoints) of the segments $[-\infty,0]$ and $[1,\infty]$ are realized respectively by a saddle connection of period 
\begin{align*}&w_0(t)=\int\limits_{-\infty}^{0} \frac{\sqrt[3]{\alpha x(x-1)}dx}{(x-t)^{3}}=
-\frac{2 \pi \sqrt[3]{\alpha } (t^2-t+1)}{9\sqrt{3}(t-1)^{\frac{5}{3}}t^{\frac{5}{3}}}+\frac{5 \sqrt[3]{\alpha } \Gamma(-\frac{5}{3})\Gamma(-\frac{5}{6}) _2F_1(\frac{4}{3},3,\frac{8}{3};\frac{1}{t})}{54\cdot 2^{\frac{2}{3}}\sqrt{\pi}t^3}, \\
&w_1(t)=\int\limits_{1}^{+\infty} \frac{\sqrt[3]{\alpha x(x-1)}dx}{(x-t)^{3}}=\frac{2 \pi \sqrt[3]{\alpha } (t^2-t+1)}{9\sqrt{3}(1-t)^{\frac{5}{3}}(-t)^{\frac{5}{3}}}-\frac{5 \sqrt[3]{\alpha } \Gamma(-\frac{5}{3})\Gamma(-\frac{5}{6}) _2F_1(\frac{4}{3},3,\frac{8}{3};\frac{1}{1-t})}{54\cdot 2^{\frac{2}{3}}\sqrt{\pi}(t-1)^3}.\end{align*}
\end{lemma}

\begin{proof}
In the proof of Lemma~\ref{lem:Delta4}, it has been shown that for $t \in (0,1)$, intervals $(-\infty,0)$ and $(1,+\infty)$ are saddle connections and their period is obtained by the integration of some branch of $\varphi^{1/3}$. As $t$ changes, a saddle connection can be destroyed only if its interior is crossed by some singularity of the flat metric. This would mean an angle of $\pi$ between the two resulting saddle connections so this can only happen when $t$ crosses $\Delta_{3}$.
\par
For a path in ${T}$ starting from $(0,1)$ and then crossing $\Delta_{3}$, the number of saddle connections increases from two to three. No saddle connection disappears. It follows that the two saddle connections already described persists under any deformation of $t$ inside the contractible domain $\mathrm{int}\,T$.

The value of the periods is obtained by direct computation.
\end{proof}

\begin{remark}
Following Lemmas~\ref{lem:reflection} and~\ref{lem:SCPeriod}, we deduce that for $t \in \mathrm{int}\,T$, we have $|w_0(t)|=|w_1(t)|$ if and only if $\mathrm{Re}(t)=\frac{1}{2}$. A direct computation shows that:
\begin{itemize}
    \item $|w_0(t)|>|w_1(t)|$ when $\mathrm{Re}(t)<\frac{1}{2}$;
    \item $|w_0(t)|<|w_1(t)|$ when $\mathrm{Re}(t)>\frac{1}{2}$.
\end{itemize}
In the limit case $t \rightarrow 0$, $|w_0(t)| \rightarrow + \infty$ while $|w_1(t)| \rightarrow + \infty$ when $t \rightarrow 1$.
\par
We also remark that where $\mathrm{Core}(\varphi)$ is a triangle, its three boundary saddle connections (in the counterclockwise order) are homotopic to $[1,\infty]$, $[-\infty,0]$, $[0,1]$ 
(and the cyclic order on its vertices are $0,1,\infty$).
\end{remark}

We can also give a characterization of the possible shapes of the spectral core in terms of the $\Delta$-loci (see Lemma~\ref{lem:d3spectral} for the possible shapes of spectral cores for spectral networks without saddle connections).

\begin{proposition}\label{prop:d3DeltaSpectral}
Consider a cubic differential $\varphi=\frac{\alpha x (x-1)dx^{\otimes 3}}{(x-t)^{9}}$. There can exist a phase $\vartheta \in \mathbb{R}/\frac{\pi}{3}\mathbb{Z}$ for which the spectral core:
\begin{itemize}
    \item is of type $\mathrm{I}$ only if $t$ lies strictly above $\Delta_{2}$;
    \item is of type $\mathrm{II}^{-}$ or $\mathrm{II}^{+}$ only if $t$ lies above $\Delta_{3}$;
    \item if of type $\mathrm{III}$ only if $t$ lies strictly between $\Delta_{2}$ and $\Delta_{4}$;
    \item is of type $\mathrm{IV}$ only if $t$ lies strictly below $\Delta_{3}$.
\end{itemize}
\end{proposition}

\begin{proof}
If $\mathrm{SCore}(\varphi)$ is of type $\mathrm{I}$, $\mathrm{Core}(\varphi)$ is a triangle inscribed in the hexagon. Since the angles in the hexagon are equal to $\frac{2\pi}{3}$, the angles of $\mathrm{Core}(\varphi)$ are strictly smaller than $\frac{2\pi}{3}$ and $t$ lies strictly above $\Delta_{2}$.
\par
If $\mathrm{SCore}(\varphi)$ is of type $\mathrm{II}^{-}$ or $\mathrm{II}^{+}$, it is a convex polygon and $\mathrm{Core}(\varphi)$ is a triangle joining three of its vertices. Therefore, we have three saddle connections and angles between that are strictly smaller than $\pi$. It follows that $t$ lies above $\Delta_{3}$. 
\par
If $\mathrm{SCore}(\varphi)$ is of type $\mathrm{III}$, the spectral core is a non-convex hexagon with interior angles $\frac{4\pi}{3},\frac{2\pi}{3},\frac{\pi}{3},\frac{2\pi}{3},\frac{\pi}{3},\frac{2\pi}{3}$ where the first, third, and fifth angles correspond to zeroes $a,b,c$ of $\varphi$. The saddle connections between $a$ and $b$ and between $a$ and $c$ form an angle $\alpha$ whose magnitude lies in $(\frac{2\pi}{3},\frac{4\pi}{3})$. Since by definition $\mathrm{Core}(\varphi)$ is contained in $\mathrm{SCore}(\varphi)$, $\mathrm{Core}(\varphi)$ is a triangle if and only if $\alpha \in (\frac{2\pi}{3},\pi)$. In this case, $t$ lies strictly between $\Delta_{2}$ and $\Delta_{3}$. Otherwise, $\mathrm{Core}(\varphi)$ is formed by two saddle connections forming an angle strictly smaller than $\frac{4\pi}{3}$ and $t$ lies strictly above $\Delta_{4}$ and below $\Delta_{3}$ (but possibly on $\Delta_{3}$).
\par
If $\mathrm{SCore}(\varphi)$ is of type $\mathrm{IV}$, $\varphi$ has exactly two saddle connections forming a pair of opposite angles strictly larger than $\pi$ so $t$ lies strictly below $\Delta_{3}$.
\end{proof}

\subsection{Tripods}\label{sub:d3Tripods}

The existence of a phase for which the spectral network contains a tripod (see Section~\ref{sub:Trajectories}) implies that $t$ belongs to a specific domain of the $t$-parameter space cut out by $\Delta_{2}$, see Definition~\ref{defn:Delta}.

\begin{proposition}\label{prop:d3Tripod}
Consider a cubic differential $\varphi=\frac{\alpha x (x-1)dx^{\otimes 3}}{(x-t)^{9}}$. If the spectral network $\mathcal{W}_\vartheta(\varphi)$ contains a {tripod}, this tripod is critical. Moreover, $\mathrm{Core}(\varphi)$ is a triangle where every interior corner has an angle strictly smaller than $\frac{2\pi}{3}$ and the intersection point of the three trajectories is the Fermat point of the triangle.
\par
Conversely, for any $\varphi$ such that the angles of the corners are strictly smaller than $\frac{2\pi}{3}$, there is a unique phase $\vartheta$ in $\mathbb{R}/\frac{\pi}{3}\mathbb{Z}$ such that three real critical trajectories of same sign of the spectral network $\mathcal{W}_{\vartheta}(\varphi)$ join at the (unique) Fermat point of the triangle. For this unique phase $\vartheta$, the spectral core is of type $\mathrm{I}$. In particular, $\mathcal{W}_{\vartheta}(\varphi)$ does not contain any saddle connection.
\par
These conditions on the angles at the corners of the core amount to require that $t$ is located strictly above the curve $\Delta_{2}$ in the fundamental domain ${T}$.
\end{proposition}

\begin{proof}
If there exists such a tripod, then the fact that the trajectories meet with mutual angles of $\frac{2\pi}{3}$ implies that the intersection point belongs to the interior of $\mathrm{SCore}(\varphi)$. Then, there should be three trajectories with a same sign that cross the interior of $\mathrm{SCore}(\varphi)$. Listing the possible shapes for the spectral core (see Section~\ref{sub:d3SpectralCore}), we check that $\mathrm{SCore}(\varphi)$ does not contain any saddle connection and should be of type $I$ (a hexagon with corner angles equal to $\frac{2\pi}{3}$). In this case, the only possible tripod is a critical tripod.
\par
If there exists such a critical tripod, denote by $o$ the meeting point of the three trajectories and consider the segments $\alpha$, $\beta$, $\gamma$ of these trajectories between $o$ and the zeroes $a, b, c$ respectively. The concatenation of two of these segments produces a topological arc whose homotopy class has a unique geodesic representative which has to be a broken geodesic formed by finitely many saddle connections. Since $\varphi$ has either two or three saddle connections and that none of theme is a closed saddle connection, it is immediate that $a,b,c$ are distinct zeros of $\varphi$, $\mathrm{Core}(\varphi)$ is a triangle and that $\alpha,\beta,\gamma$ are contained in that triangle. Since the angles between these three segments at $o$ are equal to $\frac{2\pi}{3}$, this characterizes $o$ as the Fermat point of the triangle forming $\mathrm{Core}(\varphi)$. It is well-known that a triangle with an interior Fermat point cannot have a corner of angle larger or equal to $\frac{2\pi}{3}$.
\par
Conversely, the Fermat point of such a triangle is unique and for a given zero $a$ of $\varphi$, there is only one phase (up to rotation of angle $\frac{\pi}{3}$) such that $a$ is contained in a real critical trajectory starting from $a$. It follows from the definition of the Fermat point that $o$ is also contained in real critical trajectories starting from $b$ and $c$.
\par
It order to prove that a phase $\vartheta$ for which the spectral network $\mathcal{W}_{\vartheta}(\varphi)$ contains a critical tripod cannot contain any saddle connection, we assume by contradiction that $\mathcal{W}_{\vartheta}(\varphi)$ contains a trajectory connecting $a$ and $b$ in addition to $\alpha$, $\beta$ and $\gamma$. Then all the angles of triangle $\Delta_{abo}$ are integer multiples of $\frac{\pi}{3}$ which is impossible since the angle at $o$ has to be $\frac{2\pi}{3}$.
\end{proof}

\subsection{Walls of the first kind}\label{sub:d3WallsFirst}

We introduce a stratification of the quotient space $\mathcal{T}$ by defining the singular locus as the values of $t$ for which two degenerations can appear for a same phase.

\begin{definition}\label{defn:d3WALLS}
On the quotient space $\mathcal{T}$, we introduce a stratification (see Figure~\ref{fig:introdiagram}):
\begin{itemize}
    \item \emph{vertices}: $0$ (or equivalently $1$), $e^{\frac{i\pi}{3}}$ and $\frac{1}{2}$;
    \item \emph{walls (of the first kind)}: $\Delta_1= \Delta_{1}^{-}\cup \Delta_{1}^{+}$, $\Delta_{2}$, $\Delta_{3}$ and $\Delta_{4}$ (see Definitions~\ref{defn:Delta});
    \item \emph{chambers}: connected components $\mathcal{C}_A, \mathcal{C}_{B}$, $\mathcal{C}_{C}$ and $\mathcal{C}_{D}$ of the complement of the walls and vertices.
\end{itemize}
\end{definition}

When $t$ belongs to one of the chambers, degenerations occur at distinct phases. These special phases decompose $\mathbb{R}/\frac{\pi}{3}\mathbb{Z}$ into intervals where the shape of the spectral core (see Lemma~\ref{lem:d3spectral}) is structurally stable. For values of $t$ in the singular locus, we encounter the same dichotomy except that several degenerations can appear for the same phase.

\begin{theorem}\label{thm:d3WALLS}
Consider a cubic differential $\varphi=\frac{\alpha x (x-1)dx^{\otimes 3}}{(x-t)^{9}}$. If $t$ belongs to a chamber, then the circle of phases $\mathbb{R}/\frac{\pi}{3}\mathbb{Z}$ decomposes into:
\begin{itemize}
    \item at most $4$ special phases where $\mathcal{W}_\vartheta(\varphi)$ contains exactly one saddle connection or tripod of that phase;
    \item intervals where spectral cores are of the same type (see Lemma~\ref{lem:d3spectral}) and where the spectral networks do not contain any saddle trajectory or tripod of those phases.
\end{itemize}
If $t\neq0,1$ belongs to a wall or vertex distinct, then the circle of phases $\mathbb{R}/\frac{\pi}{3}\mathbb{Z}$ decomposes into:
\begin{itemize}
    \item at most $3$ special phases where $\mathcal{W}_\vartheta(\varphi)$ contains at least one saddle connection or tripod of that phase;
    \item intervals where spectral cores are of the same type and where the spectral networks do not contain any saddle trajectory or tripod of those phases.
\end{itemize}
Moreover, in the circle of phases, the type of the spectral core only changes at special phases where there is a saddle connection (it remains the same on either side of a tripod).
\end{theorem}

\begin{proof}
Lemma~\ref{lem:d3dichotomy}, any differential $\varphi$ has at most three saddle connections so for any $t$, there are at most three phases where the spectral network contains a saddle connection. Proposition~\ref{prop:d3Tripod} similarly proves that there is at most one phase where the spectral network contains a tripod. The same proposition shows that this phase cannot be the phase of a saddle connection. Therefore, the only case where two degenerations happen at the same phase is where two saddle connections have the same phase. In the $t$-parameter space, this defines the $\Delta$-locus. Therefore, outside the $\Delta$-locus, there are at most $4$ special phases while there are only $3$ of them in the $\Delta$-locus.
\par
Outside these special phases, the spectral network does not contain any saddle connection so the boundary of the spectral core does not contain any saddle connection either. Thus, Lemma~\ref{lem:d3spectral} provides the classification of the possible shapes for the spectral core into five different classes. As the boundary vertices of the spectral core are either zeros of $\varphi$ or intersections of two critical trajectories, a small enough change in the phase does not change the type of the spectral core.
\par
In contrast with saddle connections that appear in the boundary of the spectral core, tripods correspond to the intersection of three critical trajectories in the interior of the spectral core. Therefore, in two intervals in $\mathbb{R}/\frac{\pi}{3}\mathbb{Z}$ adjacent to a special phase where the spectral network has a tripod (but no saddle connection), spectral cores are of the same type.
\end{proof}

\subsection{Lists of saddle connections and critical tripods}\label{sub:d3WallsSecond}
In this section, we each cell of the stratification introduced in Section~\ref{sub:d3WallsFirst}, we describe the spectral networks parametrized by the phases in $\mathbb{R}/\frac{\pi}{3}\mathbb{Z}$. Special phases (where degenerations happen) and intervals of structurally stable spectral networks are listed according to the counterclockwise order.
\par
We only give a classification of the spectral cores. This could be refined (by adding more walls and special phases) to a complete classification of the spectral networks but the task would be drastically more complicated.
\par
The determination of the list of networks $\mathcal{W}_\vartheta(\varphi)$ as $\vartheta\in[0,\frac{\pi}{3})$ for each chamber is deduced from:
\begin{itemize}
    \item Theorem~\ref{thm:d3WALLS} that provides the decomposition of $\mathbb{R}/\frac{\pi}{3}\mathbb{Z}$ into special phases and intervals;
    \item Proposition~\ref{prop:d3DeltaSpectral} that provides strong restrictions on the possible shapes for the spectral core (for phases without any real saddle connection);
    \item Proposition~\ref{prop:d3Tripod} that characterizes the values of $t$ for which some phase admits a tripod. 
\end{itemize}
It also follows from Definition~\ref{defn:Delta} that $\varphi$ admits three saddle connections if $t$ lies strictly above $\Delta_{3}$ and two saddle connections otherwise.

\subsubsection{Spectral networks for $t$ in chambers}

We list here the spectral networks at all phases modulo in $\mathbb{R}/\frac{\pi}{3}$ in each chamber. We illustrate these results in Figures \ref{fig:chamberD}-\ref{fig:chamberB}, noting that $0,1$ and $\infty$ have been sent to $-\frac{1}{t}$, $\frac{1}{1-t}$, and $0$, respectively to facilitate plotting. We chose $\alpha\approx-13.89 + 8.23i$.

The spectral cores and degenerations for $t$ in $\mathcal{C}_{D}$ are given (in the counter-clockwise order of their phase) by: $1$ phase with one saddle connection, $1$ interval with a spectral core of type $\mathrm{III}$, $1$ phase with one saddle connection, $1$ interval with a spectral core of type $\mathrm{IV}$. This is illustrated in Figure~\ref{fig:chamberD} and summarized in Table~\ref{table:chamberD}.

 \begin{table}[h]\centering
  \begin{tabular}{|c||c|c|c|c|} \hline
    \textbf{Type for $t\in\mathcal{C}_D$}
    & saddle
    & III & saddle & IV  \\ \hline
  \end{tabular} 
     \caption{Spectral cores and degenerations for $t\in\mathcal{C}_D$.}
     \label{table:chamberD}
  \end{table}

\begin{figure}[h]
    \centering
        \begin{subfigure}[t]{.32\textwidth}
        \centering
        \includegraphics[height=\linewidth,width=\linewidth, trim=3cm 4.5cm 3cm 4.8cm, clip]{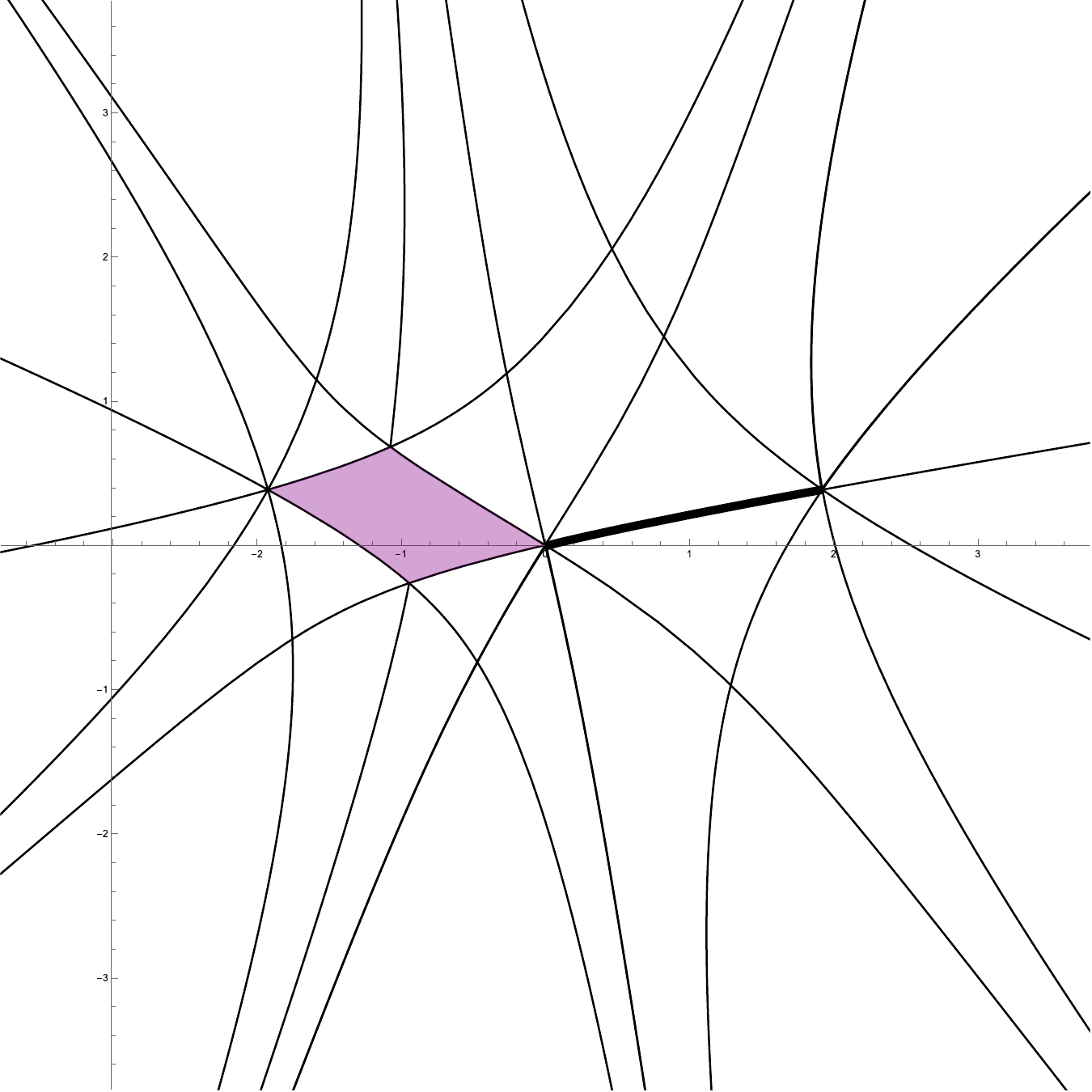}
        \caption{$\vartheta\approx0.31$, saddle}
      \label{fig:chamberD-4}
    \end{subfigure}
    \begin{subfigure}[t]{.32\textwidth}
        \centering
        \includegraphics[height=\linewidth,width=\linewidth, trim=3cm 6.5cm 3cm 4.8cm, clip]{figures/hexagontplane-param0pt5,0pt1,0pt_shaded.pdf}      
        \caption{$\vartheta \approx 0.524$, type $\mathrm{IV}$}
      \label{fig:chamberD-1}
    \end{subfigure}
    \hspace{0.0cm} \\              
    \begin{subfigure}[t]{.32\textwidth}
        \centering
        \includegraphics[height=\linewidth,width=\linewidth, trim=3cm 6.5cm 3cm 4.8cm, clip]{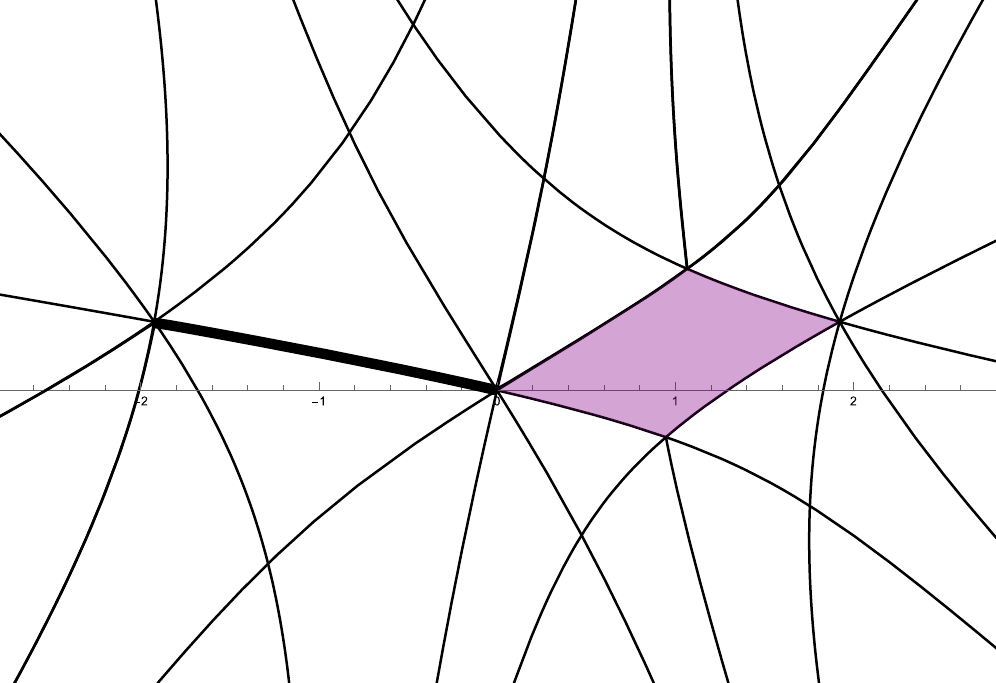}
        \caption{$\vartheta\approx0.74$, saddle}
      \label{fig:chamberD-2}
    \end{subfigure}
    \begin{subfigure}[t]{.32\textwidth}
        \centering
        \includegraphics[height=\linewidth,width=\linewidth, trim=3cm 6.5cm 3cm 4.8cm, clip]{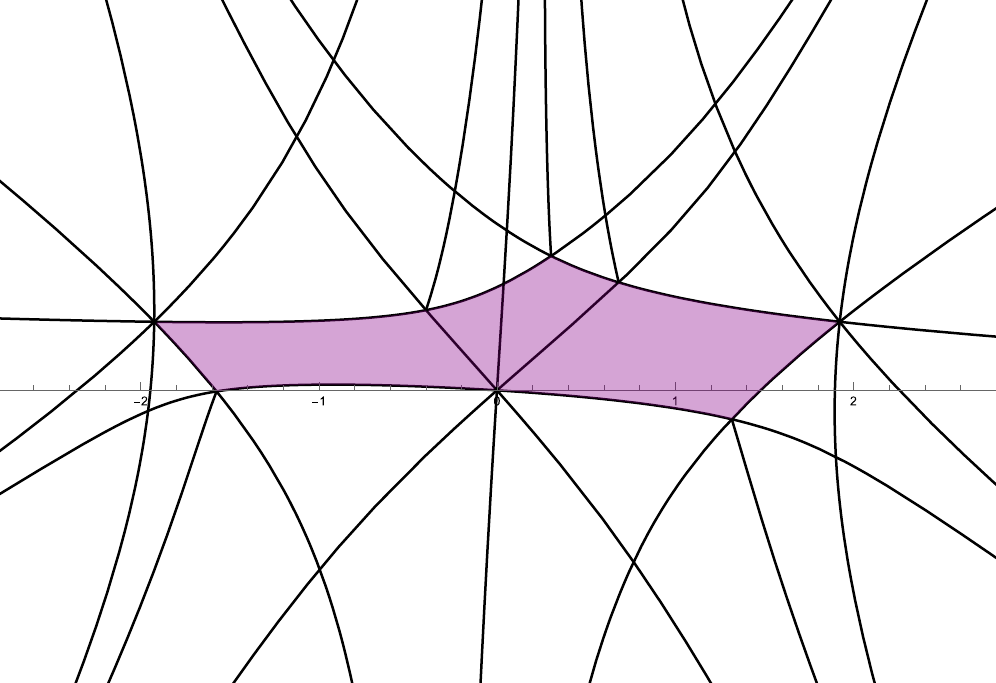}
        \caption{$\vartheta\approx0.96$, type $\mathrm{III}$}
      \label{fig:chamberD-3}
    \end{subfigure}            
 \caption{Spectral networks and spectral cores for $t\approx 0.5 + 0.1i \in \mathcal{C}_D$}
  \label{fig:chamberD}
\end{figure}

For $t$ in $\mathcal{C}_{C}$ they are given by: $1$ phase with one saddle connection, $1$ interval with a spectral core of type $\mathrm{II}^{-}$, $1$ phase with one saddle connection, $1$ interval with a spectral core of type $\mathrm{II}^{+}$, $1$ phase with one saddle connection, $1$ interval with a spectral core of type $\mathrm{III}$. This is illustrated in Figure~\ref{fig:chamberC} and summarized in Table~\ref{table:chamberC}.

 \begin{table}[h]\centering
  \begin{tabular}{|c||c|c|c|c|c|c|} \hline
    \textbf{Type for $t\in\mathcal{C}_C$}
    & saddle
    & $\mathrm{II}^-$ & saddle & $\mathrm{II}^+$  & saddle & III \\ \hline  \end{tabular} 
     \caption{Spectral cores and degenerations for $t\in\mathcal{C}_C$.}
     \label{table:chamberC}
  \end{table}

\begin{figure}
\centering
    \begin{subfigure}[t]{.32\textwidth}
        \centering
        \includegraphics[height=\linewidth,width=\linewidth, trim=3cm 6cm 3cm 4cm, clip]{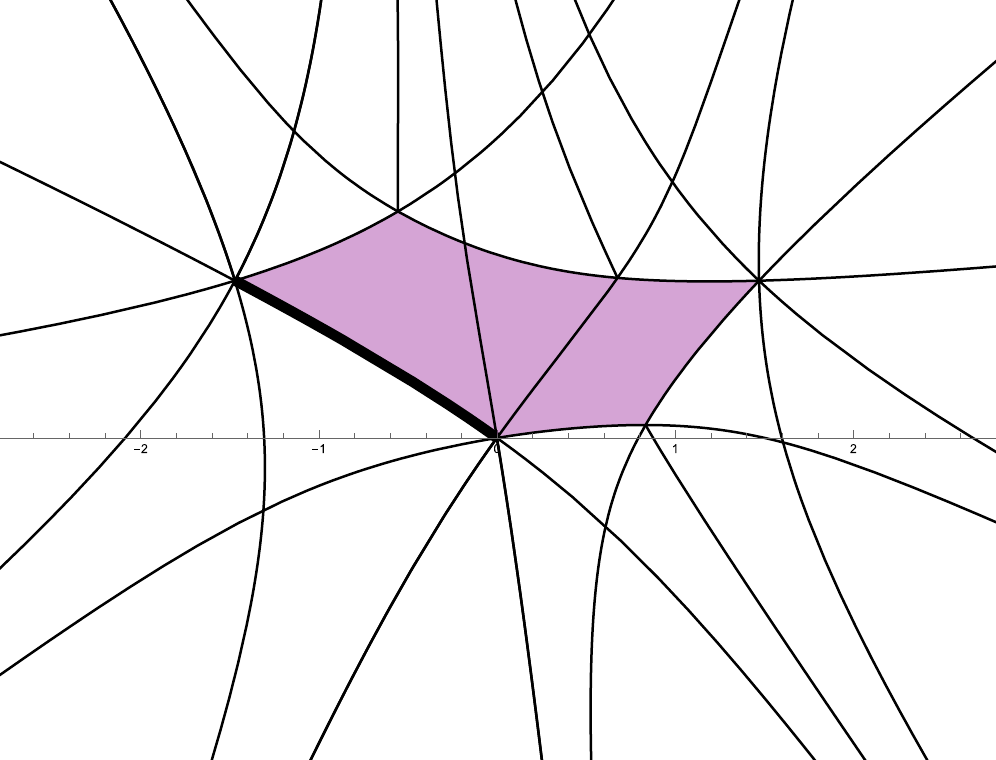}  
        \caption{$\vartheta\approx 0.22$, saddle}
      \label{fig:chamberC-5}
    \end{subfigure}
    \hspace{0.0cm}             
    \begin{subfigure}[t]{.32\textwidth}
        \centering
        \includegraphics[height=\linewidth,width=\linewidth, trim=3cm 6cm 3cm 4cm, clip]{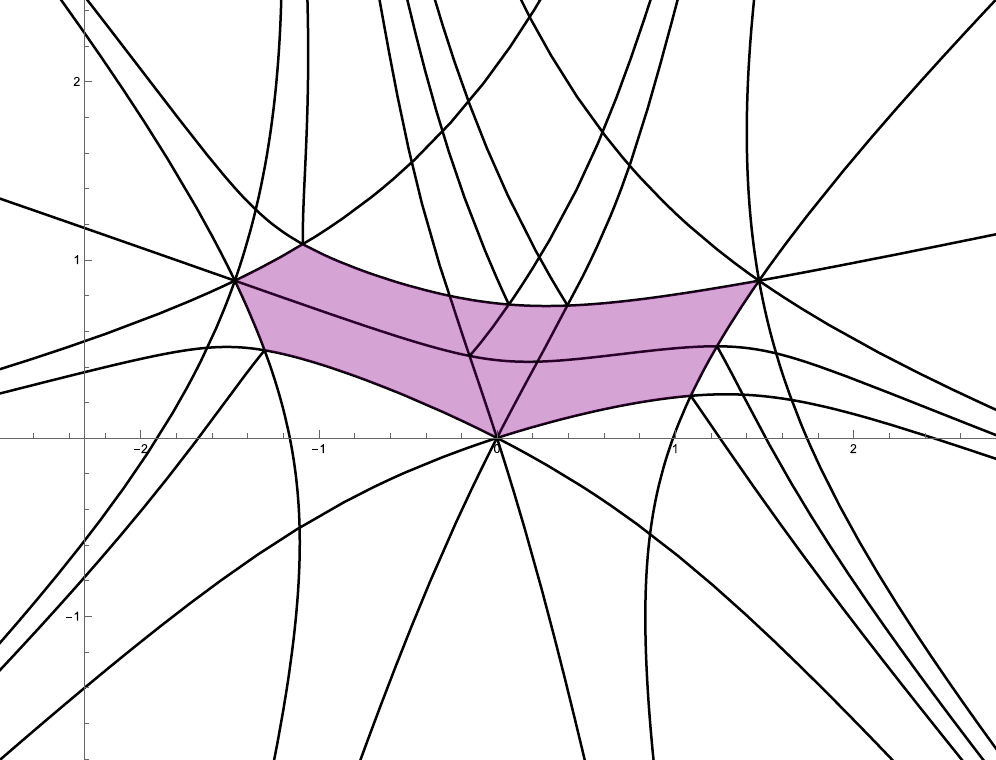}
        
        \caption{$\vartheta\approx 0.42$, type $\mathrm{II}^{-}$}
      \label{fig:chamberC-6}
    \end{subfigure}
    \begin{subfigure}[t]{.32\textwidth}
        \centering
        \includegraphics[height=\linewidth,width=\linewidth, trim=3cm 6cm 3cm 4cm, clip]{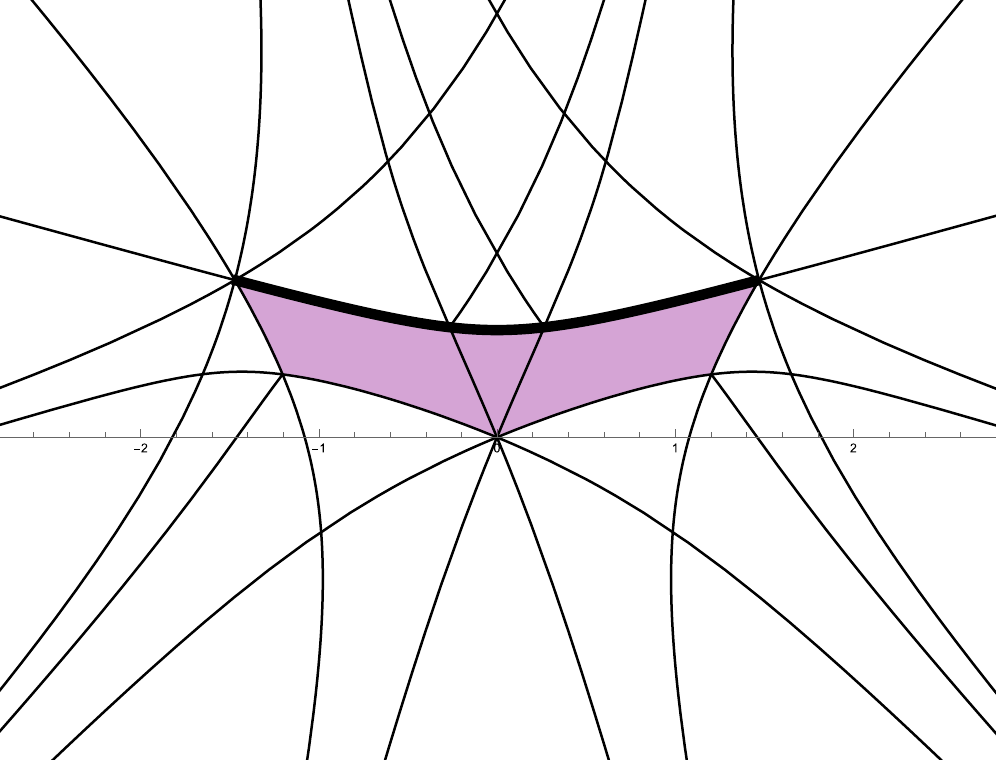}      
        \caption{$\vartheta \approx 0.524$, saddle}
      \label{fig:chamberC-1}
    \end{subfigure}
    \hspace{0.0cm}               
    \begin{subfigure}[t]{.32\textwidth}
        \centering
    \includegraphics[height=\linewidth,width=\linewidth, trim=3cm 6cm 3cm 4cm, clip]{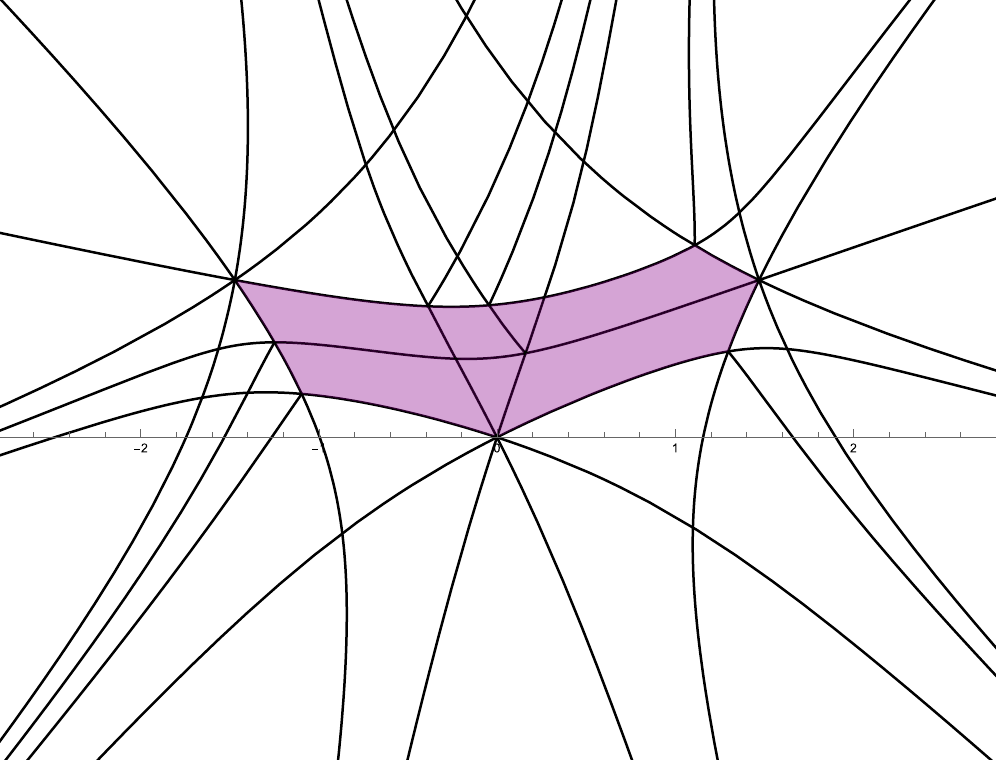}
        \caption{$\vartheta\approx0.624$, type $\mathrm{II}^{+}$}
      \label{fig:chamberC-2}
    \end{subfigure}  
    \begin{subfigure}[t]{.32\textwidth}
        \centering
        \includegraphics[height=\linewidth,width=\linewidth, trim=3cm 6cm 3cm 4cm, clip]{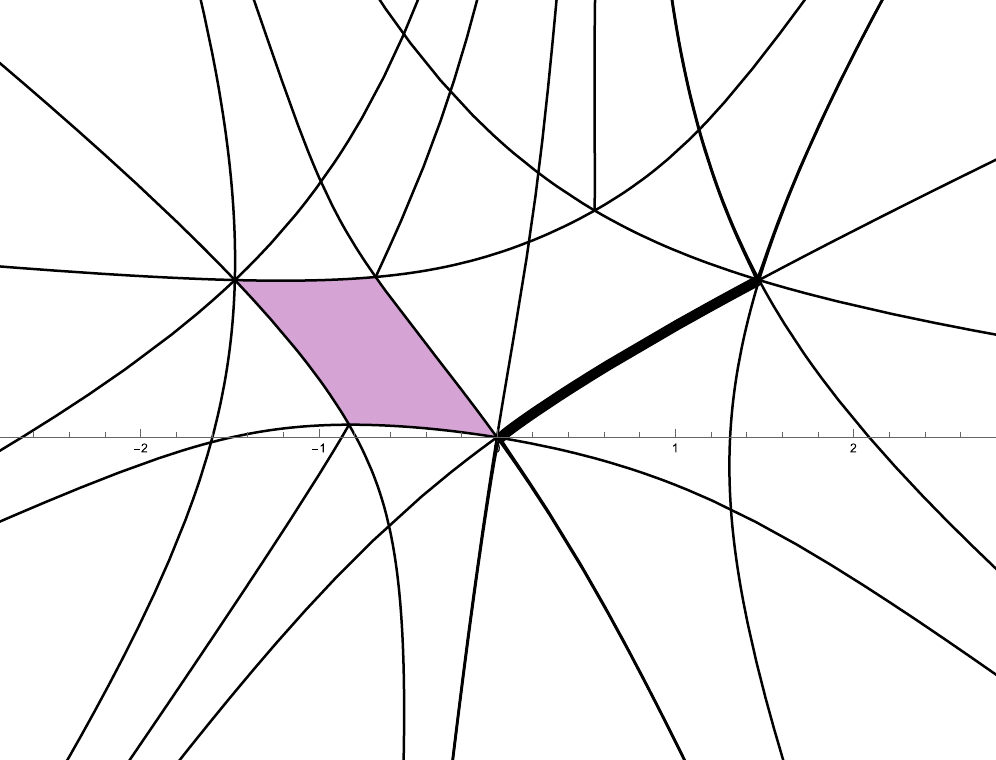}
        \caption{$\vartheta\approx0.83$, saddle}
      \label{fig:chamberC-3}
    \end{subfigure}
    \begin{subfigure}[t]{.32\textwidth}
        \centering
        \includegraphics[height=\linewidth,width=\linewidth, trim=3cm 6cm 3cm 4cm, clip]{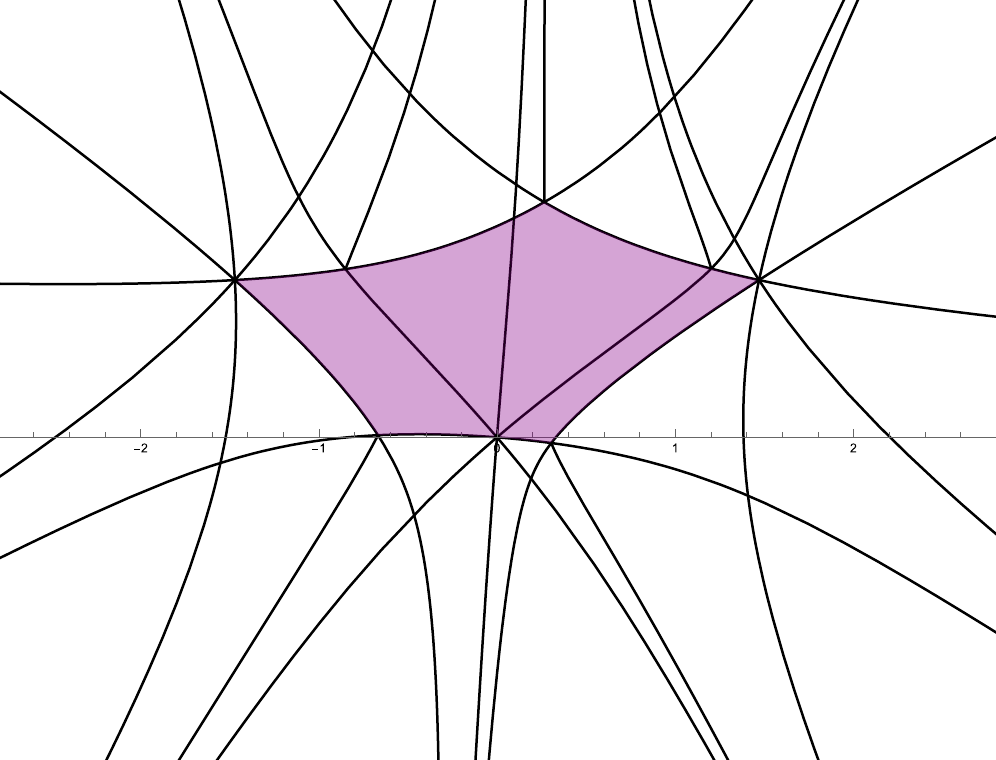}
        \caption{$\vartheta\approx0.94$, type $\mathrm{III}$}
      \label{fig:chamberC-4}
    \end{subfigure}
 \caption{Spectral networks and spectral cores for $t\approx 0.5+0.3i\in \mathcal{C}_C$}
  \label{fig:chamberC}
\end{figure}

For $t$ in $\mathcal{C}_{B}$ they are given by: $1$ phase with one saddle connection, $1$ interval with a spectral core of type $\mathrm{I}$, $1$ phase with one tripod, $1$ interval with a spectral core of type $\mathrm{I}$, $1$ phase with one saddle connection, $1$ interval with a spectral core of type $\mathrm{II}^{-}$, $1$ phase with one saddle connection, $1$ interval with a spectral core of type $\mathrm{II}^{+}$. This is illustrated in Figure~\ref{fig:chamberB} and summarized in Table~\ref{table:chamberB}.

 \begin{table}[h]\centering
  \begin{tabular}{|c||c|c|c|c|c|c|c|c|} \hline
    \textbf{Type for $t\in\mathcal{C}_B$}
    & saddle
    & $\mathrm{I}$ & tripod & $\mathrm{I}$  & saddle &$\mathrm{II}^{-}$ & saddle & $\mathrm{II}^{+}$ \\ \hline
        \textbf{Type for $t\in\mathcal{C}_A$}
    & saddle
    & $\mathrm{I}$ & tripod & $\mathrm{I}$  &saddle &$\mathrm{II}^{+}$ &saddle & $\mathrm{II}^{-}$ \\ \hline
  \end{tabular} 
     \caption{Spectral cores and degenerations for $t\in\mathcal{C}_B$ and $t\in\mathcal{C}_A$.}
     \label{table:chamberB}
  \end{table}

\begin{figure}
    \centering
        \begin{subfigure}[t]{.32\textwidth}
        \centering
        \includegraphics[height=\linewidth,width=\linewidth, trim=4.5cm 6.5cm 4.5cm 4.5cm, clip]{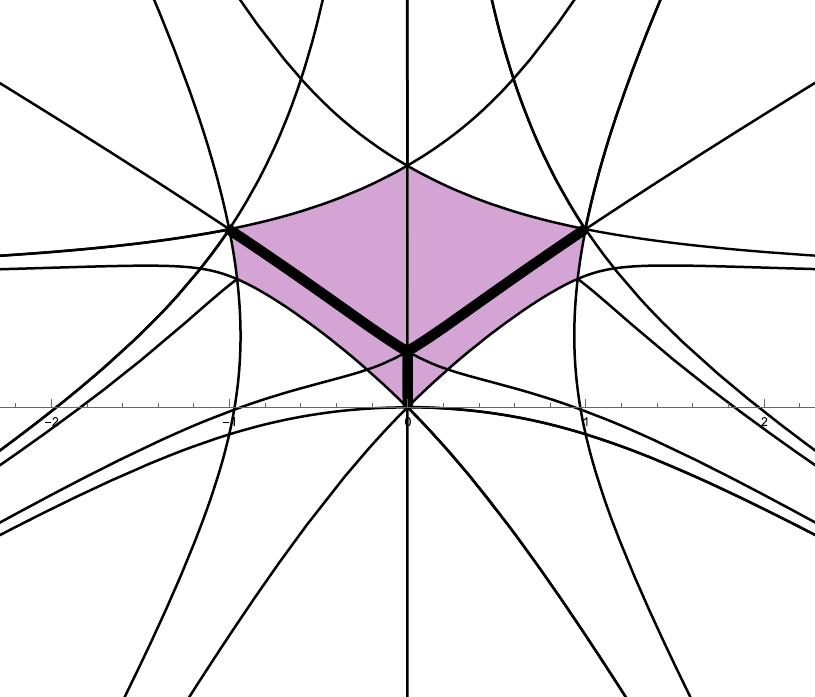}
        \caption{$\vartheta\approx 0.00$, tripod}
      \label{fig:chamberB-5}
    \end{subfigure}
    \hspace{0.0cm}             
    \begin{subfigure}[t]{.32\textwidth}
        \centering
        \includegraphics[height=\linewidth,width=\linewidth, trim=4.5cm 6.5cm 4.5cm 4.5cm, clip]{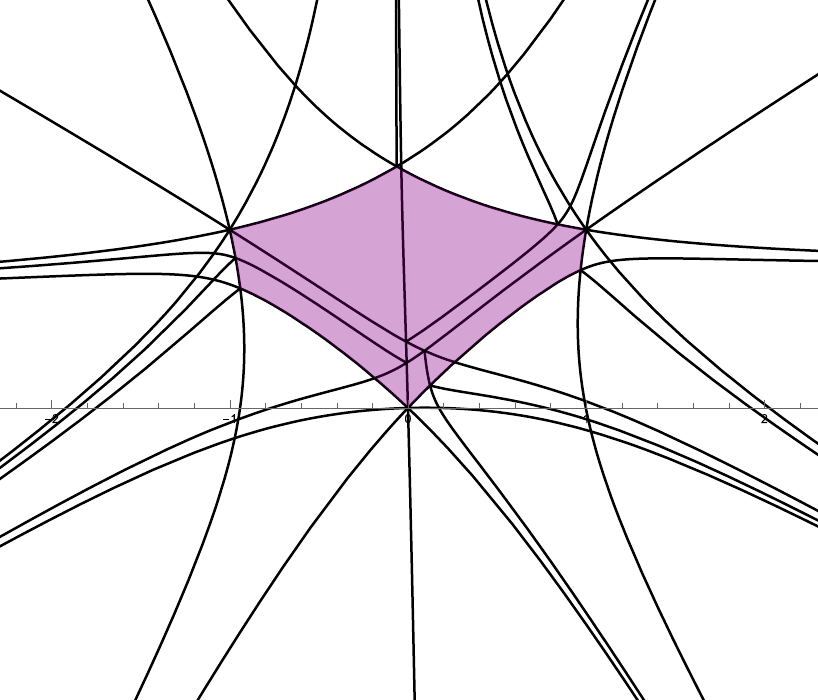}
        
        \caption{$\vartheta\approx 0.036$, type $\mathrm{I}$}
      \label{fig:chamberB-6}
    \end{subfigure}
    \hspace{0.0em}
        \begin{subfigure}[t]{.32\textwidth}
        \centering
        \includegraphics[height=\linewidth,width=\linewidth, trim=4.5cm 6.5cm 4.5cm 4.5cm, clip]{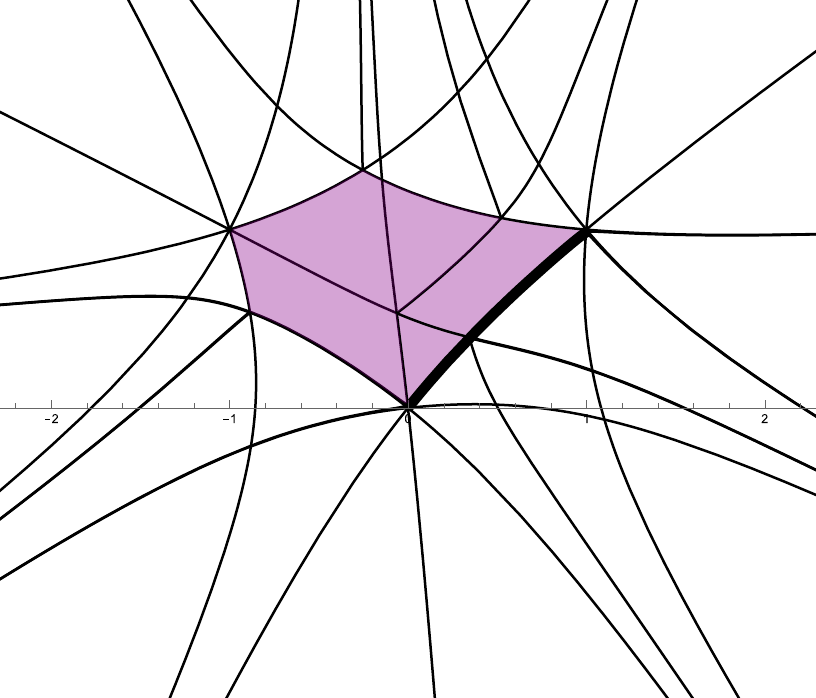}
        
        \caption{$\vartheta\approx 0.15$, saddle}
      \label{fig:chamberB-7}
    \end{subfigure}
    \hspace{0.0cm}             
    \begin{subfigure}[t]{.32\textwidth}
        \centering
        \includegraphics[height=\linewidth,width=\linewidth, trim=4.5cm 6.5cm 4.5cm 4.5cm, clip]{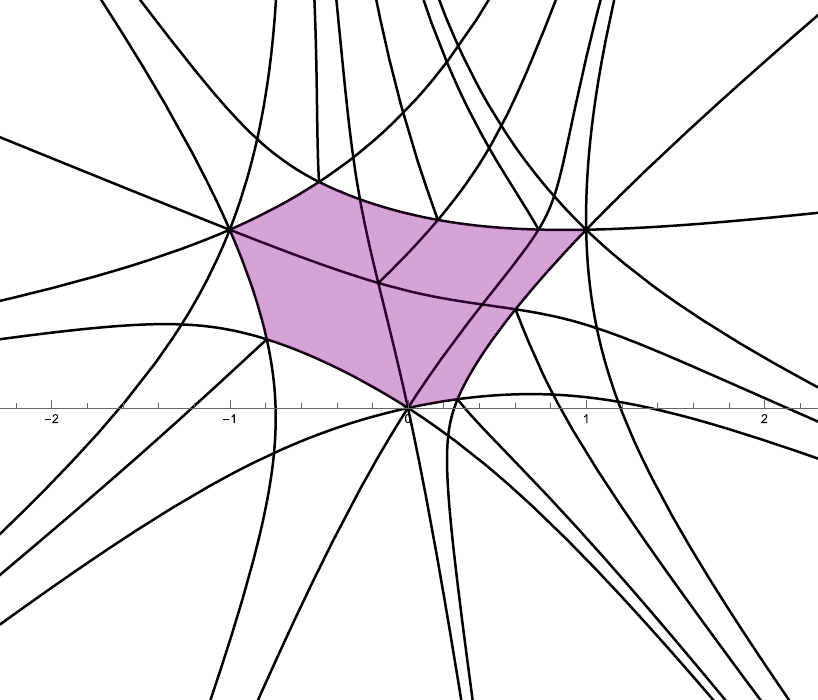}
        
        \caption{$\vartheta\approx 0.29$, type $\mathrm{II}^{-}$}
      \label{fig:chamberB-8}
    \end{subfigure}
    \begin{subfigure}[t]{.32\textwidth}
        \centering
        \includegraphics[height=\linewidth,width=\linewidth, trim=4cm 6cm 4cm 4cm, clip]{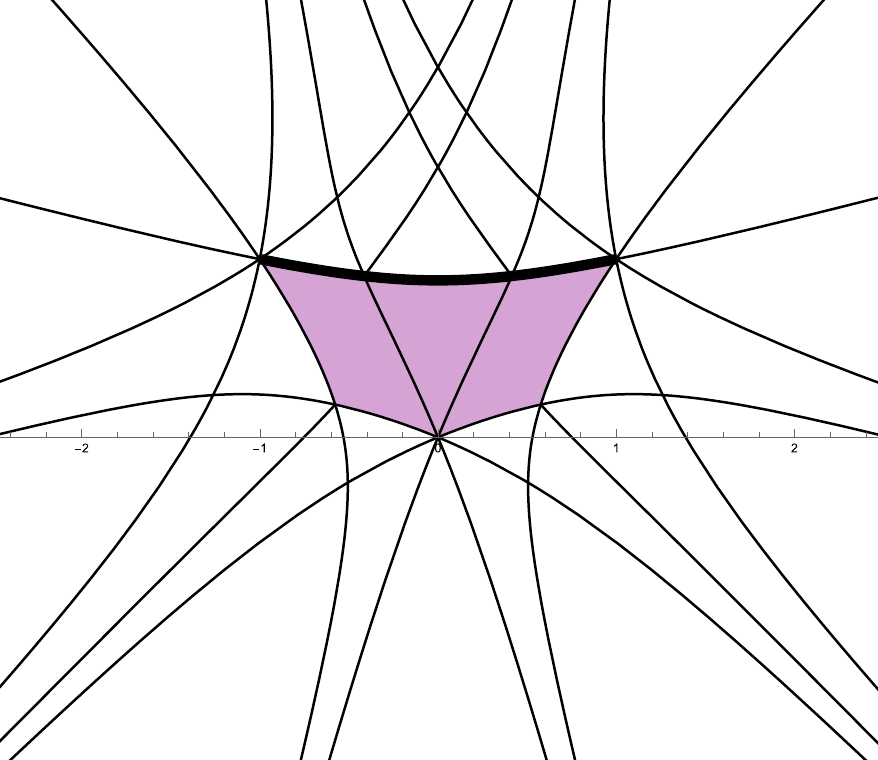}      
        \caption{$\vartheta \approx 0.524$, saddle}
      \label{fig:chamberB-1}
    \end{subfigure}
    \hspace{0.0cm}               
    \begin{subfigure}[t]{.32\textwidth}
        \centering
        \includegraphics[height=\linewidth,width=\linewidth, trim=4.5cm 6.5cm 4.5cm 4.5cm, clip]{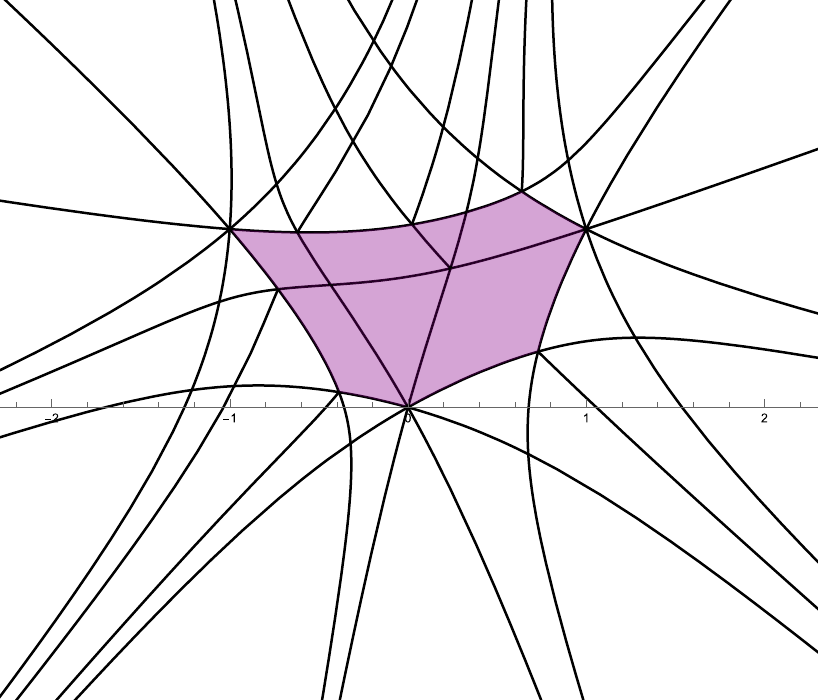}
        \caption{$\vartheta\approx0.68$, type $\mathrm{II}^{+}$}
      \label{fig:chamberB-2}
    \end{subfigure}
    \hspace{0.0cm}             
    \begin{subfigure}[t]{.32\textwidth}
        \centering
        \includegraphics[height=\linewidth,width=\linewidth, trim=4.5cm 6.5cm 4.5cm 4.5cm, clip]{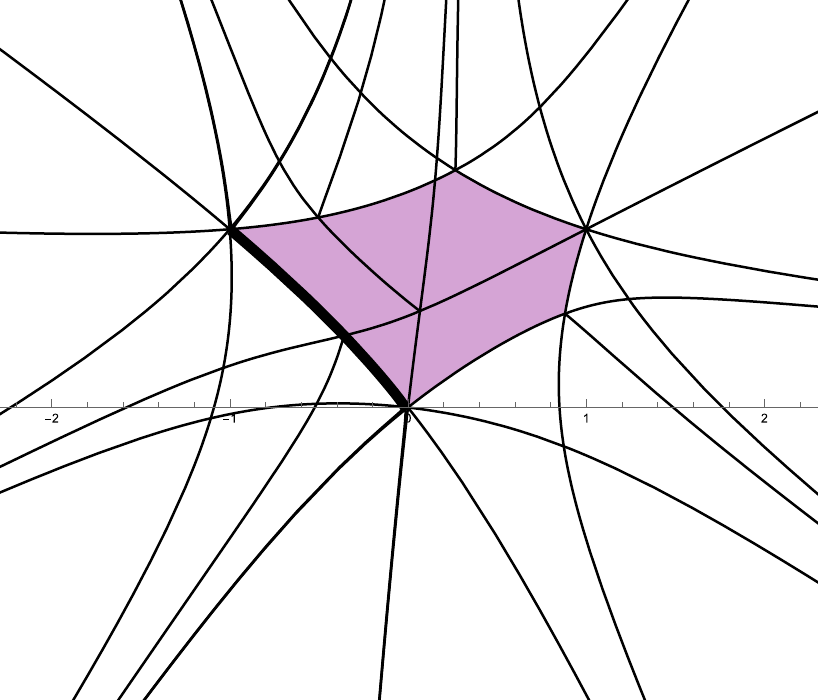}

        \caption{$\vartheta\approx0.89$, saddle}
      \label{fig:chamberB-3}
    \end{subfigure}
     \hspace{0.0cm}              
    \begin{subfigure}[t]{.32\textwidth}
        \centering
        \includegraphics[height=\linewidth,width=\linewidth, trim=4.5cm 6.5cm 4.5cm 4.5cm, clip]{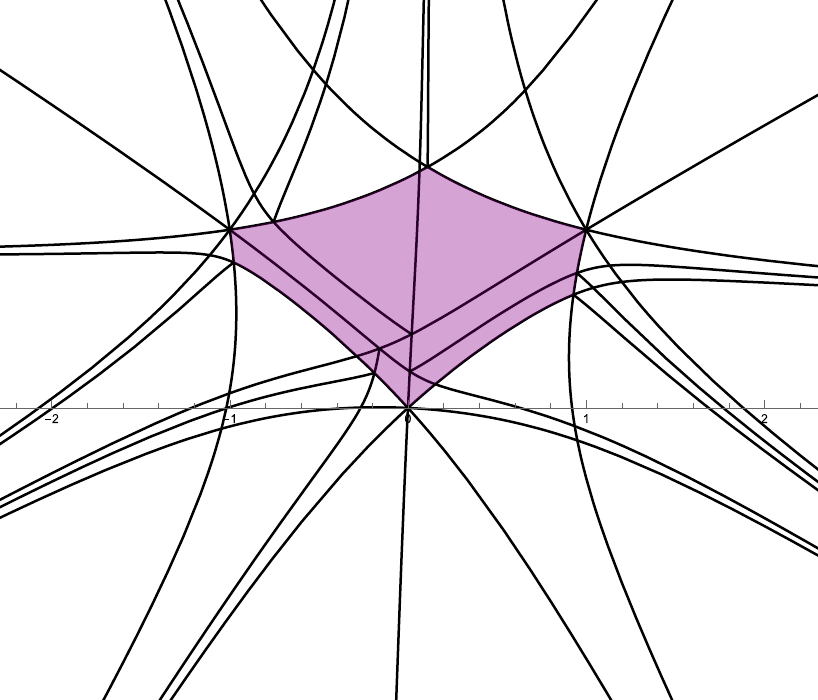}
        \caption{$\vartheta\approx0.98$, type $\mathrm{I}$}
      \label{fig:chamberB-4}
    \end{subfigure}

 \caption{Spectral networks and spectral cores for $t\approx 0.5+0.5i\in \mathcal{C}_B$}
  \label{fig:chamberB}
\end{figure}



For reasons of space we omit the illustration for chamber $\mathcal{C}_{A}$, since the spectral networks for $t$ in $\mathcal{C}_{A}$ and $\mathcal{C}_{B}$ are almost the same except that the order of the intervals with spectral cores of type $\mathrm{II}^{-}$ and $\mathrm{II}^{+}$ is reversed.

\subsubsection{Spectral networks on walls}

Walls $\Delta^{\pm}_{1}$, $\Delta_{2}$, $\Delta_{3}$ and $\Delta_{4}$ are characterized by the property that at least two saddle connections have the same slope in $\mathbb{R}/\frac{\pi}{3}\mathbb{Z}$. Similar to the previous section, their spectral cores and degenerations are described in Tables~\ref{table:chamberipi3},~\ref{table:chamberDelta+},~\ref{table:chamberDelta2} and~\ref{table:chamberDelta3} below.

\begin{table}\centering
  \begin{tabular}{|c||c|c|c|c|} \hline
    \textbf{Type for $t = e^{\frac{i\pi}{3}}$}
    & three saddles
    & $\mathrm{I}$ & tripod & $\mathrm{I}$   \\ \hline
  \end{tabular} 
     \caption{Spectral cores and degenerations for $t = e^{\frac{i\pi}{3}}$.}
     \label{table:chamberipi3}
  \end{table}


\begin{table}\centering
  \begin{tabular}{|c||c|c|c|c|c|c|} \hline
    \textbf{Type for $t \in \Delta^{+}_{1}$}
    & two saddles
    & $\mathrm{I}$ & tripod & $\mathrm{I}$  & saddle & $\mathrm{II}^{-}$ \\ \hline
    \textbf{Type for $t \in \Delta^{-}_{1}$}
    & saddle
    & $\mathrm{I}$ & tripod & $\mathrm{I}$  & two saddles & $\mathrm{II}^{+}$ \\ \hline
    
  \end{tabular} 
     \caption{Spectral cores and degenerations for $t \in \Delta^{+}_{1}$ and $t \in \Delta^{-}_{1}$.}
     \label{table:chamberDelta+}
  \end{table}



\begin{table}\centering
  \begin{tabular}{|c||c|c|c|c|} \hline
    \textbf{Type for $t \in \Delta_{2}$}
    & two saddles
    & $\mathrm{II}^{-}$ & saddle & $\mathrm{II}^{+}$  \\ \hline
  \end{tabular} 
     \caption{Spectral networks and degenerations for $t \in \Delta_{2}$.}
     \label{table:chamberDelta2}
  \end{table}


\begin{table}\centering
  \begin{tabular}{|c||c|c|} \hline
    \textbf{Type for $t \in \Delta_{3}$}
    & two saddles
    & $\mathrm{III}$  \\ \hline
    \textbf{Type for $t \in \Delta_{4}$}
    & two saddles
    & $\mathrm{IV}$  \\ \hline

  \end{tabular} 
     \caption{Spectral networks and degenerations for $t \in \Delta_{3}$ and $t \in \Delta_{4}$ (including $t=\frac{1}{2}$).}
     \label{table:chamberDelta3}
  \end{table}


  \begin{remark}
One can distinguish between loci $\Delta_{1}^{-}$ and $\Delta_1^{+}$ based on the final remark of Section~\ref{sub:d3Delta} in which it is explained that $|w_0(t)|>|w_1(t)|$ when $\mathrm{Re}(t)<\frac{1}{2}$ while $|w_0(t)|<|w_1(t)|$ when $\mathrm{Re}(t)>\frac{1}{2}$.
\end{remark}

 \section{The associated (variation of) BPS structure}
Our final result states that in the degree $d=3$ (and lower) case the family of BPS structures we have found form a variation of BPS structures -- in particular, that they vary as expected according to the Kontsevich-Soibelman wall-crossing formula. Below, we use the notation from Section \ref{sec:varbps}.

\subsection{Lattice, generators, and central charge}

Recall that the lattice for any cubic differential $\varphi_t$ is by definition $H_1(\Sigma^\times_t,\mathbb{Z})$, equipped with the intersection pairing $\langle\cdot,\cdot \rangle$. Assembling these into a local system as $t$ varies over the complex manifold $M=\mathrm{int}\,T$, we note that the resulting local system is trivial since $\mathrm{int}\,T$ is contractible. We thus identify all the lattices with a fixed $\Gamma$ given by any particular choice of $t$. It is not hard to check that $\Sigma^{\times}_t$ is in fact a twice-punctured torus, $\Gamma$ is a lattice of rank $4$, and $\langle \cdot, \cdot \rangle$ is degenerate with a kernel of rank $2$ \cite{Neitzke17}.

We will take as generators $\gamma_1,\gamma_2$ canonical lifts of the saddles homotopic to $[1,\infty]$ and $[0,1]$ respectively, together with the lifts $\gamma_3, \gamma_4$ of a small loop around $\infty$. The ambiguity in such a definition is resolved by specifying the central charges of the lifts: we have chosen this basis to coincide exactly with the one used in \cite{Neitzke17,dumasneitzke}. Then $\gamma_3,\gamma_4$ generate the kernel of $\langle \cdot , \cdot \rangle$, and we have:
\begin{lemma} \label{lem:centralcharge}
For $0<\mathrm{Re}(t)<1$ with $\mathrm{Im}(t)\neq 0$, the central charges of $\gamma \in \Gamma$ are given by the following values on generators:

\begin{align}&Z(\gamma_1)=\frac{i \sqrt[3]{\alpha } \,5 \Gamma \left(-\frac{5}{3}\right) \Gamma \left(-\frac{5}{6}\right) \, _2F_1\left(\frac{4}{3},3;\frac{8}{3};\frac{1}{1-t}\right)}{18\ 2^{\frac{2}{3}} \sqrt{3 \pi } (t-1)^3}-\frac{2 \pi i  \sqrt[3]{\alpha } ((t-1) t+1)}{9 (1-t)^{\frac{5}{3}} (-t)^{\frac{5}{3}}}\\
    &Z(\gamma_2)=\frac{\left(1+\sqrt[3]{-1}\right) \sqrt{\pi } \sqrt[3]{-\alpha }\, \Gamma \left(\frac{4}{3}\right) \, _2F_1\left(\frac{4}{3},3;\frac{8}{3};\frac{1}{t}\right)}{2\ 2^{\frac{2}{3}} \Gamma \left(\frac{11}{6}\right)t^3 }\\
    &Z(\gamma_3)=-\frac{2 \pi i  \sqrt[3]{\alpha } \sqrt[3]{(t-1) t} \left(t^2-t+1\right)}{9 (t-1)^2 t^2}, \quad
    Z(\gamma_4)= e^{\frac{2 \pi i }{3}}Z(\gamma_3)
\end{align}
\end{lemma}
\begin{proof}
    Direct computation.
\end{proof}

\subsection{BPS invariants}
Recall the wall-and-chamber structure of $T$, determined by the loci $\Delta_k$, $k=1,2,3,4$ at which two saddle connections make an angle of $k\frac{\pi}{3}$ with each other. From the determination of saddles and tripods in \S\ref{sub:d3WallsSecond}, we may take canonical lifts to obtain the active classes in each chamber. This result was stated in \cite{Neitzke17} for a fixed chamber; our results immediately establish that calculation as a theorem and extend it to the other chambers.

\begin{corollary}[cf \cite{Neitzke17, dumasneitzke}]\label{cor:classesintro} The active classes are as follows:
\begin{itemize}\item In all chambers and walls, the classes \begin{align*}
&\pm(1,0,0,0), \quad \pm(-1,1,1,1), \quad \pm(0,-1,-1,-1) \\ &\pm(0,-1,-1,0), \quad \pm(1,0-1,-1), \quad \pm(-1,1,2,1)\end{align*} 
are active and are saddle classes.
\item The classes 
\begin{equation*}\pm(0,1,0,0), \quad \pm(-1,0,1,0), \quad \pm(1,-1,-1,0)\quad \end{equation*} 
are active and are saddle classes if and only if $t$ lies above the wall $\Delta_3$.
\item The classes 
\begin{equation*}\pm(1,1,0,0), \quad \pm(-2,1,2,1),\quad \pm(1,-2,-2,-1) 
\end{equation*}
are active and are tripod classes if and only if $t$ lies above the wall $\Delta_2$.
\end{itemize}
\end{corollary}

The BPS invariants (depending on $t$) are then obtained by setting $\Omega_t(\gamma)=1$ for an active class $\gamma$, and $0$ on all others. Together with the previous subsection we obtain a family of BPS structures $(\Gamma_t,Z_t,\Omega_t)$ depending on $t\in\mathrm{int}\,T$.

\subsection{Wall-crossing formula}

Our final result is to show that the family of BPS structures constructed thus far is in fact a variation of BPS structures. In particular, we show that the wall-crossing property (Definition \ref{defn:varbps}, \ref{item:wallcrossing})) is satisfied, so their jumping behaviour is as expected in both the quantum field theory and Donaldson-Thomas contexts.

\begin{theorem}
    The family $(\Gamma_t,Z_t,\Omega_t)$ forms a variation of BPS structures over $M=\mathrm{int}\,T$.
\end{theorem}
\begin{proof}
    The first property is automatic since the local system is trivial. The second property holds similarly since the central charges are clearly holomorphic on $\mathrm{int}\,T$ by Lemma \ref{lem:centralcharge}. Since the BPS structure is finite, the constant $C$ in the third property is certainly uniform. It remains to show the wall-crossing formula. To prove the wall-crossing formula, we need to check that the composition of BPS automorphisms ${S}_{\!\leftslice}$ \eqref{eq:BPS-auto} agrees on any small sector $\leftslice$ in the central charge plane for any value of $t$, in particular when crossing a wall of the first kind. 

    We have four generators for the twisted torus, $x_{(1,0,0,0)}, x_{(0,1,0,0)}, x_{(0,0,1,0)}, x_{(0,0,0,1)}$ which we denote by $x_1, x_2, x_3, x_4$ for aesthetic clarity.    

There are three walls of the first kind in $\mathrm{int}\,T$, corresponding to $\Delta_i|_{i=1,2,3}$ giving four chambers. The latter two chambers $\mathcal{C}_B, \mathcal{C}_A$ are essentially the same: $\Omega$ is identical on either side and only the value of the central charges changes (as some BPS rays whose classes have zero pairing move past each other). Thus wall-crossing is satisfied trivially over $\Delta_1$. In crossing $\Delta_3$ from $\mathcal{C}_D$ to $\mathcal{C}_C$, the calculation is well-known as the and is exactly as in the case of the ``$A_2$ quiver'' described in e.g. \cite{GAIOTTO2013239,bridgeland2019riemann} as the \emph{pentagon identity}. The corresponding ray diagram depicting the BPS rays in the central charge plane is (schematically) given in Figure~\ref{fig:wcfpic}, where we write 
\begin{equation*}\gamma_A^{(1)}=(1,-1,-2,-1), \quad \gamma_C^{(1)}=(1,0,-1,0), \quad \gamma_B^{(1)}=(0,1,1,1) \end{equation*}
with other superscripts denoting the classes obtained by acting with the cyclic symmetry of $\Sigma$.

    \begin{figure}[h]
        \centering

               \begin{subfigure}{0.4\textwidth}
                \includegraphics[width=\linewidth]{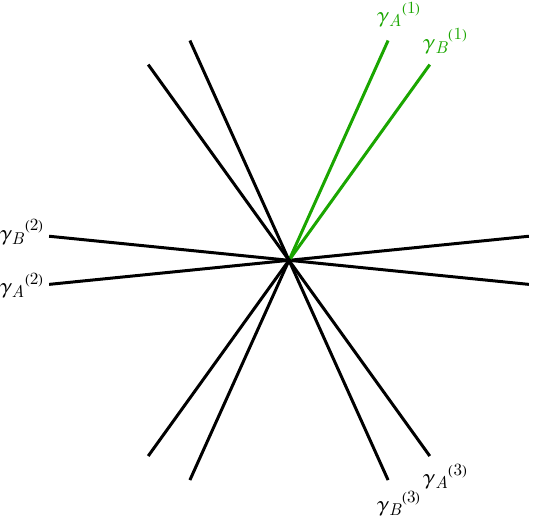}
                \caption{The ray diagram in $\mathcal{C}_D$ near $\Delta_3$}
                \label{fig:wcfpica}
        \end{subfigure}
         \hspace{0.5cm}
        \begin{subfigure}{0.4\textwidth}
        \includegraphics[width=\linewidth]{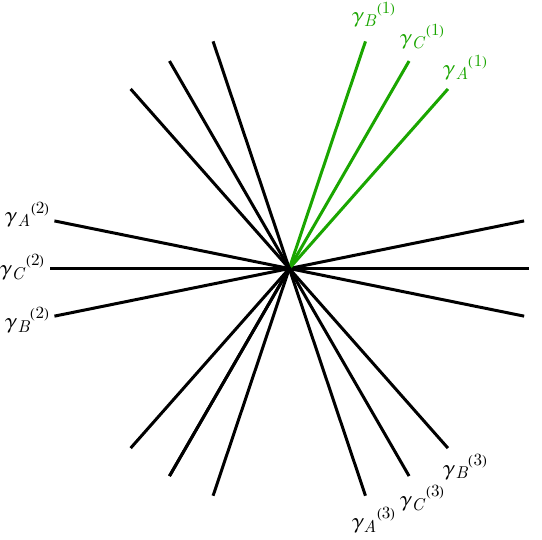}            \caption{The ray diagram in $\mathcal{C}_C$ near $\Delta_3$}
        \label{fig:wcfpicb}
        \end{subfigure}
 
        \caption{Crossing the wall $\Delta_3$}
        \label{fig:wcfpic}
    \end{figure}

Thus, we give the calculation only for the case in which we begin in the region $\mathcal{C}_C$ and move to $\mathcal{C}_B$, where the central charges of some (saddle) classes $\gamma_l$ and $\gamma_r$ align, resulting in a new active (tripod) class $\gamma_m$ with $\arg Z(\gamma_l)<\arg Z({\gamma_m})<\arg Z(\gamma_r)$ in $\mathcal{C}_B$. For example, let us consider the classes 

\begin{equation*}{\gamma_l=(1,0,0,0)=-\gamma_{B}^{(2)},\quad \gamma_m=(2,-1,-2,-1),\quad  \gamma_r=(1,-1,-2,-1)}=\gamma_A^{(1)}\end{equation*}
(checking their central charges slightly above and below the wall confirms that this is a correct triple and ordering of classes to consider) as in Figure \ref{fig:wcfpic2}.

        \begin{figure}[h]
        \centering

               \begin{subfigure}{0.45\textwidth}
                \includegraphics[width=\linewidth]{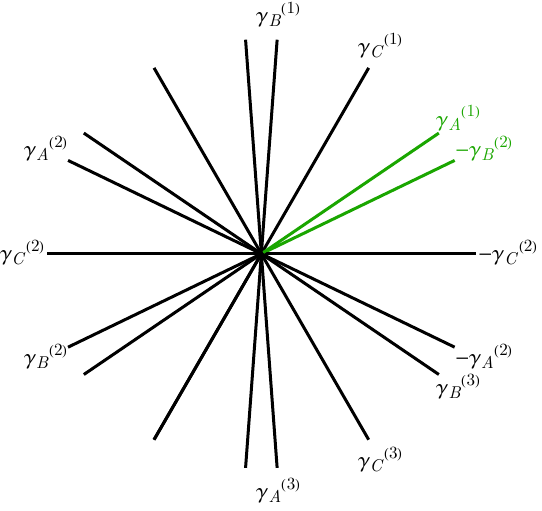}
                \caption{The ray diagram in $\mathcal{C}_C$ near $\Delta_2$}
                \label{fig:wcfpic2a}
        \end{subfigure}
         \hspace{0.5cm}
        \begin{subfigure}{0.45\textwidth}
        \includegraphics[width=\linewidth]{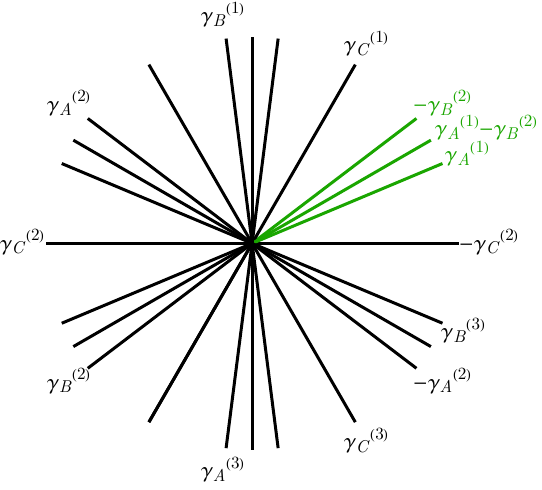}            \caption{The ray diagram in $\mathcal{C}_B$ near $\Delta_2$}
        \label{fig:wcfpic2b}
        \end{subfigure}
 
        \caption{Crossing the wall $\Delta_2$}
        \label{fig:wcfpic2}
    \end{figure}

\noindent For $t$ in the region $\mathcal{C}_C$ or on the wall $\Delta_2$, we have over the sector ${\leftslice}$ containing $\gamma_{l}, \gamma_{m}, \gamma_{r}$,

   \begin{equation}
       S^-_{\!\leftslice}=S_{\ell_l}S_{\ell_m}S_{\ell_r}
   \end{equation}
   where each $\ell_\cdot$ corresponds to a single active class $\gamma_{\cdot}$. We can compute the action on, say, $x_{1}=x_{\gamma_1}$, noting that $S_{\ell_m}^*$ acts trivially since $\Omega(\gamma_m)=0$ in $\mathcal{C}_C$, and $S_{\ell_l}^*$ acts trivially since $\gamma_l=\gamma_1$:


         \begin{align}
       S_{\ell_r}^*S_{\ell_l}^*(x_1)&= S_{\ell_r}^*\left(x_1\right)\\ &=x_{1}(1-x_{\gamma_r})^{\Omega(\gamma_r)\langle \gamma_r, \gamma_1 \rangle}\\
       &=x_1\left(1+\frac{x_1}{x_2 x_3^2 x_4}\right)
   \end{align}

\noindent where we used the twisted torus relation \eqref{eq:ttorus} to express everything in terms of $x_{1,2,3,4}$.

   On the other hand, as we cross the wall into region $\mathcal{C}_B$ the class $\gamma_m=(2,-1,-2,-1)$ becomes active and the ordering of BPS rays reverses and we obtain (again using the twisted torus relation \eqref{eq:ttorus}):

   \begin{align}
       S_{\ell_l}^*S_{\ell_m}^*S_{\ell_r}^*(x_1)&=S^*_{\ell_l}S^*_{\ell_m}\left(x_1\left(1+\frac{x_1}{x_2 x_3^2 x_4}\right)\right)\\
       &=S_{\ell_l}^*\left(x_1\left(1-\frac{x_1\left(x_1-1\right) }{x_2 x_3^2 x_4}\right)\right)\\
       &=x_1\left(1+\frac{x_1}{x_2 x_3^2 x_4}\right)
   \end{align}
exactly as above.

We can similarly determine the action of $S^{\pm}_{\!\leftslice}$ on $x_2$ and check that they give the same result:

\begin{align}
    S_{\ell_r}^*S_{\ell_l}^*(x_2)&= S_{\ell_r}^*\big(x_2\left(1-x_1\right)\big)\\
        &= x_2\left(1+\frac{x_1 \left(x_2 x_3^2 x_4-\left(x_2 x_3^2 x_4+x_1\right)^2\right)}{x_2^2 x_3^4 x_4^2}\right)
\end{align}

\noindent and likewise 

\begin{align}
    S_{\ell_l}^*S_{\ell_m}^*S_{\ell_r}^*(x_2)&= S_{\ell_l}^*S_{\ell_m}^*\left(1+\frac{x_1}{x_2 x_3^2 x_4}\right)\\
        &= S_{\ell_l}^*\left(x_2 \left(1+\frac{x_1^4-2 x_1^2 x_2 x_3^2 x_4}{x_2^2 x_3^4 x_4^2}\right)\right)\\
        &= x_2\left(1+\frac{x_1 \left(x_2 x_3^2 x_4-\left(x_2 x_3^2 x_4+x_1\right)^2\right)}{x_2^2 x_3^4 x_4^2}\right)
\end{align}
\noindent exactly as desired.

The action on $x_3$ and $x_4$ is trivial, since they lie in the kernel of the pairing $\langle \cdot, \cdot \rangle$. 

Similar calculations hold for the other classes (which can be obtained by applying the cyclic symmetry of $\Sigma$). Thus we obtain that the wall-crossing formula is satisfied when passing from chamber $\mathcal{C}_C$ to $\mathcal{C}_B$ for a small sector $\leftslice$ containing the central charge of $\gamma_l, \gamma_m,\gamma_r$. 
It follows that the wall-crossing formula is satisfied for arbitrary acute $\leftslice$, so that $(\Gamma_t,Z_t,\Omega_t)$ is a variation of BPS structures over $\mathcal{C}_C\cup \Delta_2 \cup \mathcal{C}_B$.

Thus, the claim holds for all regions in $M=\mathrm{int}\,T$.  
\end{proof}

\appendix



\bibliography{refs}

\begin{thebibliography}{10}
\bibitem{koikeschafke}
T.~Koike and R.~Sch\"afke, ``{On the Borel summability of WKB solutions of
  Sch\"odinger equations with polynomial potentials and its application}.'' in
  preparation.
\end{thebibliography}
\bibliographystyle{utphys.bst}

\end{document}